%% file: Main.tex
\RequirePackage[l2tabu,orthodox]{nag}		% for deprecated commands
\documentclass[reqno]{amsart}
\usepackage[margin=1.6in,bottom=1.25in]{geometry}		% for tighter margins (80 CPL)

\usepackage{amsmath}		% for AMS macros
\usepackage{amssymb}		% for AMS symbols
\usepackage{amsfonts}		% for AMS fonts
\usepackage{amsthm}		% for theorems
\usepackage[foot]{amsaddr}		% for author footnotes

\usepackage{mathtools}		% for advanced math

\mathtoolsset{%
}

\usepackage[utf8]{inputenc}		% for source encoding
\usepackage[T1]{fontenc}		% for font encoding

\usepackage[%		% for math font selection
cal=cm,
]
{mathalfa}

\usepackage{dsfont}		% for blackboard bold font
\usepackage[proportional,tabular,lining,sf,mono=false]{libertine}
\usepackage{acronym}		% for acronyms
\renewcommand*{\aclabelfont}[1]{\acsfont{#1}}		% for acronym label font
\newcommand{\acli}[1]{\emph{\acl{#1}}}		% for italicized acro
\newcommand{\aclip}[1]{\emph{\aclp{#1}}}		% for italicized acro (plural)
\newcommand{\acdef}[1]{\emph{\acl{#1}} \textup{(\acs{#1})}\acused{#1}}		% for acro def
\usepackage[labelfont={bf,small},labelsep=colon,font=small]{caption}	% for caption control
\usepackage{multirow}
\usepackage[dvipsnames,svgnames]{xcolor}		% for color
\colorlet{MyRed}{Crimson!75!Black}
\colorlet{MyGreen}{DarkGreen!80!Black}
\colorlet{MyBlue}{MediumBlue}

\usepackage{setspace}
\usepackage[]{titlesec}		% for section headings (may create problem with appendices)
\titleformat{name=\section}{}{\thetitle.}{0.8em}{\centering\scshape}
\titleformat{name=\subsection}[runin]{}{\thetitle.}{0.5em}{\bfseries}[.]
\titleformat{name=\subsubsection}[runin]{}{\thetitle.}{0.5em}{\bfseries}[.]

\titleformat{name=\paragraph,numberless}[runin]{}{}{0em}{\bfseries}[]
\titlespacing{\paragraph}{0em}{\medskipamount}{1em}
\titleformat{name=\subparagraph,numberless}[runin]{}{}{0em}{}[.]
\titlespacing{\subparagraph}{0em}{0em}{0.5em}

\newcommand{\nexthead}{.\;}		% for changing headings
\newcommand{\para}[1]{\paragraph{\textbf{#1\nexthead}}}

\usepackage{latexsym}		% for symbols
\usepackage{subcaption}		% for subfigures
\usepackage{tikz}		% for figures
\usetikzlibrary{calc,patterns}		% for basic tikz figures
\usepackage{pgfplots}

\usepackage{array}		% for flexible tables
\usepackage{booktabs}		% for better tables
\usepackage[inline,shortlabels]{enumitem}		% for lists
\setenumerate{itemsep=\smallskipamount,topsep=\smallskipamount}		% for list controls

\usepackage[kerning=true]{microtype}		% for microtypography

\usepackage{xspace}		% for flexible spaces

\usepackage[compress]{natbib}		% for citations

\bibpunct[, ]{(}{)}{,}{}{,}{,}

\usepackage{hyperref}
\hypersetup{
colorlinks=true,
linktocpage=true,
pdfstartview=FitH,
breaklinks=true,
pdfpagemode=UseNone,
pageanchor=true,
pdfpagemode=UseOutlines,
plainpages=false,
bookmarksnumbered,
bookmarksopen=false,
bookmarksopenlevel=1,
hypertexnames=true,
pdfhighlight=/O,
urlcolor=MyBlue,linkcolor=MyBlue,citecolor=MyBlue,	% for on-screen
pdftitle={},
pdfauthor={},
pdfsubject={},
pdfkeywords={},
pdfcreator={pdfLaTeX},
pdfproducer={LaTeX with hyperref}
}

\usepackage[sort&compress,capitalize,nameinlink]{cleveref}		% for cleveref formatting
\crefname{assumption}{Assumption}{Assumptions}
\usepackage{algorithm}		% for algorithm environments
\usepackage{algpseudocode}		% for algorithm macros
\usepackage{thmtools}		% for theorem tools
\usepackage{thm-restate}		% for restating theorems

\theoremstyle{plain}
\newtheorem{theorem}{Theorem}		% for theorems
\newtheorem{lemma}{Lemma}		% for lemmas
\newtheorem{proposition}{Proposition}		% for propositions

\newtheorem*{corollary*}{Corollary}		% for corollaries (unnumbered)

\theoremstyle{definition}
\newtheorem{definition}{Definition}		% for definitions
\newtheorem{assumption}{Assumption}		% for assumptions
\newtheorem{example}{Example}		% for examples

\newtheorem*{definition*}{Definition}		% for definitions (unnumbered)
\newtheorem*{assumption*}{Assumptions}		% for assumptions (unnumbered)
\newtheorem*{example*}{Example}		% for examples (unnumbered)

\theoremstyle{remark}
\newtheorem*{remark*}{Remark}		% for remarks (unnumbered)

\def\endenv{\hfill{\small$\lozenge$\smallskip}}

\newcounter{proofpart}

\numberwithin{example}{section}		% for example numbering

\usepackage[showdeletions]{color-edits}		% for editing macros / use [suppress] for final
\usepackage[normalem]{ulem}		% for strikeout text
\newcommand{\debug}[1]{#1}		% for removing macro coloring

\newcommand{\newmacro}[2]{\newcommand{#1}{\debug{#2}}}		% for shorthand definitions
\newcommand{\newop}[2]{\DeclareMathOperator{#1}{\debug{#2}}}		% for shorthand definitions

\DeclarePairedDelimiter{\bracks}{[}{]}		% for brackets
\DeclarePairedDelimiter{\parens}{(}{)}		% for parentheses

\DeclarePairedDelimiter{\clip}{[}{]}		% for clipping
\DeclarePairedDelimiterX{\setdef}[2]{\{}{\}}{#1:#2}		% for set builder notation
\DeclarePairedDelimiterXPP{\exclude}[1]{\mathopen{}\setminus}{\{}{\}}{}{#1}

\newcommand{\N}{\mathbb{N}}		% for naturals
\newcommand{\R}{\mathbb{R}}		% for reals
\DeclareMathOperator*{\argmin}{arg\,min}		% for argmin
\DeclareMathOperator*{\intersect}{\bigcap}		% for intersections
\DeclarePairedDelimiterXPP{\bigoh}[1]{\mathcal{O}}{(}{)}{}{#1}		% for Landau O
\DeclareMathOperator{\dist}{dist}		% for distance
\DeclareMathOperator{\dom}{dom}		% for domain
\DeclareMathOperator{\one}{\mathds{1}}		% for indicator
\DeclareMathOperator{\relint}{ri}		% for relative interior
\newcommand{\cf}{cf.\xspace}		% for consistency
\newcommand{\eg}{e.g.,\xspace}		% for consistency
\newcommand{\ie}{i.e.,\xspace}		% for consistency
\newcommand{\dis}{\displaystyle}		% for forcing display style
\newcommand{\txs}{\textstyle}		% for forcing inline style

\newcommand{\alt}[1]{#1'}		% for variant version
\newcommand{\altalt}[1]{#1''}		% for second variant
\newmacro{\dd}{\:d}		% for integration
\newcommand{\eps}{\varepsilon}		% for better epsilon
\newmacro{\const}{c}		% for generic constant
\newmacro{\coef}{\lambda}		% for generic coefficient
\newmacro{\param}{\theta}		% for parameter
\newmacro{\params}{\Theta}		% for set of parameters

\newmacro{\pexp}{p}		% for first exponent
\newmacro{\qexp}{q}		% for second exponent
\newmacro{\rexp}{r}		% for third exponent

\newmacro{\beforestart}{0}		% for prev start index
\newmacro{\start}{1}		% for start index
\newmacro{\afterstart}{2}		% for second index

\newmacro{\halfrunning}{\textbf{\debug{FIXME}}}		% for running
\newmacro{\running}{\start,\afterstart,\dotsc}		% for running

\newmacro{\run}{t}		% for main sequence index
\newmacro{\runalt}{s}		% for variant index
\newmacro{\runaltalt}{\tau}		% for second variant
\newmacro{\nRuns}{T}		% for total number of runs
\newmacro{\runs}{\mathcal{\nRuns}}		% for set of indices

\newmacro{\half}{\frac{1}{2}}
\newcommand{\inv}[1]{\frac{1}{#1}}

\newmacro{\state}{X}		% for main iterate
\newmacro{\statealt}{Y}		% for variant state
\newmacro{\statealtalt}{Z}		% for second variant

\newcommand{\new}[1][\point]{#1^{+}}		% for new iterate (x by default)

\newcommand{\beforeinit}[1][\state]{\debug{#1}_{\beforestart}}		% for zeroth iterate (X by default)
\newcommand{\init}[1][\state]{\debug{#1}_{\start}}		% for initial iterate (X by default)
\newcommand{\initlead}[1][\state]{\debug{#1}_{1/2}}		% for initial iterate (X by default)
\newcommand{\afterinitlead}[1][\state]{\debug{#1}_{3/2}}		% for initial iterate (X by default)

\newcommand{\prev}[1][\state]{\debug{#1}_{\run-1}}		% for previous iterate (X by default)
\newcommand{\curr}[1][\state]{\debug{#1}_{\run}}		% for current iterate (X by default)
\renewcommand{\next}[1][\state]{\debug{#1}_{\run+1}}		% for next iterate (X by default)

\newcommand{\beforelead}[1][\state]{\debug{#1}_{\run-1/2}}		% for prev lead iterate (X by default)
\newcommand{\lead}[1][\state]{\debug{#1}_{\run+1/2}}		% for lead iterate (X by default)
\newcommand{\beforebeforelead}[1][\state]{\debug{#1}_{\run-3/2}}		

\newmacro{\tstart}{0}		% for time start
\newmacro{\timealt}{s}		% for dummy continuous time
\newmacro{\horizon}{T}		% for horizon

\newmacro{\traj}{x}		% for trajectory
\newmacro{\trajalt}{y}		% for variant trajectory
\newmacro{\trajaltalt}{z}		% for second variant

\newmacro{\flowmap}{\Theta}		% for (semi)flow map
\DeclarePairedDelimiterXPP{\flowof}[2]{\flowmap_{#1}}{(}{)}{}{#2}		% for flow

\newop{\Nash}{NE}		% for Nash equilibria
\newop{\CE}{CE}		% for correlated equilibria
\newop{\CCE}{CCE}		% for Hannan set
\newop{\NI}{NI}		% for Nikaido-Isoda function

\newop{\brep}{br}		% for best responses
\newop{\reg}{Reg}		% for regret
\newop{\preg}{\overline{Reg}}		% for pseudo-regret
\newop{\val}{val}		% for value function

\newmacro{\play}{i}		% for player index
\newmacro{\playalt}{j}		% for variant player index
\newmacro{\playaltlalt}{k}		% for second variant
\newmacro{\nPlayers}{N}		% for number of players
\newmacro{\players}{\mathcal{\nPlayers}}		% for set of players

\newmacro{\pure}{\alpha}		% for pure strategy
\newmacro{\purealt}{\beta}		% for variant pure strategy
\newmacro{\purealtalt}{\gamma}		% for second variant
\newmacro{\nPures}{A}		% for number of pure strategies
\newmacro{\pures}{\mathcal{\nPures}}		% for set of pure strategies

\newmacro{\loss}{\ell}		% for loss function
\newmacro{\pay}{u}		% for payoff function
\newmacro{\payv}{v}		% for payoff vector
\newmacro{\pot}{f}		% for potential function

\newmacro{\game}{\mathcal{G}}		% for game
\newmacro{\gamefull}{\game(\players,\points,\pay)}		% for full game

\newmacro{\fingame}{\Gamma}		% for finite game
\newmacro{\fingamefull}{\Gamma(\players,\pures,\pay)}		% for full finite game

\newmacro{\gmat}{g}		% for metric tensor
\newmacro{\gdist}{\dist_{\gmat}}
\newmacro{\mfld}{M}		% for manifold
\newmacro{\form}{\omega}		% for generic form

\newmacro{\graph}{\mathcal{G}}
\newmacro{\vertices}{\mathcal{V}}
\newmacro{\edges}{\mathcal{E}}

\newmacro{\mat}{M}		% for generic matrix
\newmacro{\ones}{\mathbf{1}}		% for matrix of ones
\newmacro{\eye}{I}		% for identity matrix
\newmacro{\zer}{\mathbf{0}}		% for zero matrix

\DeclarePairedDelimiter{\norm}{\lVert}{\rVert}		% for norm
\DeclarePairedDelimiterXPP{\dnorm}[1]{}{\lVert}{\rVert}{_{\ast}}{#1}		% for dual norm

\DeclarePairedDelimiterXPP{\onenorm}[1]{}{\lVert}{\rVert}{_{1}}{#1}		% for dual norm
\DeclarePairedDelimiterXPP{\twonorm}[1]{}{\lVert}{\rVert}{_{2}}{#1}		% for dual norm
\DeclarePairedDelimiterXPP{\supnorm}[1]{}{\lVert}{\rVert}{_{\infty}}{#1}		% for dual norm

\DeclarePairedDelimiterX{\braket}[2]{\langle}{\rangle}{#1,#2}		% for brakets
\DeclarePairedDelimiterX{\internalInner}[1]{\langle}{\rangle}{#1}
\usepackage{xparse}
\NewDocumentCommand \inner {m g}{
    \internalInner{#1\IfValueT{#2}{,#2}}
}

\newmacro{\vecspace}{\mathcal{V}}		% for generic vector space
\newmacro{\subspace}{\mathcal{W}}		% for vector subspace

\newmacro{\coord}{i}		% for coordinate index
\newmacro{\coordalt}{j}		% for variant coordinate
\newmacro{\coordaltalt}{k}		% for second variant
\newmacro{\nCoords}{d}		% for number of coordinates
\newmacro{\dims}{\nCoords}		% for dimension
\newmacro{\vdim}{\nCoords}		% for dimension (legacy alias)

\newmacro{\pvec}{v}		% for primal vector
\newmacro{\tvec}{z}		% for tangent vector
\newmacro{\uvec}{u}		% for unit vector

\newmacro{\bvec}{e}		% for basis vector
\newmacro{\bvecs}{\mathcal{E}}		% for basis vectors

\newmacro{\ball}{\mathbb{B}}		% for ball
\newmacro{\sphere}{\mathbb{S}}		% for sphere

\newmacro{\pspace}{\mathcal{X}}		% for primal space
\newmacro{\dspace}{\mathcal{Y}}		% for dual space

\newmacro{\dvec}{w}		% for dual vector
\newmacro{\dbvec}{\eps}		% for dual basis vectors

\newmacro{\dpoint}{y}		% for generic dual point
\newmacro{\dpointalt}{\alt\dpoint}		% for variant dual point
\newmacro{\dpointaltalt}{\altalt\dpoint}		% for second variant
\newmacro{\dpoints}{\mathcal{Y}}		% for set of dual points

\newmacro{\dstate}{Y}		% for dual state

\newcommand{\defeq}{\coloneqq}		% for direct definition
\newcommand{\from}{\colon}		% for function definition
\newop{\Opt}{Opt}		% for value of problem
\newop{\Sol}{Sol}		% for solution of problem
\newop{\gap}{Gap}		% for gap function

\newop{\orcl}{\mathsf{V}}		% for oracle
\newop{\err}{\mathsf{Err}}		% for error

\newmacro{\tfun}{g}		% for test function
\newmacro{\obj}{f}		% for objective function
\newmacro{\objalt}{g}		% for variant objective (smooth etc.)
\newmacro{\sobj}{F}		% for stochastic objective

\newmacro{\gvec}{g}		% for gradient vector
\newmacro{\oper}{A}		% for operator
\newmacro{\vecfield}{v}		% for vector field (selection etc.)
\newmacro{\hmat}{H}		% for Hessian matrix

\newcommand{\sol}[1][\point]{#1^{\ast}}		% for solution point (x by default)
\newcommand{\test}[1][\point]{\hat#1}		% for test point (x by default)

\newmacro{\signal}{V}		% for signal
\newmacro{\step}{\gamma}		% for step-size
\newmacro{\learn}{\eta}		% for learning rate

\newmacro{\vbound}{G}		% for vector bound
\newmacro{\lips}{L}		% for Lipschitz modulus
\newmacro{\lipsalt}{\overline{L}}
\newmacro{\strong}{\mu}		% for strong convexity modulus
\newmacro{\smooth}{\beta}		% for strong smoothness modulus

\newop{\cone}{cone}
\newop{\tspace}{T}		% for tangent space
\newop{\tcone}{TC}		% for tangent cone
\newop{\dcone}{\tcone^{\ast}}		% for dual cone
\newop{\ncone}{NC}		% for normal cone
\newop{\pcone}{PC}		% for polar cone
\newop{\hull}{\Delta}		% for simplices

\newmacro{\cvx}{\mathcal{C}}		% for generic convex set
\newmacro{\subd}{\partial}		% for subdifferential

\newmacro{\minmax}{\Phi}		% for minmax objective

\newmacro{\minvar}{\point_{1}}		% for minimization variable
\newmacro{\minvaralt}{\alt\minvar}		% for variant minvar
\newmacro{\minvars}{\points_{1}}		% for set of minvars
\newmacro{\minsol}{\sol[\minvar]}		% for min solution

\newmacro{\maxvar}{\point_{2}}		% for maximization variable
\newmacro{\maxvaaltr}{\alt\maxvar}		% for variant maxvar
\newmacro{\maxvars}{\points_{2}}		% for set of maxvars
\newmacro{\maxsol}{\sol[\maxvar]}		% for max solution

\newop{\Eucl}{\Pi}		% for Euclidean projection
\newop{\logit}{\Lambda}		% for logit map
\newop{\dkl}{KL}		% for Kullback Leibler

\newmacro{\hreg}{h}		% for regularizer
\newmacro{\breg}{D}		% for Bregman divergence
\newmacro{\mprox}{P}		% for Bregman prox-mapping
\newmacro{\mirror}{Q}		% for mirror map
\newmacro{\fench}{F}		% for Fenchel coupling
\newmacro{\hstr}{K}		% for strong convexity constant
\newmacro{\depth}{H}		% for regularizer depth
\newmacro{\proxdom}{\points_{\hreg}}		% for prox-domain
\newmacro{\zone}{\mathbb{D}}		% for Bregman zone
\newmacro{\bregkernel}{\theta} % For the kernel of the Bregman divergence

\DeclarePairedDelimiterXPP{\proxof}[2]{\mprox_{#1}}{(}{)}{}{#2}		% for Bregman prox step
\newmacro{\bregexp}{\alpha}
\newmacro{\bregcst}{M}

\newmacro{\kernelcst}{C}
\newmacro{\kernelexp}{q}

\newmacro{\point}{x}		% for generic point
\newmacro{\pointalt}{\alt\point}		% for variant point
\newmacro{\pointaltalt}{\altalt\point}		% for second variant
\newmacro{\points}{\mathcal{K}}		% for set of points
\newmacro{\intpoints}{\relint\points}		%for point set interior

\newmacro{\base}{p}		% for reference point
\newmacro{\basealt}{q}		% for variant reference point
\newmacro{\basealtalt}{u}		% for second variant

\newmacro{\open}{\mathcal{U}}		% for open sets
\newmacro{\closed}{\mathcal{C}}		% for closed sets
\newmacro{\cpt}{\mathcal{C}}		% for compact sets
\newmacro{\nhd}{\mathcal{U}}		% for neighborhoods

\newop{\ex}{\mathbb{E}}		% for expectations
\newop{\prob}{\mathbb{P}}		% for probability
\newop{\Var}{Var}		% for variance
\newop{\simplex}{\hull}		% for simplices

\providecommand\given{}		% empty command for conditionals

\DeclarePairedDelimiterXPP{\exof}[1]{\ex}{[}{]}{}{%		% for conditional expectations
\renewcommand\given{\nonscript\,\delimsize\vert\nonscript\,\mathopen{}} #1}

\DeclarePairedDelimiterXPP{\probof}[1]{\prob}{(}{)}{}{%		% for conditional probabilities
\renewcommand\given{\nonscript\:\delimsize\vert\nonscript\:\mathopen{}} #1}

\DeclarePairedDelimiterXPP{\oneof}[1]{\one}{\{}{\}}{}{%		% for conditional expectations
\renewcommand\given{\nonscript\,\delimsize\vert\nonscript\,\mathopen{}} #1}

\newmacro{\sample}{\omega}		% for sample
\newmacro{\samples}{\Omega}		% for set of samples

\newmacro{\filter}{\mathcal{F}}		% for filtration
\newmacro{\probspace}{(\samples,\filter,\prob)}		% for probability space

\newcommand{\Condexp}[2][t]{\ex\left[\left.#2\,\right|\,\filter_{#1}\right]}

\newmacro{\event}{\mathcal{E}}       % for event
\newmacro{\eventalt}{\mathcal{H}}       % for variant event
\newmacro{\mean}{\mu}		% for mean of distribution
\newmacro{\sdev}{\sigma}		% for mean of distribution
\newmacro{\variance}{\sdev^{2}}		% for mean of distribution

\newmacro{\proper}{\tau}		% for proper time
\newmacro{\error}{Z}		% for error
\newmacro{\noise}{U}		% for noise
\newmacro{\bias}{b}		% for bias
\newmacro{\brown}{W}		% for Wiener process

\newmacro{\serror}{\theta}		% for scalar error
\newmacro{\snoise}{\xi}		% for scalar noise
\newmacro{\sbias}{\psi}		% for scalar bias

\newmacro{\sbound}{M}		% for signal bound
\newmacro{\bbound}{B}		% for bias bound
\newmacro{\noisepar}{\sdev}		% for noise parameter
\newmacro{\noisevar}{\variance}		% for noise variance

\newmacro{\leg}{\alpha}		% Legendre exponent
\newmacro{\legexp}{\beta}		% for ``inverse'' Legendre exponent
\newmacro{\legconst}{\kappa}		% for Legendre constant

\newmacro{\expstep}{\eta}
\newmacro{\cst}{q} % Constants for sequences
\newmacro{\cstalt}{q'} 
\newmacro{\seq}{a} % To be used in conjunction with prev, lead, current ie \curr[\seq]
\newmacro{\seqalt}{b}

\newif\ifcomment
\commenttrue

\newmacro{\pnhd}{\mathcal{U}}		% the nhd used in the proof of Prop. 1
\newmacro{\legnhd}{\mathcal{V}}		% the nhd of the Legendre exp

\newmacro{\fixmap}{F}		% for fixed-point iteration

\newmacro{\thres}{\delta}		% for confidence level
\newmacro{\basin}{\mathcal{B}}		% for basin of attraction
\newmacro{\inhd}{\init[\nhd]}

\newmacro{\seed}{\theta}		% for seed
\newmacro{\seeds}{\Theta}		% for seed space
\newmacro{\pdist}{P}		% for seed law
\newmacro{\history}{\mathcal{H}}		% for filtrations

\begin{document}

%**********************************************************************
%***    FRONT MATTER AND METADATA
%**********************************************************************

%----------------------------------------------------------------------
%%% TITLE & AUTHORS
%----------------------------------------------------------------------
\newcommand{\longtitle}{\uppercase{The Last-Iterate Convergence Rate of 
Optimistic Mirror Descent in Stochastic Variational Inequalities}}
\newcommand{\runtitle}{\uppercase{The Last-Iterate Convergence Rate of Optimistic Mirror Descent}}

\title
[\runtitle]		% for runtitle
{\longtitle}		% for title

%-------------------------------------------------------------------
\author[W.~Azizian]{Waïss Azizian$^{\ast}$}
\address{$^{\ast}$ DI, ENS, Univ.~PSL, 75005, Paris, France}
\email{waiss.azizian@ens.fr}

%-------------------------------------------------------------------
\author[F.~Iutzeler]{Franck Iutzeler$^{\dag}$}
\address{$^{\dag}$ Univ. Grenoble Alpes, LJK, Grenoble, 38000, France}
\email{franck.iutzeler@univ-grenoble-alpes.fr}

%-------------------------------------------------------------------
\author[J.~Malick]{\\Jérôme Malick$^{\ddag}$}
\address{$^{\ddag}$ Univ. Grenoble Alpes, CNRS, Grenoble INP, LJK, 38000 Grenoble, France}
\email{jerome.malick@univ-grenoble-alpes.fr}

%-------------------------------------------------------------------
\author
[P.~Mertikopoulos]
{Panayotis Mertikopoulos$^{\sharp}$}
\address{$^{\sharp}$\,%
Univ. Grenoble Alpes, CNRS, Inria, Grenoble INP, LIG, 38000, Grenoble, France.}
\address{$^{\diamond}$\,%
Criteo AI Lab.}
\email{panayotis.mertikopoulos@imag.fr}

%----------------------------------------------------------------------
%%% KEYWORDS
%----------------------------------------------------------------------


%----------------------------------------------------------------------
%%% ACKNOWLEDGMENTS
%----------------------------------------------------------------------
%\thanks{\input{Thanks.tex}}

%----------------------------------------------------------------------
%%% ACRONYMS
%----------------------------------------------------------------------
\newacro{LHS}{left-hand side}
\newacro{RHS}{right-hand side}
\newacro{iid}[i.i.d.]{independent and identically distributed}
\newacro{lsc}[l.s.c.]{lower semi-continuous}
\newacro{NE}{Nash equilibrium}
\newacroplural{NE}[NE]{Nash equilibria}

\newacro{ABP}{abstract Bregman proximal}
\newacro{BP}{Bregman proximal}

\newacro{DGF}{distance-generating function}
\newacro{EG}{extra-gradient}
\newacro{MP}{mirror-prox}
\newacro{MD}{mirror descent}
\newacro{OMD}{optimistic mirror descent}
\newacro{OMWU}{optimistic multiplicative weights update}
\newacro{PMP}{past mirror-prox}
\newacro{AMP}{abstract mirror-prox}
\newacro{MPT}{mirror-prox template}

\newacro{VI}{variational inequality}
\newacroplural{VI}[VIs]{variational inequalities}
\newacro{VIP}{variational inequality problem}
\newacro{KKT}{Karush\textendash Kuhn\textendash Tucker}
\newacro{FOS}{first-order stationary}
\newacro{SOO}{second-order optimality}
\newacro{SOS}{second-order sufficient}
\newacro{DGF}{distance-generating function}
\newacro{SFO}{stochastic first-order oracle}

%----------------------------------------------------------------------
%%% ABSTRACT
%----------------------------------------------------------------------
\begin{abstract}
\input{Abstract.tex}
\end{abstract}

%----------------------------------------------------------------------
%%% KEYWORDS
%----------------------------------------------------------------------

%**********************************************************************
%***    BODY TEXT
%**********************************************************************
\allowdisplaybreaks		% for breaking long displays
\acresetall		% for resetting acros
\maketitle

%----------------------------------------------------------------------
%%% INTRODUCTION
%----------------------------------------------------------------------
\section{Introduction}
\label{sec:introduction}
\input{Introduction.tex}

%----------------------------------------------------------------------
%%% SETUP
%----------------------------------------------------------------------
\section{Problem setup and preliminaries}
\label{sec:setup}
\input{Setup.tex}

%----------------------------------------------------------------------
%%% OMD
%----------------------------------------------------------------------
\section{The \acl{OMD} algorithm}
\label{sec:OMD}
\input{OMD.tex}

%----------------------------------------------------------------------
%%% LEGENDRE
%----------------------------------------------------------------------
\section{The geometry of \aclp{DGF}}
\label{sec:Bregman}
\input{Bregman.tex}

%----------------------------------------------------------------------
%%% GENERAL
%----------------------------------------------------------------------
\section{Analysis and results}
\label{sec:results}
\input{Results.tex}

%----------------------------------------------------------------------
%%% DISCUSSION
%----------------------------------------------------------------------
\section{Conclusion}
\label{sec:conclusion}
\input{Conclusion.tex}

%**********************************************************************
%***    APPENDICES
%**********************************************************************
\numberwithin{lemma}{section}		% for numbering  in the appendix
\numberwithin{proposition}{section}		% for numbering  in the appendix
\numberwithin{equation}{section}		% for numbering in the appendix
\appendix

%----------------------------------------------------------------------
%%% APP: AUXILIARY
%----------------------------------------------------------------------
\section{Auxiliary lemmas}
\label{app:lemmas}
\input{AppLemmas.tex}

%----------------------------------------------------------------------
%%% APP: OMD
%----------------------------------------------------------------------
\section{Descent and stability of optimistic mirror descent}
\label{app:omdsto}
\input{AppOMDSto.tex}

%----------------------------------------------------------------------
%%% APP: CONVERGENCE
%----------------------------------------------------------------------
\section{Convergence of optimistic mirror descent}
\label{app:omdcv}
\input{AppOMDCv.tex}

%----------------------------------------------------------------------
%%% APP: NUMERICS
%----------------------------------------------------------------------
\section{Numerical experiments}
\label{app:num}
\input{AppNum.tex}

%----------------------------------------------------------------------
%%% THANKS
%----------------------------------------------------------------------
\section*{Acknowledgments}
{\small\input{Thanks.tex}}

%**********************************************************************
%***    BIBLIOGRAPHY
%**********************************************************************
\bibliographystyle{plainfull}
\bibliography{bibtex/IEEEabrv,bibtex/Bibliography-PM}

\end{document}

%% file: Abstract.tex
%----------------------------------------------------------------------
%%% ABSTRACT
%----------------------------------------------------------------------
% !TEX root = ./Main.tex

In this paper, we analyze the local convergence rate of \acl{OMD} methods in stochastic \aclp{VI}, a class of optimization problems with important applications to learning theory and machine learning.
Our analysis reveals an intricate relation between the algorithm's rate of convergence and the local geometry induced by the method's underlying Bregman function.
We quantify this relation by means of the \emph{Legendre exponent}, a notion that we introduce to measure the growth rate of the Bregman divergence relative to the ambient norm near a solution.
We show that this exponent determines both the optimal step-size policy of the algorithm and the optimal rates attained, explaining in this way the differences observed for some popular Bregman functions (Euclidean projection, negative entropy, fractional power, etc.).

%% file: Introduction.tex
%----------------------------------------------------------------------
%%% INTRODUCTION
%----------------------------------------------------------------------
% !TEX root = ./Main.tex

Variational inequalities \textendash\ and, in particular, \emph{stochastic} \aclp{VI} \textendash\ have recently attracted considerable attention in machine learning and learning theory as a flexible paradigm for ``optimization beyond minimization'' \textendash\ \ie for problems where finding an optimal solution does not necessarily involve minimizing a loss function.
In this context, our paper examines the local rate of convergence of \acdef{OMD}, a state-of-the-art algorithmic template for solving \acp{VI} that incorporates an ``optimistic'' look-ahead step with a ``mirror descent'' apparatus relying on a suitably chosen Bregman kernel.
Our contributions focus exclusively on the stochastic case, which is of central interest to learning theory.
To put them in context, we begin with a general overview below and discuss more specialized references in \cref{sec:results}.

\para{General overview}

Algorithms for solving \aclp{VI} have a very rich history in optimization that goes back at least to the original proximal point algorithm of \cite{Mar70} and \cite{Roc76};
for a survey, see \cite{FP03}.
At a high level, if the vector field defining the problem is bounded and strictly monotone, simple forward-backward schemes are known to converge \textendash\ and if combined with a Polyak\textendash Ruppert averaging scheme, they achieve an $\bigoh{1/\sqrt{\run}}$ rate of convergence without the strictness caveat \citep{Bru77,Pas79,Nes09,BC17}.
If, in addition, the problem's defining vector field is Lipschitz continuous, the \acf{EG} algorithm of \cite{Kor76} achieves trajectory convergence without strict monotonicity requirements, while the method's ergodic average converges at a $\bigoh{1/\run}$ rate \citep{Nem04,Nes07}.
Finally, if the problem is \emph{strongly} monotone, forward-backward methods
%with a $1/\run$ step-size
achieve an $\bigoh{1/\run}$ convergence speed;
and if the operator is also Lipschitz continuous, classical results in operator theory show that simple forward-backward methods suffice to achieve a linear convergence rate \citep{FP03,BC17}.

The stochastic version of the problem arises in two closely related and interconnected ways:
the defining operator could have itself a stochastic structure, or the optimizer could only be able to access a stochastic estimate thereof.
Here, the landscape is considerably different:
For \ac{VI} problems with a monotone operator, the stochastic version of the \ac{EG} algorithm achieves an ergodic $\bigoh{1/\sqrt{\run}}$ convergence rate \citep{JNT11,GBVV+19}.
On the other hand, if the problem is \emph{strongly} monotone, this rate improves to $\bigoh{1/\run}$ and it is achieved by simple forward-backward methods, with or without averaging \citep{NJLS09}.

In this context, the \acl{OMD} algorithm has been designed to meet two complementing objectives:
\begin{enumerate*}
[(\itshape i\hspace*{.5pt}\upshape)]
\item
improve the dependence of the above rates on the problem's dimensionality in cases with a favorable geometry;
and
\item
minimize the number of oracle queries per iteration.
\end{enumerate*}
The first of these objectives is achieved by the ``\acl{MD}'' component:
by employing a suitable \acdef{DGF} \textendash\ like the negative entropy on the simplex or the log-barrier on the orthant) \textendash\ \acl{MD} achieves convergence rates that are (almost) dimension-free in problems with a favorable geometry.
This idea dates back to \cite{NY83}, and it is also the main building block of the \ac{MP} algorithm which achieves order-optimal rates with two oracle queries per iteration \citep{Nem04}.

The ``optimistic'' module then clicks on top of the \acs{MD}/\acs{MP} template by replacing one of the two operator queries by already observed information.
This ``information reuse'' idea was originally due to \cite{Pop80}, and it has recently resurfaced several times in learning theory, \cf \citet{CYLM+12}, \cite{RS13-NIPS,RS13-COLT}, \citet{DISZ18}, \citet{GBVV+19}, \citet{HIMM19,HAM21}, and references therein.
However, the quantitative behavior of the last-iterate in such mirror methods still remains an open question.

\para{Our contributions}
In view of the aove, the aim of our paper is to examine the local convergence rate of \ac{OMD} in stochastic \ac{VI} problems.
For generality, we focus on non-monotone \acp{VI}, and we investigate the algorithm's convergence to local solutions that satisfy a second-order sufficient condition.

In this regard, our first finding is that the algorithm's rate of convergence depends sharply on the local geometry induced by the method's underlying Bregman function.
We formalize this by introducing the notion of the \emph{Legendre exponent}, which can roughly be described as the logarithmic ratio of the volume of a regular ball w.r.t.~the ambient norm centered at the solution under study to that of a Bregman ball of the same radius.
For example, the ordinary Euclidean version of \ac{OMD} has a Legendre exponent of $\legexp=0$;
on the other hand, the entropic variant has a Legendre exponent of $\legexp=1/2$ on boundary points.

We then obtain the following rates as a function of $\legexp$:
\begin{equation}
\bigoh*{\run^{-(1-\legexp)}} ~\text{ if }~ 0\leq\legexp<1/2
\qquad\text{and}\qquad
\bigoh*{\run^{-\frac{1-\eps}{2\legexp} (1-\legexp)}} ~\text{ if }~ 1/2\leq\legexp<1
\end{equation}
for arbitrarily small $\eps>0$.
Interestingly, these guarantees undergo a first-order phase transition between the Euclidean-like phase ($0\leq\legexp<1/2$) and the entropy-like phase ($1/2\leq\legexp<1$).
This coincides with the dichotomy between steep and non-steep \acp{DGF} (\ie \acp{DGF} that are differentiable only on the interior of the problem's domain versus those that are differentiable throughout). Moreover, these rate guarantees are reached for different step-size schedules, respectively $\curr[\step] = \Theta(1/\run^{1-\legexp})$ and $\Theta(1/\run^{(1-\eps)/2})$, also depending on the Legendre exponent.

To the best of our knowledge, the only comparable result in the literature is the recent work of 
\citet{HIMM19} that derives an $\bigoh{1/\run}$ convergence rate for the Euclidean case (\ie when $\legexp=0$).
Along with some other recent works on optimistic gradient methods, we discuss this in detail in \cref{sec:results}, after the precise statement of our results.

%% file: Setup.tex
%----------------------------------------------------------------------
%%% SETUP
%----------------------------------------------------------------------
% !TEX root = ./Main.tex

%Owing to their relation to 

%----------------------------------------------------------------------
%%% PROBLEM
%----------------------------------------------------------------------
\para{Problem formulation, examples, and blanket assumptions}

Throughout what follows, we will focus on solving (Stampacchia) \aclp{VI}\acused{VI} of the general form:
\begin{equation}
\label{eq:VI}
\tag{VI}
\text{Find $\sol\in\points$ such that}
	\;\;
	\braket{\vecfield(\sol)}{\point - \sol}
	\geq 0
	\;\;
	\text{for all $\point\in\points$}.
\end{equation}
In the above,
$\points$ is a closed convex subset of a $\vdim$-dimensional real space $\pspace$ with norm $\norm{\cdot}$,
$\braket{\dpoint}{\point}$ denotes the canonical pairing between $\dpoint\in\dpoints$ and $\point\in\pspace$,
and
the \emph{defining vector field} $\vecfield\from\points\to\dpoints$ of the problem is a single-valued operator with values in the dual space $\dspace \defeq \pspace^{\ast}$ of $\pspace$.
Letting $\dnorm{\dpoint} \defeq \max \setdef{\braket{\dpoint}{\point}}{\norm{\point}\leq 1}$ denote the induced dual norm on $\dpoints$, we will make the following blanket assumption for $\vecfield$.

\begin{assumption}
[Lipschitz continuity]
\label{asm:Lipschitz}
The vector field $\vecfield$ is \emph{$\lips$-Lipschitz continuous}, \ie
\begin{equation}
\label{eq:Lipschitz}
\tag{LC}
\dnorm{\vecfield(\pointalt) - \vecfield(\point)}
	\leq \lips \norm{\pointalt - \point}
	\quad
	\text{for all $\point,\pointalt\in\points$}.
\end{equation}
\end{assumption}

For concreteness, we provide below some archetypal examples of \ac{VI} problems.

%%% Min problems
%----------------------------------------------------------------------
\begin{example}
%[Function minimization]
Consider the minimization problem
\begin{equation}
\label{eq:opt}
\tag{Min}
\begin{aligned}
\textrm{minimize}_{\point\in\points}
	&\quad
	\obj(\point)
%	\\
%\textrm{subject to}
%	&\quad
%	\point\in\points
\end{aligned}
\end{equation}
with $\obj\from\points\to\R$ assumed $C^{1}$-smooth.
Then, letting $\vecfield(\point) = \nabla\obj(\point)$, the solutions of \eqref{eq:VI} are precisely the \ac{KKT} points of \eqref{eq:opt}, \cf \cite{FP03}.
\endenv
\end{example}

%%% Min-max problems
%----------------------------------------------------------------------
\begin{example}
%[Saddle-point problems]
A \emph{saddle-point} \textendash\ or \emph{min-max} \textendash\ problem can be stated, in normal form, 
as
\begin{equation}
\label{eq:SP}
\tag{SP}
\min_{\minvar\in\minvars} \max_{\maxvar\in\maxvars} \, \minmax(\minvar,\maxvar)
\end{equation}
where $\minvars\subseteq\R^{\vdim_{1}}$ and $\maxvars\subseteq\R^{\vdim_{2}}$ are convex and closed, and the problem's objective function $\minmax\from\minvars \times \maxvars \to \R$ is again assumed to be smooth.
In the game-theoretic interpretation of \citet{vN28},
$\minvar$ is controlled by a player seeking to minimize $\minmax(\cdot,\maxvar)$,
whereas $\maxvar$ is controlled by a player seeking to maximize $\minmax(\minvar,\cdot)$.
Accordingly, solving \eqref{eq:SP} consists of finding a \emph{min-max point} $(\minsol,\maxsol) \in \points \defeq \minvars \times \maxvars$ such that
\begin{equation}
\label{eq:minmax}
\tag{MinMax}
\minmax(\minsol,\maxvar)
	\leq \minmax(\minsol,\maxsol)
	\leq \minmax(\minvar,\maxsol)
	\quad
	\text{for all $\minvar\in\minvars$, $\maxvar\in\maxvars$}.
\end{equation}
Min-max points may not exist if $\minmax$ is not convex-concave.
In this case, one looks instead for \acl{FOS} points of $\minmax$, \ie action profiles $(\minsol,\maxsol) \in \minvars \times \maxvars$ such that $\minsol$ is a \ac{KKT} point of $\minmax(\cdot,\maxsol)$ and %, respectively, 
$\maxsol$ is a \ac{KKT} point of $-\minmax(\minsol,\cdot)$.
%Then, 
Letting $\point = (\minvar,\maxvar)$
%$\points = \points_{1}\times\points_{2}$
and $\vecfield = (\nabla_{\minvar}\minmax,-\nabla_{\maxvar}\minmax)$, 
%it is straightforward to check 
we see that the solutions of \eqref{eq:VI} are precisely the \acl{FOS} points of $\minmax$.
\endenv
\end{example}

The above examples show that not all solutions of \eqref{eq:VI} are desirable:
for example, such a solution could be a local \emph{maximum} of $\obj$ in the case of \eqref{eq:opt} or a max-min point in the case of \eqref{eq:minmax}.
For this reason, we will concentrate on solutions of \eqref{eq:VI} that satisfy the following second-order condition:

\begin{assumption}
[Second-order sufficiency]
\label{asm:strong}
For a solution $\sol$ of \eqref{eq:VI}, there exists
a neighborhood $\basin$
and
a positive constant $\strong > 0$
such that
\begin{equation}
\label{eq:strong}
\tag{SOS}
\braket{\vecfield(\point)}{\point - \sol}
	\geq \strong \norm{\point - \sol}^{2}
	\quad
	\text{for all $\point\in\basin$}.
\end{equation}
\end{assumption}

In the context of \eqref{eq:opt}, \cref{asm:strong} implies that $\obj$ grows (at least) quadratically along every ray emanating from $\sol$;
in particular, we have $\obj(\point) - \obj(\sol) \geq \braket{\nabla\obj(\sol)}{\point - \sol} + (\strong/2) \norm{\point - \sol}^{2} = \Omega(\norm{\point-\sol}^{2})$ for all $\point\in\basin$ (though this does not mean that $\obj$ is strongly convex in $\basin$).
Likewise, in the context of \eqref{eq:minmax}, \cref{asm:strong} gives $\minmax(\minvar,\maxsol) = \Omega(\norm{\minvar - \minsol}^{2})$ and $\minmax(\minsol,\maxvar) = - \Omega(\norm{\maxvar - \maxsol}^{2})$, so $\sol$ is a local \acl{NE} of $\minmax$.
In general, \Cref{asm:strong} guarantees that $\sol$ is the unique solution of \eqref{eq:VI} in $\basin$.
(Indeed, any other solution $\test\neq\sol$ of \eqref{eq:VI} would satisfy $0 \geq \braket{\vecfield(\test)}{\test - \sol} \geq \strong \norm{\test - \sol}^{2} > 0$, a contradiction.)

%% file: OMD.tex
%----------------------------------------------------------------------
%%% OMD
%----------------------------------------------------------------------
% !TEX root = ./Main.tex

In this section, we %proceed to define the so-called 
recall the \acdef{OMD} method for solving \eqref{eq:VI}.
To streamline the flow of our paper, we first discuss the type of feedback available to the optimizer.

%----------------------------------------------------------------------
%%% Oracle
%----------------------------------------------------------------------
\subsection{The oracle model}

%Viewed abstractly, 
\ac{OMD} is a first-order method that only requires access to the problem's defining vector field via a ``black-box oracle'' that returns a (possibly imperfect) version of $\vecfield(\point)$ at the selected query point $\point\in\points$.
Concretely, we focus on \aclip{SFO} of the form
\begin{align}
\label{eq:SFO}
\tag{SFO}
\orcl(\point;\seed)
	&= \vecfield(\point)
	+ \err(\point;\seed)
\end{align}
where
$\seed$ is a random variable taking values in some abstract measure space $\seeds$
and
$\err(\point;\seed)$ is an umbrella error term capturing all sources of uncertainty in the model.
%(observational noise, partial payoff information in finite games, minibatch selection in machine learning applications, etc.).

In practice, the oracle is called repeatedly at a sequence of query points $\curr\in\points$ with a different random seed $\curr[\seed]$ at each time.%
\footnote{In the sequel, we will allow the index $\run$ to take both integer and half-integer values.}
The sequence of oracle signals $\curr[\signal] = \orcl(\curr;\curr[\seed])$ may then be written as
\begin{equation}
\label{eq:signal}
\curr[\signal]
	= \vecfield(\curr) + \curr[\noise]
\end{equation}
where $\curr[\noise] = \err(\curr;\curr[\seed])$ denotes the error of the oracle model at stage $\run$.
To keep track of this sequence of events, we will treat $\curr$ as a stochastic process on some complete probability space $\probspace$, and we will write $\curr[\filter] \defeq \filter(\init,\dotsc,\curr) \subseteq \filter$ for the \emph{history of play} up to stage $\run$ (inclusive).
Since the randomness entering the oracle is triggered only \emph{after} a query point has been selected, we will posit throughout that $\curr[\seed]$ (and hence $\curr[\noise]$ and $\curr[\signal]$) is $\next[\filter]$-measurable \textendash\ but not necessarily $\curr[\filter]$-measurable.
We will also make the blanket assumption below for the oracle.

\begin{assumption}
[Oracle signal]
\label{asm:oracle}
The error term $\noise_{\run}$ in \eqref{eq:signal} satisfies the following properties:
\vspace{-\smallskipamount}
\begin{subequations}
\label{eq:sigbounds}
\begin{alignat}{3}
\label{eq:bbound}
a)
	\quad
	&\textit{Zero-mean:}
	&\hspace{1em}
	&\exof{\curr[\noise] \given \curr[\filter]}
		= 0.
%	&\textit{Unbiasedness:}
%	&\hspace{1em}
%	&\exof{\curr[\signal] \given \curr[\filter]}
%		= \vecfield(\curr).
	\\
\label{eq:variance}
b)
	\quad
	&\textit{Finite variance:}
	&\hspace{1em}
	&\exof{\dnorm{\curr[\noise]}^{2} \given \curr[\filter]}
		\leq \sdev^{2}.
		\hspace{17em}
\end{alignat}
\end{subequations}
\end{assumption}

\noindent
Both assumptions are standard in the literature on stochastic methods in optimization, \cf \citet{Pol87}, \citet{Haz12}, \citet{Bub15}, and references therein.

%----------------------------------------------------------------------
%%% OMD
%----------------------------------------------------------------------
\subsection{\Acl{OMD}}

With all this in hand, the \acl{OMD} algorithm is defined in recursive form as follows:
\begin{equation}
\label{eq:OMD}
\tag{OMD}
\begin{aligned}
\lead
	&= \proxof{\curr}{-\curr[\step]\beforelead[\signal]}
	\\
\next
	&= \proxof{\curr}{-\curr[\step]\lead[\signal]}
\end{aligned}
\end{equation}
In the above,
$\run=\running$, is the algorithm's iteration counter,
$\curr$ and $\lead$ respectively denote the algorithm's \emph{base} and \emph{leading} states at stage $\run$,
$\lead[\signal]$ is an oracle signal obtained by querying \eqref{eq:SFO} at $\lead$,
and
$\proxof{\point}{\dpoint}$ denotes the method's so-called ``prox-mapping''.
In terms of initialization, we also take $\init = \initlead \in \relint\points$ for simplicity.
All these elements are defined in detail below.
%\PM{Put a phrase for the initialization, short and simple.}

At a high level, \eqref{eq:OMD} seeks to leverage past feedback to anticipate the landscape of $\vecfield$ and perform more informed steps in subsequent iterations.
In more detail, starting at some base state $\curr$, $\run=\running$, the algorithm first generates an intermediate, leading state $\lead$, and then updates the base state with oracle input from $\lead$ \textendash\  that is, $\lead[\signal]$.
This is also the main idea behind the \acl{EG} algorithm of \citet{Kor76};
the key difference here is that \ac{OMD} avoids making two oracle queries per iteration by using the oracle input $\beforelead[\signal]$ received at the previous leading state $\beforelead$ as a proxy for $\curr[\signal]$ (which is never requested or received by the optimizer).
%This ``gradient reuse'' idea dates back to \citet{Pop80}, and it has been subsequently popularized in machine learning and other fields by \citet{CYLM+12,RS13-NIPS,DISZ18,GBVV+19}, and many others;
%for an overview, see \citet{HIMM19} and references therein.
%\PM{N2S: should probably cite some MIT papers here\dots leave this up until I do it}

The second basic component of the method is the prox-mapping $\mprox$, which, loosely speaking, can be viewed as a generalized, proximal/projection operator adapted to the geometry of the problem's primitives.
To define it formally, we will need the notion of a \acli{DGF} on $\points$:

\begin{definition}
\label{def:Bregman}
A convex \acl{lsc} function $\hreg\from\pspace\to\R\cup\{\infty\}$ (with $\dom\hreg = \points$) is a \acdef{DGF} on $\points$ if
\begin{enumerate}
%\item
%$\hreg$ is proper, \ac{lsc} and convex.

% \item
% $\hreg$ is supported on $\points$, \ie $\dom\hreg = \points$.
\item
The subdifferential of $\hreg$ admits a \emph{continuous selection}, \ie there exists a continuous mapping $\nabla\hreg\from\dom\subd\hreg\to\dpoints$ such that $\nabla\hreg(\point) \in \subd\hreg(\point)$ for all $\point\in\dom\subd\hreg$.

\item
$\hreg$ is continuous and $1$-strongly convex on $\points$, \ie for all $\point\in\dom\subd\hreg, \pointalt\in\dom\hreg$, we have:
\begin{equation}
\label{eq:hstr}
\hreg(\pointalt)
	\geq \hreg(\point)
		+ \braket{\nabla\hreg(\point)}{\pointalt - \point}
		+ \tfrac{1}{2} \norm{\pointalt - \point}^{2}.
\end{equation}
\end{enumerate}
For posterity, the set $\proxdom \defeq \dom\subd\hreg$ will be referred to as the \emph{prox-domain} of $\hreg$.
We also define the \emph{Bregman divergence} of $\hreg$ as
\begin{alignat}{2}
\label{eq:Breg}
\breg(\base,\point)
	&= \hreg(\base)
		- \hreg(\point)
		- \braket{\nabla\hreg(\point)}{\base - \point},
	&\qquad
	&\text{for all $\base\in\points$, $\point\in\proxdom$}
\intertext{and the induced \emph{prox-mapping} $\mprox\from\proxdom\times\dspace\to\proxdom$ as}
\label{eq:prox}
\proxof{\point}{\dpoint}
	&= \argmin_{\pointalt\in\points} \{ \braket{\dpoint}{\point - \pointalt} + \breg(\pointalt,\point) \}
	&\qquad
	&\text{for all $\point\in\proxdom$, $\dpoint\in\dpoints$}.
\end{alignat}
\end{definition}

In the rest of this section we take a closer look at some commonly used \acp{DGF} and the corresponding prox-mappings;
for simplicity, we focus on one-dimensional problems on $\points = [0,1]$.

\begin{example}[Euclidean projection]
\label{ex:Eucl}
For the quadratic \ac{DGF} $\hreg(\point) = \point^{2}/2$ for $\point\in\points$,
the Bregman divergence is $\breg(\base,\point) = (\base-\point)^{2}/2$ and the induced prox-mapping is the Euclidean projector $\dis\proxof{\point}{\dpoint} = \clip{\point+\dpoint}_{0}^{1}$.
The prox-domain of $\hreg$ is the entire feasible region, \ie $\proxdom = \points = [0,1]$.
\endenv
\end{example}

\begin{example}[Negative entropy]
\label{ex:entropy}
Another popular example is the entropic \ac{DGF} $\hreg(\point) = \point\log\point + (1-\point)\log(1-\point)$.
The corresponding Bregman divergence is the relative entropy $\breg(\base,\point) = \base\log\frac{\base}{\point} + (1-\base) \log\frac{1-\base}{1-\point}$, and a standard calculation shows that $\proxof{\point}{\dpoint} = \point e^{\dpoint} / \bracks{(1-\point) + \point e^{\dpoint}}$.
In contrast to \cref{ex:Eucl}, the prox-domain of $\hreg$ is $\proxdom = %\relint\points =
(0,1)$.
For a (highly incomplete) list of references on the use of this \ac{DGF} in learning theory and optimization, \cf \citet{ACBFS95,BecTeb03,SS11,AHK12,CGM15,BMSS19,LS20}.
\endenv
\end{example}

\begin{example}[Tsallis entropy]
\label{ex:Tsallis}
An alternative to the Gibbs\textendash Shannon negentropy is the \emph{Tsallis} / \emph{fractional power} \ac{DGF} $\hreg(\point) = -\frac{1}{\qexp(1-\qexp)} \cdot \bracks{\point^{\qexp} + (1-\point)^{\qexp}}$ for $\qexp>0$ \citep{Tsa88,ABB04,MS16,MerSan18}.
Formally, to define $\hreg$ for $\qexp\gets 1$, we will employ the continuity convention $t^{\qexp}/(\qexp - 1) \gets \log t$, which yields the entropic \ac{DGF} of \cref{ex:entropy}.
Instead, for $\qexp\gets 2$, we readily get the Euclidean \ac{DGF} of \cref{ex:Eucl}.
The corresponding prox-mapping does not admit a closed-form expression for all $\pexp$, but it can always be calculated in logarithmic time with a simple binary search.
Finally, we note that the prox-domain $\proxdom$ of $\hreg$ is the entire space $\points$ for $\qexp>1$;
on the other hand, for $\qexp\in(0,1]$, we have $\proxdom = (0,1)$.
\endenv
\end{example}

The above examples will play a key role in illustrating the analysis to come and we will use them as running examples throughout.

%% file: Bregman.tex
%----------------------------------------------------------------------
%%% RESULTS
%----------------------------------------------------------------------
% !TEX root = ./Main.tex

%----------------------------------------------------------------------
%%% Topology
%----------------------------------------------------------------------
\para{The Bregman topology}

To proceed with our analysis and the statement of our results for \eqref{eq:OMD}, we will need to take a closer look at the geometry induced on $\points$ by the choice of $\hreg$.
The first observation is that, by the strong convexity requirement for $\hreg$, we have:
\begin{equation}
\label{eq:Breg-lower}
\breg(\base,\point)
	\geq \tfrac{1}{2} \norm{\base - \point}^{2}
	\quad
	\text{for all $\base\in\points$, $\point\in\proxdom$}.
\end{equation}
Topologically, this means that the convergence topology induced by the Bregman divergence of $\hreg$ on $\points$ is \emph{at least as fine} as the corresponding norm topology:
if a sequence $\curr[\point]$ converges to $\base\in\points$ in the Bregman sense ($\breg(\base,\curr[\point]) \to 0$), it also converges in the ordinary norm topology ($\norm{\curr[\point] - \base}\to0$).

On the other hand, the converse of this statement is, in general, false:
specifically, the norm topology could be \emph{strictly coarser} than the Bregman topology. % on $\points$.
To see this, let $\points$ be the unit Euclidean ball in $\R^{\vdim}$, \ie $\points = \setdef{\point\in\R^{\vdim}}{\twonorm{\point} \leq 1}$.
For this space, a popular choice of \ac{DGF} is the \emph{Hellinger regularizer} $\hreg(\point) = -\sqrt{1 - \twonorm{\point}^{2}}$, which has $\proxdom = \relint\points = \setdef{\point\in\R^{\vdim}}{\twonorm{\point} < 1}$ \citep{ABB04,BBT17}.
Then, for all $\base$ on the boundary of $\points$ (\ie $\twonorm{\base} = 1$), the Bregman divergence is
\begin{equation}
\label{eq:Breg-Hellinger}
\breg(\base,\point)
	= \frac{1 - \sum_{\coord=1}^{\vdim} \base_{\coord}\point_{\coord}}{\sqrt{1 - \twonorm{\point}^{2}}}.
\end{equation}
Shadowing the Euclidean definition, a ``Hellinger sphere'' of radius $r$ centered at $\base$ is defined as the corresponding level set of $\breg(\base,\point)$, \ie $S_{\base}(r) = \setdef{\point\in\proxdom}{\breg(\base,\point) = r}$.
Hence, by \eqref{eq:Breg-Hellinger}, this means that a point $\point\in\proxdom$ is at ``Hellinger distance'' $r$ relative to $\base$ if
\begin{equation}
\label{eq:Breg-level}
1 - \sum\nolimits_{\coord} \base_{\coord}\point_{\coord}
	= r \sqrt{1-\twonorm{\point}^{2}}.
\end{equation}
Now, if we take a sequence $\curr[\point]$ converging to $\base$ in the ordinary Euclidean sense (\ie $\twonorm{\curr[\point] - \base} \to 0$), both sides of \eqref{eq:Breg-level} converge to $0$.
This has several counterintuitive consequences:
\begin{enumerate}
[(\itshape i\hspace*{.5pt}\upshape)]
\item
$S_{r}(\base)$ is \emph{not closed} in the Euclidean topology;
and
\item
the ``Hellinger center'' $\base$ of $S_{r}(\base)$ actually belongs to the Euclidean closure of $S_{r}(\base)$.
\end{enumerate}
As a result, we can have a sequence $\curr[\point]$ with $\twonorm{\curr[\point] - \base} \to 0$, but which remains at \emph{constant} Hellinger divergence relative to $\base$.

From a qualitative standpoint, this discrepancy means that $\breg(\base,\point)$ is not a reliable measure of convergence of $\point$ to $\base$:
if $\breg(\base,\point)$ is large, this does not necessarily mean that $\base$ and $\point$ are topologically ``far''.
For this reason, in the analysis of non-Euclidean instances of Bregman proximal methods, it is common to make the so-called ``reciprocity'' assumption
\begin{equation}
\label{eq:reciprocity}
\breg(\base,\curr[\point])
	\to 0
	\quad
	\text{whenever}
	\quad
\curr[\point]
	\to \base.
\end{equation}
A straightforward verification shows that the \acp{DGF} of \cref{ex:Eucl,ex:entropy,ex:Tsallis} all satisfy \eqref{eq:reciprocity};
by contrast, this requirement of course fails for the Hellinger regularizer above.
For a detailed discussion of this condition, we refer the reader to \citet{CT93}, \citet{Kiw97}, \citet{BecTeb03}, \citet{ABB04}, \citet{MerSta18b}, \citet{MZ19}, \citet{ZMBB+20}, \citet{HAM21}, and references therein.

%----------------------------------------------------------------------
%%% Legendre
%----------------------------------------------------------------------
\para{The Legendre exponent}

From a \emph{quantitative} standpoint, even if \eqref{eq:reciprocity} is satisfied, the local geometry induced by a Bregman \ac{DGF} could still be significantly different from the base, Euclidean case.
A clear example of this is provided by contrasting the Euclidean and entropic \acp{DGF} of \cref{ex:Eucl,ex:entropy}.
Concretely, if we focus on the base point $\base = 0$, we have:
\smallskip
\begin{enumerate}
\item
In the Euclidean case (\cref{ex:Eucl}):
	$\breg(0,\point) = \point^{2}/2 = \Theta(\point^{2})$ for small $\point\geq0$;
\item
In the entropic case (\cref{ex:entropy}):
	$\breg(0,\point) = - \log(1-\point) = \Theta(\point)$ for small $\point\geq0$.
\end{enumerate}
Consequently, the rate of convergence of $\curr[\point]$ to $\base = 0$ is drastically different if it is measured relative to the Euclidean divergence of \cref{ex:Eucl} or the Kullback-Leibler divergence of \cref{ex:entropy}.
As an example, if we take the sequence $\curr[\point] = 1/\run$, we get $\breg(0,\curr[\point]) = \Theta(1/\run^{2})$ in terms of the Euclidean divergence, but $\breg(0,\curr[\point]) = \Theta(1/\run)$ in terms of the Kullback-Leibler divergence.

This disparity is due to the fact that it might be impossible to invert the inequality \eqref{eq:Breg-lower}, even to leading order.
To account for this, we introduce below the notion of the \emph{Legendre exponent} of $\hreg$.

\begin{definition}
\label{def:Legendre}
Let $\hreg$ be a \ac{DGF} on $\points$.
Then the \emph{Legendre exponent} of $\hreg$ at $\base\in\points$ is defined to be the smallest $\legexp\in[0,1)$ such that there exists a neighborhood $\legnhd$ of $\base$
%(in the standard topology of $\proxdom$)
and some $\legconst\geq0$ such that
%for which there exists a constant $\legconst\geq0$ such that
\begin{equation}
\label{eq:Legendre}
\breg(\base,\point)
%	= \bigoh*{\norm{\base - \point}^{2(1-\legexp)}}
	\leq \tfrac{1}{2} \legconst \norm{\base - \point}^{2(1-\legexp)}
	\quad
	\text{for all $\point\in\legnhd \cap \proxdom $}.
\end{equation}
\end{definition}

Heuristically, the Legendre exponent measures the difference in relative size between ordinary ``norm neighborhoods'' in $\points$ and the corresponding ``Bregman neighborhoods'' induced by the sublevel sets of the Bregman divergence.
It is straightforward to see that the requirement $\legexp\geq0$ is imposed by the strong convexity of $\hreg$:
since $\breg(\base,\point) = \Omega(\norm{\base - \point}^{2})$ for $\point$ close to $\base$, we cannot also have $\breg(\base,\point) = \bigoh{\norm{\base-\point}^{2+\eps}}$ for any $\eps>0$.
Note also that the notion of a Legendre exponent is only relevant if $\hreg$ satisfies the reciprocity requirement \eqref{eq:reciprocity}:
otherwise, we could have $\sup_{\point\in\nhd} \breg(\base,\point) = \infty$ for any neighborhood $\nhd$ of $\base$ in $\proxdom$, in which case we say that the Legendre exponent is $1$ by convention.

To see the Legendre exponents in action, we compute them for our running examples:
\begin{enumerate}
[left=\parindent]

\item
\emph{Euclidean projection}
(\cref{ex:Eucl}):
%$\breg(\base,\point) = \norm{\base-\point}^{2}/2$ 
$\breg(\base,\point) = (\base-\point)^{2}/2$ 
for all $\base,\point\in\points$, so $\legexp=0$ for all $\base$.

\item
\emph{Negative entropy}
(\cref{ex:entropy}):
If $\base\in\proxdom = (0, 1)$, a Taylor expansion with Lagrange remainder gives
% we can compute for any   $\point\in\proxdom$, the derivative $\nabla_{\point} \breg(\base,\point)|_{\point=\base} = 0$ and the second order derivative $\nabla^2_{\point \point} \breg(\base,\point)|_{\point=\base} =1/\base - 1/(1-\base)$. Thus, 
$\breg(\base,\point) = \hreg(\base) - \hreg(\point) - \hreg'(\point)(\base - \point) = \bigoh{(\base-\point)^2 }$ and thus $\legexp = 0$. 
Now if $\base=0$, $\breg(0,\point) = - \log(1-\point) = \Theta(\point)$ which implies $\legexp = 1/2$. Symmetrically, if $\base=1$, $\legexp = 1/2$\;too.

\item
\emph{Tsallis entropy}
(\cref{ex:Tsallis}):
Assume that $\qexp \neq 1$ (the case $\qexp = 1$ has been treated above).
% the Tsallis entropy becomes the negative entropy.
We then have two cases:
if $\base\in\proxdom = (0, 1)$, $\hreg$ is twice continuously differentiable in a neighborhood of $\base$ so that, as previously, by the Taylor formula,
$
\breg(\base, \point) 
%= \hreg(\base) - \hreg(\point) - \hreg'(\point)(\base - \point) = \bigoh{(\base - \point)^2}\,.
$
On the other hand, if $\base = 0$,
\begin{equation}
\breg(0, \point)
	= \frac{\point^\qexp}{\qexp}
		- \frac{1}{\qexp(1-\qexp)} \bracks{1 - (1 - \point)^\qexp - \qexp(1 - \point)^{\qexp - 1}\point}\,,
\end{equation}
so $\breg(0, \point) = \Theta(\point^{\min(\qexp, 2)})$ when $\point$ goes to zero.
By symmetry, the situation is the same at $\base = 1$, so the Legendre exponent of the Tsallis entropy is $\legexp= \max(0, 1 - \qexp/2)$.

\item
\emph{Hellinger regularizer:}
As we saw at the beginning of section, the Hellinger divergence \eqref{eq:Breg-Hellinger} does not satisfy \eqref{eq:reciprocity}, so $\legexp = 1$ by convention.
\end{enumerate}

\vspace*{1ex}

Of course, beyond our running examples, the same rationale applies 
when $\points$ is a subset of a $\vdim$-dimensional Euclidean space:
If $\base\in\proxdom$, then the associated Legendre exponent is typically $0$ (see \cref{lem:trivial-beta} for a precise statement), while if $\base$ lies at the boundary of $\points$, the exponent may be positive;
we illustrate this below for the important case of the simplex with the entropic \ac{DGF}.

\begin{example}[Entropy on the simplex]\label{ex:simplex} 
Consider the $\vdim$-dimensional simplex $\points=\{ \point\in\mathbb{R}_+^\vdim :  \sum_{i=1}^\vdim \point_i  = 1\}$ and the entropic regularizer $\hreg(\point) = \sum_{i=1}^\vdim \point_i \log \point_i$.
The associated divergence (also called Kullback-Leibler divergence) can be decomposed as
\begin{align}
\breg(\base,\point)
	&= \sum_{i:\base_i = 0} (\log \point_i + 1)\point_i - \point_i \log \point_i
	\notag\\
	&+ \sum_{i:\base_i \neq 0} \base_i \log \base_i - \point_i \log \point_i - (\log \point_i + 1)(\base_i - \point_i).
\end{align}
The first sum is simply $ \sum_{i : \base_i = 0} \point_i$ and the terms of the second are bounded by $\bigoh{(\base_i - \point_i)^2}$, so
\begin{align}
\breg(\base, \point)
	&\txs
	= \sum_{i : \base_i = 0} \point_i + \bigoh*{\sum_{i : \base_i \neq 0} (\base_i - \point_i)^2}.
\end{align}
If $\base\in\proxdom = \{ \point\in\mathbb{R}_{++}^\vdim :  \sum_{i=1}^\vdim \point_i  = 1\}$, we get $\legexp = 0$.
If $\base\not\in\proxdom$ 
there is at least one %index $i$ such that 
$\base_i = 0$, $\breg(\base, \point)$ is not bounded by $\bigoh{\norm{\base - \point}^2}$ but by $\bigoh{\norm{\base - \point}}$ so that $\legexp = 1/2$ for such $\base$.
\endenv
\end{example}

%% file: Results.tex
%----------------------------------------------------------------------
%%% RESULTS
%----------------------------------------------------------------------
% !TEX root = ./Main.tex

%----------------------------------------------------------------------
%%% Result
%----------------------------------------------------------------------
\subsection{Main result and discussion}

We are finally in a position to state and prove our main result for the convergence rate of \eqref{eq:OMD}.
Since we make no global monotonicity assumptions for the defining vector field $\vecfield$, this result is \emph{de facto} of a local nature.
We provide below the complete statement, followed by a series of remarks and an outline of the main steps of the proof.

\begin{theorem}
\label{thm:rates}
Let $\sol$ be a solution of \eqref{eq:VI} and fix some confidence level $\thres>0$.
Suppose further that \cref{asm:Lipschitz,asm:strong,asm:oracle} hold and \eqref{eq:OMD} is run with
a step-size of the form $\curr[\step] = \step/(\run+t_0)^\expstep$ with $\expstep\in(1/2,1]$ and $\step,t_0> 0$.
Then, \eqref{eq:OMD} enjoys the following properties:
\begin{description}
\item[\textit{Stochastic stability}]
If $\step/\beforeinit[\run]$ is small enough,
$\sol$ is stochastically stable:
for every neighborhood $\nhd$ of $\sol$, there exists a neighborhood $\inhd$ of $\sol$ such that the event
\begin{equation}
\label{eq:stable}
\event_{\nhd}
	= \{ \text{$\curr\in\nhd$ for all $\run = \running$} \}
\end{equation}
occurs with probability at least $1-\thres$ if $\init\in\inhd$.

\item[\textit{Convergence}]
$\curr$ converges to $\sol$ at the following rate:
\begin{enumerate}
[leftmargin=3em,font=\bfseries,label={Case \Roman*:}]

\item
If the Legendre exponent of $\hreg$ at $\sol$ is $\legexp\in(0,1)$ and $\min\{\step,1/t_{0}\}$ is small enough,
the iterates $\curr$ of \eqref{eq:OMD} enjoy the convergence rate
\begin{equation}
\label{eq:rate}
\exof{ \breg(\sol,\curr) \given \event_{\nhd} }
	= \begin{cases}
		\bigoh*{(\log\run)^{-(1-\legexp)/\legexp}}
			&\quad
			\text{if $\expstep = 1$},
			\\
		\bigoh*{\run^{-\min\{(1-\expstep)(1-\legexp)/\legexp,\expstep\}}}
			&\quad
			\text{if $\frac{1}{2} < \expstep < 1$},
	\end{cases}
\end{equation}
provided that $\nhd$ is small enough and $\init\in\inhd$.
%provided that $\init = \initlead \in\inhd$.
In particular, this rate undergoes a phase transition at $\legexp=1/2$:
%\PM{I incorporated the $\init=\initlead$ condition in \eqref{eq:OMD} so I simplified here.}
\begin{itemize}
[left=0pt]
\item
Euclidean-like phase:
If $\legexp \in (0,1/2)$ and $\expstep = 1-\legexp$, we obtain the optimized rate
\begin{equation}
\label{eq:rate-opt-eucl}
\exof{\breg(\sol,\curr) \given \event_{\nhd}}
	= \bigoh*{\run^{-(1-\legexp)}}.
\end{equation}
\item
Entropy-like phase:
If $\legexp \in [1/2,1)$ and $\expstep = (1+\epsilon)/2$ for any small $\varepsilon>0$,
%taking $1/2 < \expstep < 1$,
we get
\begin{equation}
\label{eq:rate-opt-leg}
\exof{\breg(\sol,\curr) \given \event_{\nhd}}
	= \bigoh*{\run^{-\frac{1-\varepsilon}{2\legexp} (1-\legexp)}}.
\end{equation}
\end{itemize}

\item
If $\legexp=0$ and $\step$ and $t_{0}$ are sufficiently large,
%then
%\begin{enumerate*}
%[\upshape(\itshape \alph*\hspace*{.5pt}\upshape)]
%\item
%\item
the iterates $\curr$ of \eqref{eq:OMD} enjoy the rate
%\end{enumerate*}
\begin{equation}
\label{eq:rate-quad}
\exof{ \breg(\sol,\curr) \given \event_{\nhd} }
	= \bigoh{\run^{-\expstep}},
\end{equation}
provided that $\nhd$ is small enough and $\init\in\inhd$.
In particular, for $\expstep=1$ we get
\begin{equation}
\label{eq:rate-opt-quad}
\exof{ \breg(\sol,\curr) \given \event_{\nhd} }
	= \bigoh{\run^{-1}}.
\end{equation}
\end{enumerate}
\end{description}
\end{theorem}

%----------------------------------------------------------------------
%% Rate figure begins here
%----------------------------------------------------------------------
\begin{figure}[h!]
\centering
\begin{subfigure}[t]{0.43\textwidth}
\centering
\small
	\begin{tikzpicture}[scale=4]
	\draw[thick,-stealth] (-0.1, 0) -- (1.1, 0) node[right] {$\legexp$};
	\draw (0.05,0) node[below] {$0$};
	\draw (1.0,0.025) -- (1.0,-0.025) node[below] {$1$};
	\draw (0.5,0.025) -- (0.5,-0.025) node[below] {$\frac{1}{2}$};
	\draw[dashed,gray] (0, 0.5) -- (1,0.5);
	\draw[thick,-stealth] (0, -0.1) -- (0, 1.1) node[above] {Rate exponent ($\nu$)};
	\draw (0,0.05) node[left] {$0$};
	\draw (0.025,1.0) -- (-0.025,1.0) node[left] {$1$};
	\draw[dashed,gray] (0.5, 0) -- (0.5, 1);
	\draw (0.025,0.5) -- (-0.025,0.5) node[left] {$\frac{1}{2}$};
	\draw[very thick,domain=0:0.5, smooth, variable=\x, magenta] plot ({\x}, {1-\x});
	\draw (0.1,0.92) node[right,color=magenta] {{\footnotesize Euclidean-like phase}};
	\draw[very thick,domain=0.5:1.0, smooth, variable=\x, blue,dashed] plot ({\x}, {1/(2*\x)-0.5});
	\draw (0.24,0.05) node[right,color=blue] {{\footnotesize Entropy-like phase}};
	\end{tikzpicture}
	\caption{Predicted rate $\bigoh{1/\run^{\nu}}$ vs.\;Legendre exponent $\legexp$}
	\label{fig:comp-ite}
\end{subfigure}
\hfill
\begin{subfigure}[t]{0.56\textwidth}
\centering
	\resizebox{\textwidth}{!}{
	\input{Figs/correctness_rate}
	}
	\caption{Fit of the predicted rates with the simulation.}
	\label{fig:comp-breg}
\end{subfigure}

\caption{Illustration of the rates predicted by \cref{thm:rates} (left: theory; right: numerical validation).
The dashed line in \cref{fig:comp-ite} indicates that the rate prediction is only valid up to an arbitrarily small exponent or suitable logarithmic factor.
The experiments in \cref{fig:comp-breg} were conducted for a linear vector field $\vecfield(\point) = \point $ on $\points=[0,+\infty)$ with the regularizer setups of \cref{ex:Eucl,ex:entropy,ex:Tsallis} (for more details, see \cref{app:num}). We run \ac{OMD} with the stepsizes prescribed by  \cref{thm:rates} and plot $\breg(\sol,\curr)\times\run^\nu$ where $\nu$ is the predicted rate ($\nu=1, 0.5, 0.167, 0.75$ respectively for Euclidean, entropy, Tsallis entropy with $q=0.5$ and $q=1.5$). A horizontal line corresponds to a match between the theoretical and observed rates. }
\label{fig:rates}
\end{figure}
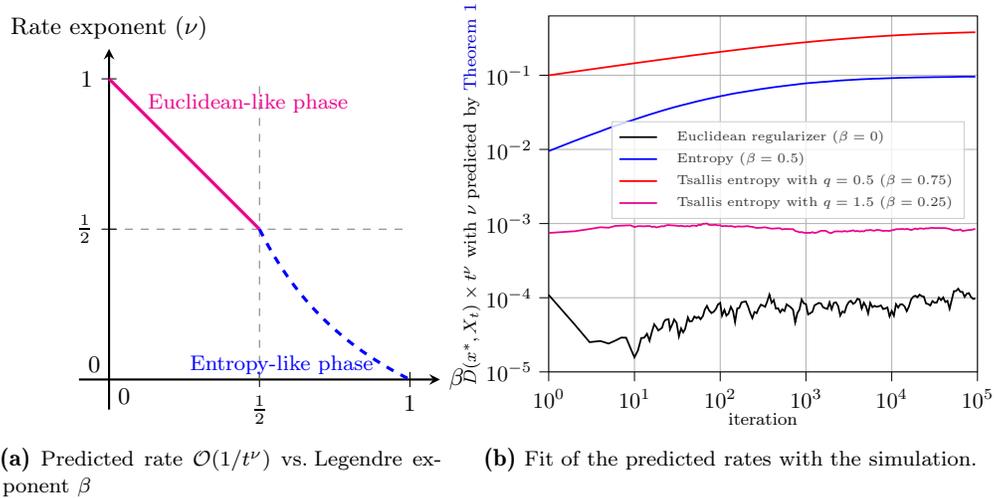
%----------------------------------------------------------------------
%% Rate figure ends here
%----------------------------------------------------------------------

% to be delete or expand. To save time I comment out.
%\Cref{thm:rates} is our main result, so several remarks are in order.

\para{Related work}

%To begin with, 
The rate guarantees of \cref{thm:rates} should be compared and contrasted to a series of recent results on the convergence speed of \ac{OMD} and its variants. % in \aclp{VI}.
Most of these results concern the deterministic case, \ie when $\curr[\noise] = 0$ for all $\run$.
In this regime, and focusing on the last iterate of the method, \citet{GPD20,GPDO20} showed that the  convergence speed of the (optimistic) extra-gradient algorithm in unconstrained, smooth, monotone variational inequalities is $\bigoh{1/\sqrt{\run}}$ in terms of $\norm{\vecfield(\curr)}$.
If, in addition, $\vecfield$ is \emph{strongly} monotone, it is well known that this rate becomes linear;
for a recent take on different (Euclidean) variants of \eqref{eq:OMD}, see \citet{Mal15}, \citet{GBVV+19}, \citet{HIMM19}, \citet{MOP19a,MOP19b}, and references therein.
%Importantly, these rates are typically achieved with a \emph{constant} step-size that is fine-tuned in advance as a function of the problem's Lipschitz continuity modulus \textendash\ and, in problems that are not strongly monotone, the algorithm's running length.

Unsurprisingly, the stochastic regime is fundamentally different.
In the merely monotone case (\ie neither strongly nor strictly monotone), we are not aware of any result on the method's last-iterate convergence speed:
existing results in the literature either cover the algorithm's ergodic average \citep{JNT11,GBVV+19,CS16} or are asymptotic in nature \citep{IJOT17}.
Because of the lack of deterministic correlation between gradients at the base and leading states in stochastic problems, smoothness does not help to achieve better rates in this case \textendash\ except for mollifying the requirement for bounded gradient signals % that are bounded in $L^{2}$,\cf
\citep{JNT11}.
By contrast, strong monotonicity \textendash\ local or global \textendash\ \emph{does} help:
as was shown by \citet{HIMM19}, the iterates of the \emph{Euclidean} version of \eqref{eq:OMD} run with a $1/\run$ step-size schedule converge locally to solutions satisfying \eqref{eq:strong} at a $\bigoh{1/\run}$ rate in terms of the mean square distance to the solution.
This result is a special instance of Case II of \cref{thm:rates}:
in particular, for $\legexp=0$, $\expstep=1$, \eqref{eq:rate-quad} indicates that the $1/\run$ step-size schedule examined by \cite{HIMM19} is optimal for the Euclidean case;
however, as we explain below, it is \emph{not} optimal for non-Euclidean \acp{DGF}.

\para{Step-size tuning}

A key take-away from \cref{thm:rates} is that the step-size schedule that optimizes the rate guarantees in \cref{eq:rate-opt-eucl,eq:rate-opt-leg} is either $\curr[\step] = \Theta(1/\run^{1-\legexp})$ when $0\leq\legexp<1/2$, or as close as possible to $\Theta(\run^{1/2})$ when $1/2 \leq \legexp < 1$;
in both cases, the optimal tuning of the algorithm's step-size \emph{depends on the Legendre exponent of $\sol$}. We feel that this is an important parameter to keep in mind when deploying non-Euclidean versions of \eqref{eq:OMD} in a stochastic setting \textendash\ \eg as in the \ac{OMWU} algorithm of \citet{DP19}.
\cref{thm:rates} suggests that the ``best practices'' of the Euclidean setting do not carry over to non-Euclidean ones:
Euclidean step-size policy $1/\run$ would lead to a catastrophic drop from $1/\sqrt{\run}$ to $1/\log\run$ in the rate guarantees for \eqref{eq:OMD}.

\para{Geometry-dependent rates}

One final observation for \cref{thm:rates} is that the convergence speed of \eqref{eq:OMD} is measured in terms of the Bregman divergence $\breg(\sol,\curr)$.
Since $\norm{\curr - \sol} = \bigoh{\sqrt{\breg(\sol,\curr)}}$, this rate immediately translates to a rate on $\norm{\curr - \sol}$.
However, this rate is \emph{not} tight, again because of the disparity between the Bregman and norm geometries discussed in \cref{sec:Bregman}.
Specifically, in many cases, a non-zero Legendre exponent implies a sharper lower bound on the Bregman divergence of the form $\breg(\sol,\curr) = \Omega(\norm{\curr-\sol}^{2-2\legexp})$.
Thus, for $\legexp \in [0,1/2)$, Jensen's inequality applied to the optimized bound \eqref{eq:rate-opt-eucl} gives $\exof{\norm{\curr - \sol}} = \bigoh{1/\sqrt{\run}}$.
Importantly, this rate guarantee is the same as in the Euclidean case ($\legexp = 0$), though it %still 
requires a completely different step-size policy to achieve it.
Determining when the Legendre exponent also provides a lower bound is a very interesting direction for future research, but one that lies beyond the scope of this work.

%----------------------------------------------------------------------
%%% Proof
%----------------------------------------------------------------------
\subsection{Main ideas of the proof}

The heavy lifting in the proof of \cref{thm:rates} is provided by \cref{prop: last iterate past mirror prox stochastic} below, which also makes explicit some of the hidden constants in the statement of the theorem.

\begin{proposition}
\label{prop: last iterate past mirror prox stochastic}
With notation and assumptions as in \cref{thm:rates}, fix some $r>0$ and let
\begin{equation*}
\nhd
	= \ball({\sol},{r}) \cap \points
	\quad
	\text{and}
	\quad
\inhd
	= \setdef{\point}{\breg(\sol, \point) \leq r^{2}/12}
%	\quad
%	\text{for some $r>0$}
	.
\end{equation*}
%and take $\init[\state] = \initlead[\state] \in \inhd$.
If $\init\in\inhd$, then the event $\event_{\nhd}$ defined in \eqref{eq:stable} satisfies
%\begin{equation}
$\probof{\event_{\nhd}}
%	\geq 1 - \bigoh*{\frac{\sdev^2 \step^2} {(2\expstep - 1)t_0^{2\expstep-1}r^2}}\,.
	\geq 1 - \bigoh*{{\rho^{2}}/{r^{2}}}$
%\end{equation}
where $\rho^{2} = \sdev^{2} \step^{2} t_{0}^{1-2\expstep} / (2\expstep - 1)$.
Moreover, if $r$ is taken small enough, $\ex\bracks*{\breg(\sol, \curr) \mid \event_\nhd}$ is bounded according to the following table and conditions:

%----------------------------------------------------------------------
%% Rate table begins here

\begin{center}
\renewcommand{\arraystretch}{1.2}
\setlength{\tabcolsep}{1em}
\small
\begin{tabular}{lccc}
\toprule
Legendre exponent
	\quad
	&Rate $(\expstep=1)$
	&Rate $(\tfrac{1}{2} < \expstep < 1)$
	\\
\midrule
$\legexp \in (0,1)$
	&$\bigoh*{(\log\run)^{-\frac{1-\legexp}{\legexp}}}$
	&$\bigoh*{\run^{-\frac{(1-\expstep)(1-\legexp)}{\legexp}} + \run^{-\expstep}}$
	\\
	\addlinespace[1ex]
Conditions:
	&\multicolumn{2}{c}{$\step^{1 + 1/\legexp} \leq \frac{c(\expstep,\leg) \max(1,\Phi)}{4\sdev^{2}(\strong/\legconst)^{{1/\legexp - 1}}}$\hspace*{1em}}
	\\
	\midrule
$\legexp = 0$
	&$\bigoh{1/\run}$
	&$\bigoh{1/\run^{\expstep}}$
	\\
	\addlinespace[\smallskipamount]
Conditions:
	&$\step > \legconst/\strong$
	&\textemdash
	\\
\bottomrule
\end{tabular}
\end{center}

%% Rate table ends here
%----------------------------------------------------------------------

where
$\leg = \legexp / (1-\legexp)$,
$\Phi = r^{2}/12 + 8\rho^{2}$,
and
\begin{equation}
c(\expstep,\leg)
	= \begin{cases}
		\frac{1 - \expstep}{(1+\leg)^{1+1/\leg}}
			&\quad
			\text{if $\expstep \leq \frac{1+\leg}{1+2\leg}$},
		\\[\smallskipamount]
		\leg\left(\frac{1 - 2^{-(1+\leg)}}{1+\leg}\right)^{1 + 1/\leg}
		&\quad
		\text{if $\expstep > \frac{1+\leg}{1+2\leg}$}.
	\end{cases}
\end{equation}
\end{proposition}

We outline here the main ideas of the proof, deferring all details to \cref{app:omdcv}.

\begin{proof}[Sketch of proof]
Our proof strategy is to show that, under the stated conditions for $r$ and the algorithm's initialization, the iterates of \eqref{eq:OMD} remain within a neighborhood of $\sol$ where \eqref{eq:strong} holds with high probability.
Then, conditioning on this event, we use the Bregman divergence as a stochastic potential function, and we derive the algorithm's rate of convergence using a series of lemmas on (random) numerical sequences.
The step-by-step process is as follows:

\begin{enumerate}
[leftmargin=\parindent]

\item
The cornerstone of the proof is a descent inequality for the iterates of \ac{OMD}, which relates $\breg(\sol, \next)$ to $\breg(\sol, \curr)$.
As stated in \cref{lemma:descent}, we have
\begin{align}
\label{eq: proof sketch past mirror prox stoch descent ineq2}
\breg(\sol, \next) + \phi_{\run+1}
	&\leq \breg(\sol,\curr)
		+ (1 - \curr[\step] \strong)\phi_\run - \curr[\step] \inner{\lead[\signal], \lead - \sol}
	\\
	&+ \parens*{ 4\curr[\step]^2\lips^2 - \half } \norm{\lead - \curr}^{2}
		+ 4\curr[\step]^2 \bracks*{ \dnorm{\lead[\noise]}^{2} + \dnorm{\beforelead[\noise]}^{2} }\notag
\end{align}
where
$
	\phi_{\run} = \frac{\prev[\step]^2}{2}\|\beforelead[\signal] - \beforebeforelead[\signal] \|^2$ for all $\run \geq 2
$, 
and $\phi_1 = 0$. %since X_1 = X_1/2 now. 
%and $\phi_1 = \half\|\init[\state] - \initlead[\state]\|^2$.

\item
The next step is to show that, with high probability and if initialized close to $\sol$, the iterates of \ac{OMD} remain in $\nhd$.
This is a technical argument relying on a use of the Doob-Kolmogorov maximal inequality for submartingales (in the spirit of \cite{HIMM20}) and the condition $\expstep>1/2$ (which in turn guarantees that the step-size is square summable, so the corresponding submartingale error terms are bounded in $L^{1}$).
The formal statement is \cref{lemma: stochastic stability}.% in \cref{app:omdsto}.

\item
To proceed, choose $r > 0$ small enough so that the second order sufficiency condition \eqref{eq:strong} holds and the Legendre exponent estimate \eqref{eq:Legendre} both hold.
% and the Legendre exponent of $\hreg$ at $\sol$ (\cref{def:Legendre}).
More precisely, we take $r$ small enough so that $\pnhd =\ball({\sol},{r}) \cap \points$  is included in the domains of validity of both properties ($\basin$ and $\legnhd$ respectively), and we will work on the event
$\curr[\event] = \{\state_{\runalt+1/2} \in \pnhd \; \text{for all $\runalt=\running,\run$} \}$.
%\cap\inhd$.
%\left\{\breg(\sol, \init[\state])+\phi_1 \leq \frac{r^2}{12}\right\}$ for $\run \geq 0$.
Then, conditioning on $\curr[\event]$, \cref{asm:strong} gives
%we get that $\lead$ is in the basin $\basin$ of \cref{asm:strong},
\begin{equation*}
	\inner{\vecfield(\lead), \lead - \sol} \geq \frac{\strong}{2}\|\lead - \sol\|^2\,,
\end{equation*}
and hence, by \cref{lemma: strg mon extra}, we get
%which implies, as guaranteed by \cref{lemma: strg mon extra},
	\begin{equation*}
		\inner{\vecfield(\lead), \lead - \sol} \geq \frac{\strong}{2}\|\curr - \sol\|^2 - \strong \|\lead - \curr\|^2\,.
	\end{equation*}
However, the descent inequality only involves $\breg(\sol, \curr)$ and not $\norm{\curr - \sol}^2$, so we need to relate the latter to the former. This is where the Legendre exponent comes into play. More precisely, \cref{def:Legendre} with $1+\leg = \frac{1}{1 - \legexp}$ gives
\begin{equation*}%\label{eq:newlegendre}
\breg(\sol,\point)^{1+\leg}
%	= \bigoh*{\norm{\base - \point}^{2(1-\legexp)}}
	\leq \tfrac{1}{2} \legconst \norm{\sol - \point}^{2},
\end{equation*}
so, with $\lead\in\legnhd$, we have
\begin{equation*}
\inner{\vecfield(\lead), \lead - \sol} \geq \frac{\strong}{\legconst}\breg(\sol, \curr)^{1+\leg} - \strong \|\lead - \curr\|^2\,.
\end{equation*}
	
\item
Thus, going back to the descent inequality \eqref{eq: proof sketch past mirror prox stoch descent ineq2}, employing the bound derived above, and taking expectations ultimately yields (after some technical calculations) the iterative bound
% this inequality with the bound above and taking the expectation, gives, after a few technical steps,
\begin{multline}
\ex\left[(\breg(\sol, \next) + \phi_{\run+1})\mathds{1}_{\curr[\event]}\right]
	\\
\leq \ex\left[\left(\breg(\sol, \curr)-\frac{\strong\curr[\step]}{\legconst}\breg(\sol, \curr)^{1+\leg} + \left(1 - \frac{\curr[\step] \strong}{\legconst}\right)\phi_\run\right)\mathds{1}_{\prev[\event]}\right]  + 8\curr[\step]^2\sdev^2
\label{eq: proof sketch past mirror prox stoch descent ineq3}
\end{multline}
where we used \cref{asm:oracle} to bound the expectations of $\|\lead[\noise]\|_*^2$ and 
$\|\beforelead[\noise]\|_*^2$.

\item
We are now in a position to reduce the problem under study to the behavior of the sequence $a_{\run} =  \ex\left[(\breg(\sol,\curr) + \phi_{t})\mathds{1}_{\prev[\event]}\right]$, $\run\geq 1$.
Indeed, taking a closer look at $\curr[\event]$ shows that
\begin{equation*}
	\ex\bracks*{\breg(\sol, \curr) \mid \event_\nhd} = \bigoh{\curr[a]}\,.
\end{equation*}
Consequently, the inequality \eqref{eq: proof sketch past mirror prox stoch descent ineq3} above can be rewritten as follows:
\begin{align}
\label{eq:sketch ratebase}
a_{\run+1} \leq a_{\run} -  \frac{\strong\curr[\step]}{\legconst  } \ex\left[ \left( \breg(\sol, \curr)^{1+\leg} +  \phi_\run \right) \mathds{1}_{\prev[\event]} \right]  + 8\curr[\step]^2\sdev^2\,.
\end{align}

\item
The behavior of this last sequence hinges heavily on whether $\leg > 0$ or not.
In detail, we have:
\begin{itemize}
\item
If $\leg=0$, \eqref{eq:sketch ratebase} gives
\vspace*{-3ex}
\begin{equation*}
a_{\run+1}
	\leq \left(1 - \frac{\strong\step}{\legconst (\run+t_0)^\expstep } \right) a_{\run}  + \frac{8\step^2\sdev^2}{(\run+t_0)^{2\expstep}}\, .
\end{equation*}
The long-run behavior of this recursive inequality is described by \cref{lemma: chung 1,lemma: chung 4}:
as long as $\strong\step/\legexp > 1$ for $\expstep=1$ (and only for $\expstep=1$), we have $a_{\run} = \bigoh{1/\run^{\expstep}}$.

\item
If $\leg>0$, a series of further technical calculations in the same spirit %as above 
yields %starting from \eqref{eq:sketch ratebase} and applying a couple of technical steps gives,
\begin{equation*}
a_{\run+1} 
	\leq a_{\run} -  \frac{\strong\step}{ 2^\leg \max(1,\Phi^\leg) \legconst (\run + \beforeinit[\run])^\expstep} a_{\run}^{1+\leg}
		+ \frac{8\step^2\sdev^2}{(\run+t_0)^{2\expstep}}
\end{equation*}
The final step of our proof is to upper bound the behavior of $\curr[a]$ based on this recursive inequality.
The necessary groundwork is provided by \cref{lemma: stoch sequence alpha,lemma: stoch sequence alpha bis}, and hinges on whether $\expstep \geq \frac{1+\leg}{1+2\leg}$ or not (and is also where the value of $c(\expstep,\leg)$ comes in).
\end{itemize}
\end{enumerate}
Putting everything together, we obtain the conditional rates in the statement of the proposition.
\end{proof}

%% file: Figs/correctness_rate.tex
% This file was created by tikzplotlib v0.9.6.
\begin{tikzpicture}

\begin{axis}[
legend cell align={left},
legend style={font=\tiny, fill opacity=0.8, draw opacity=1, text opacity=1, at={(0.97,0.43)}, anchor=south east, draw=white!80!black},
log basis x={10},
log basis y={10},
tick align=outside,
tick pos=left,
x grid style={white!69.0196078431373!black},
xlabel={{\footnotesize iteration}},
xmajorgrids,
xmin=1, xmax=1e5,
xmode=log,
xtick style={color=black},
xtick={0.01,0.1,1,10,100,1000,10000,100000,1000000,10000000},
xticklabels={\(\displaystyle {10^{-2}}\),\(\displaystyle {10^{-1}}\),\(\displaystyle {10^{0}}\),\(\displaystyle {10^{1}}\),\(\displaystyle {10^{2}}\),\(\displaystyle {10^{3}}\),\(\displaystyle {10^{4}}\),\(\displaystyle {10^{5}}\),\(\displaystyle {10^{6}}\),\(\displaystyle {10^{7}}\)},
y grid style={white!69.0196078431373!black},
ylabel={{\footnotesize $\breg(\sol,\curr) \times \run^\nu$ with $\nu$ predicted by \cref{thm:rates}}},
ymajorgrids,
ymin=1e-5, ymax=0.629721722701674,
ymode=log,
ytick style={color=black},
ytick={1e-07,1e-06,1e-05,0.0001,0.001,0.01,0.1,1,10},
yticklabels={\(\displaystyle {10^{-7}}\),\(\displaystyle {10^{-6}}\),\(\displaystyle {10^{-5}}\),\(\displaystyle {10^{-4}}\),\(\displaystyle {10^{-3}}\),\(\displaystyle {10^{-2}}\),\(\displaystyle {10^{-1}}\),\(\displaystyle {10^{0}}\),\(\displaystyle {10^{1}}\)}
]
\addplot [thick, black]
table {%
1 0.000109114264001614
1 0.000109114264001614
1 0.000109114264001614
1 0.000109114264001614
1 0.000109114264001614
1 0.000109114264001614
1 0.000109114264001614
1 0.000109114264001614
1 0.000109114264001614
1 0.000109114264001614
1 0.000109114264001614
1 0.000109114264001614
1 0.000109114264001614
2 4.38035233874455e-05
2 4.38035233874455e-05
2 4.38035233874455e-05
2 4.38035233874455e-05
2 4.38035233874455e-05
2 4.38035233874455e-05
2 4.38035233874455e-05
3 2.51701213962835e-05
3 2.51701213962835e-05
3 2.51701213962835e-05
3 2.51701213962835e-05
3 2.51701213962835e-05
4 2.61076625730117e-05
4 2.61076625730117e-05
4 2.61076625730117e-05
5 2.41083870426432e-05
5 2.41083870426432e-05
5 2.41083870426432e-05
5 2.41083870426432e-05
6 2.64639061140679e-05
6 2.64639061140679e-05
7 2.90101793765922e-05
7 2.90101793765922e-05
7 2.90101793765922e-05
8 2.89406631126936e-05
8 2.89406631126936e-05
9 2.08559997293285e-05
10 1.57004495738957e-05
10 1.57004495738957e-05
11 1.95559041852852e-05
11 1.95559041852852e-05
12 2.86372484620198e-05
13 2.89119162135215e-05
14 3.26379117207063e-05
14 3.26379117207063e-05
15 2.63137663260722e-05
16 3.48686122350857e-05
17 4.38405116590276e-05
18 4.00852766538874e-05
19 3.61882561085428e-05
21 3.93789678798369e-05
22 4.7003988507573e-05
23 5.02916117672305e-05
25 5.13824235355406e-05
26 3.74513307653697e-05
28 4.5361680707133e-05
29 5.09624069286767e-05
31 4.587004778716e-05
33 3.00762830277447e-05
35 3.8566047506372e-05
37 3.40883087287957e-05
39 3.79627927911137e-05
42 5.42553223571961e-05
44 4.79352477706654e-05
47 4.8435909087948e-05
50 5.58819662607883e-05
53 5.30997639770538e-05
56 5.53150862753927e-05
59 5.1249154250732e-05
63 5.69395857282477e-05
66 6.88646212159112e-05
70 7.76946501423326e-05
74 6.77850189651582e-05
79 6.20606092308115e-05
84 6.46892619954408e-05
89 7.03789804511501e-05
94 8.09716293730244e-05
100 8.52929655581854e-05
105 5.73310162182413e-05
112 6.84541095788718e-05
118 5.73534689485858e-05
125 6.10554695529121e-05
133 7.53984728049531e-05
141 6.61416334435851e-05
149 4.69006532930648e-05
158 7.18042389160654e-05
167 7.2787807344814e-05
177 8.65921082831361e-05
188 8.80681312549397e-05
199 8.90142676197926e-05
211 7.6402809610079e-05
223 6.2806515511193e-05
237 5.31640983162888e-05
251 6.424563918488e-05
266 6.19586942014432e-05
281 7.17147503471542e-05
298 8.07264979977874e-05
316 7.83125095931514e-05
334 7.53710245824732e-05
354 8.95403807556707e-05
375 0.000109935130869985
398 8.53801091548622e-05
421 9.30173185990355e-05
446 7.1368113458893e-05
473 7.21845282066071e-05
501 6.3616493310274e-05
530 6.98774206523192e-05
562 6.6227624246484e-05
595 8.15494745864527e-05
630 7.778915969206e-05
668 7.39272361067107e-05
707 6.88022585404152e-05
749 4.65512256662291e-05
794 6.19189595048376e-05
841 6.95085759427662e-05
891 9.47211269368897e-05
944 8.40304340562655e-05
1000 8.21257721501231e-05
1059 5.94468588858342e-05
1122 6.19029933581512e-05
1188 6.44757696987399e-05
1258 7.88864186860153e-05
1333 8.19787132785185e-05
1412 8.51675336948861e-05
1496 8.37945302347167e-05
1584 7.44198389719365e-05
1678 7.21780700089417e-05
1778 7.06062862906765e-05
1883 8.68122495514997e-05
1995 6.99806166601928e-05
2113 8.51257339588003e-05
2238 8.66117359987683e-05
2371 7.74095032758433e-05
2511 5.7468897972791e-05
2660 6.53759239741618e-05
2818 6.31353258274342e-05
2985 7.80936850926616e-05
3162 8.75628682898561e-05
3349 9.09933009666477e-05
3548 8.24892880902397e-05
3758 8.65078088678996e-05
3981 9.69626876160306e-05
4216 9.4660528197037e-05
4466 7.17107813455708e-05
4731 8.5635584962407e-05
5011 7.6570912363726e-05
5308 6.12884523417278e-05
5623 8.51624326001414e-05
5956 7.4536471318395e-05
6309 7.78253133525957e-05
6683 7.5423386342227e-05
7079 7.67402127807108e-05
7498 7.87401473525488e-05
7943 7.0278850054735e-05
8413 7.15923285855967e-05
8912 7.40580206538961e-05
9440 5.61485062010192e-05
10000 6.63725585993998e-05
10592 5.73103980389105e-05
11220 8.18404668135449e-05
11885 7.35737732547535e-05
12589 8.14372300922747e-05
13335 8.14582224364788e-05
14125 7.63175857943893e-05
14962 6.58586883463008e-05
15848 7.5766878310361e-05
16788 9.4087930923956e-05
17782 0.000103068064224756
18836 0.00010637127488111
19952 0.000101381837016424
21134 0.000110245294693907
22387 0.000104958809100787
23713 0.000103378434974927
25118 8.36844563615569e-05
26607 8.21438225395615e-05
28183 7.8605966559978e-05
29853 9.02279519708981e-05
31622 8.55910073256362e-05
33496 8.59703763437947e-05
35481 0.000100391803661486
37583 9.2776656504188e-05
39810 0.000100908743939963
42169 8.74558592736745e-05
44668 9.09561602827415e-05
47315 7.09681650898841e-05
50118 9.3332990653345e-05
53088 0.000122640632447789
56234 0.00012300217589814
59566 0.000132406785999443
63095 0.000119979980758971
66834 0.000120110252332361
70794 0.000102503389241435
74989 0.000113399563736503
79432 0.000115700591422738
84139 0.000102000285186954
89125 9.62001829575165e-05
94406 9.82632300247144e-05
};
\addlegendentry{Euclidean regularizer ($\legexp = 0$)}
\addplot [thick, blue]
table {%
1 0.00949364717979203
1 0.00949364717979203
1 0.00949364717979203
1 0.00949364717979203
1 0.00949364717979203
1 0.00949364717979203
1 0.00949364717979203
1 0.00949364717979203
1 0.00949364717979203
1 0.00949364717979203
1 0.00949364717979203
1 0.00949364717979203
1 0.00949364717979203
2 0.013010073132428
2 0.013010073132428
2 0.013010073132428
2 0.013010073132428
2 0.013010073132428
2 0.013010073132428
2 0.013010073132428
3 0.015539519364639
3 0.015539519364639
3 0.015539519364639
3 0.015539519364639
3 0.015539519364639
4 0.017504702177787
4 0.017504702177787
4 0.017504702177787
5 0.019332893473753
5 0.019332893473753
5 0.019332893473753
5 0.019332893473753
6 0.0207994085301887
6 0.0207994085301887
7 0.0220951269045152
7 0.0220951269045152
7 0.0220951269045152
8 0.0232755360383837
8 0.0232755360383837
9 0.0242811796874458
10 0.025312472396317
10 0.025312472396317
11 0.0262323307712098
11 0.0262323307712098
12 0.0270809280140293
13 0.0279606630001029
14 0.028737895310281
14 0.028737895310281
15 0.0295109896762913
16 0.0301810938397567
17 0.0308296141448705
18 0.0314802838999821
19 0.0320740076399506
21 0.0332342866353852
22 0.0337803557720277
23 0.0342307754315078
25 0.0352527325224678
26 0.035691027957629
28 0.03657449064607
29 0.0370043377352674
31 0.0377631456127774
33 0.0384348868131224
35 0.0391793932187933
37 0.0398352156130884
39 0.0404228266580662
42 0.0412802889631974
44 0.0418658211740561
47 0.0427095849941706
50 0.0434863296216593
53 0.0442174166725689
56 0.0449067088797454
59 0.0454888241807833
63 0.046334956427351
66 0.0469069643094846
70 0.047673832163923
74 0.0483452348854347
79 0.0491051242437116
84 0.0498058714214638
89 0.0505810263852248
94 0.0513125829731154
100 0.0520655177881013
105 0.0526854996705545
112 0.0534427867415033
118 0.0540469201185351
125 0.0548901817875606
133 0.0556666619757196
141 0.0563040432183825
149 0.0570439483202766
158 0.0576813233124435
167 0.0583101691959444
177 0.0588414745200207
188 0.0595981527598375
199 0.0602082144374267
211 0.0609032418042614
223 0.0615528811778539
237 0.0622060495260107
251 0.0628976550845552
266 0.0635913705930354
281 0.0643240554182914
298 0.0651289315987758
316 0.0656478716763875
334 0.0663330060680173
354 0.0669080246834717
375 0.0675374596032835
398 0.0682013711160211
421 0.0688719284578055
446 0.0695347313124883
473 0.0700836272337969
501 0.0706268166848824
530 0.0712569594847124
562 0.0717985955309865
595 0.0724200547168592
630 0.0730762487468093
668 0.0736400992090617
707 0.0741520482779778
749 0.074640368974705
794 0.0754498354090603
841 0.0761068956185749
891 0.0766426244760272
944 0.0772987055202767
1000 0.0778206582289785
1059 0.0782943669202333
1122 0.0786823250800815
1188 0.0791687592197798
1258 0.0795065716692657
1333 0.0799147473114671
1412 0.08053601729295
1496 0.0808434065534613
1584 0.0813156395284692
1678 0.0815322819321224
1778 0.0819826059737844
1883 0.0825364021306928
1995 0.0829450249399822
2113 0.0835260993093123
2238 0.0838649820367039
2371 0.0843271782657422
2511 0.0848292044149853
2660 0.0852070190054213
2818 0.0855018686772264
2985 0.0858590608898089
3162 0.0861245227247469
3349 0.0865038846423467
3548 0.0868563452270763
3758 0.0871812169091757
3981 0.0875187903076496
4216 0.0879080456954362
4466 0.0883444633676804
4731 0.0884255923642077
5011 0.0887694694387828
5308 0.0890031653866108
5623 0.0891678074900101
5956 0.0895460873824236
6309 0.0897451561066711
6683 0.0898990537173568
7079 0.089957706537906
7498 0.090283848162182
7943 0.0905646640644547
8413 0.0908528158370646
8912 0.0910154174199851
9440 0.0912347284011147
10000 0.0914842977267965
10592 0.0916651379353656
11220 0.0918719514354562
11885 0.0921215799115327
12589 0.092424702782876
13335 0.0924647543774799
14125 0.0924464598245073
14962 0.0927925839849614
15848 0.0928685515791769
16788 0.0929506014188734
17782 0.0931309193975841
18836 0.0933897911150965
19952 0.0934782602427192
21134 0.0935288401443845
22387 0.0936416435191167
23713 0.0937747614234059
25118 0.0938723625793276
26607 0.0938823061483153
28183 0.0939918610776977
29853 0.0941726849859081
31622 0.0943534977099964
33496 0.0942118198353317
35481 0.0944483240813315
37583 0.0945253223783151
39810 0.0945784451115414
42169 0.0946423430083249
44668 0.0947015294312023
47315 0.0948066545284297
50118 0.0948575250816652
53088 0.0950604925070692
56234 0.0952064487571033
59566 0.0953382766616214
63095 0.0954924879487958
66834 0.0954234643577863
70794 0.0955905280014674
74989 0.0956550586070027
79432 0.0956242704768748
84139 0.0954753098641073
89125 0.0955188171598884
94406 0.0955332243200696
};
\addlegendentry{Entropy ($\legexp=0.5$)}
\addplot [thick, red]
table {%
1 0.0997293380855863
1 0.0997293380855863
1 0.0997293380855863
1 0.0997293380855863
1 0.0997293380855863
1 0.0997293380855863
1 0.0997293380855863
1 0.0997293380855863
1 0.0997293380855863
1 0.0997293380855863
1 0.0997293380855863
1 0.0997293380855863
1 0.0997293380855863
2 0.111778176454631
2 0.111778176454631
2 0.111778176454631
2 0.111778176454631
2 0.111778176454631
2 0.111778176454631
2 0.111778176454631
3 0.119432011030429
3 0.119432011030429
3 0.119432011030429
3 0.119432011030429
3 0.119432011030429
4 0.125140471902327
4 0.125140471902327
4 0.125140471902327
5 0.129742187605781
5 0.129742187605781
5 0.129742187605781
5 0.129742187605781
6 0.133610701731126
6 0.133610701731126
7 0.136958132402827
7 0.136958132402827
7 0.136958132402827
8 0.139939209589492
8 0.139939209589492
9 0.14260234815581
10 0.145034688280031
10 0.145034688280031
11 0.147250228635216
11 0.147250228635216
12 0.14929449218077
13 0.151209263544802
14 0.152989881414321
14 0.152989881414321
15 0.154666348965756
16 0.156256085162859
17 0.157766150931757
18 0.159172070763538
19 0.160532950584763
21 0.163061427802548
22 0.1642572914377
23 0.165400349297568
25 0.167549751152274
26 0.168577900867093
28 0.170502140984756
29 0.17145124011023
31 0.173211112618999
33 0.174881164458612
35 0.176469261882118
37 0.177985847746422
39 0.179429728243665
42 0.181468985323078
44 0.182744108415766
47 0.18459589351788
50 0.186326141332294
53 0.187965601159425
56 0.189525790239857
59 0.191010839561534
63 0.192904768097123
66 0.19423547211418
70 0.195912014131406
74 0.197504673700437
79 0.199425888091188
84 0.201216905232002
89 0.202923438584239
94 0.204538975091397
100 0.206386381704357
105 0.207828246653063
112 0.20976854226836
118 0.211347862269753
125 0.213087107602689
133 0.214983833556553
141 0.216778116322362
149 0.218478839122396
158 0.220284494283508
167 0.222026281445206
177 0.223845337616766
188 0.225732639256255
199 0.227514826429334
211 0.229344482368695
223 0.231091729640229
237 0.232983692630124
251 0.23479954784785
266 0.236646483091864
281 0.238384494370827
298 0.240243329611186
316 0.242111828119939
334 0.243863231946229
354 0.245699343923855
375 0.247532831043337
398 0.249445344662247
421 0.25124917010868
446 0.253068039341174
473 0.254911821216635
501 0.256774448262855
530 0.258588354617013
562 0.260467567525149
595 0.262286566120603
630 0.26411714302558
668 0.265974464535873
707 0.267806797076529
749 0.269636995523354
794 0.271491266011342
841 0.273310109961854
891 0.2751378792056
944 0.27696182978219
1000 0.278776214447726
1059 0.280597102847765
1122 0.282402512616257
1188 0.284210694622779
1258 0.285952563449047
1333 0.287742048614915
1412 0.28951939211213
1496 0.291263822874219
1584 0.293004441122513
1678 0.294755676445033
1778 0.296522913296922
1883 0.29823696872738
1995 0.29997708008423
2113 0.301698126975501
2238 0.30338938303619
2371 0.305086370438412
2511 0.306758953656925
2660 0.308409457381258
2818 0.310059485559405
2985 0.31171395255225
3162 0.313355636701065
3349 0.314959653641951
3548 0.316575340694794
3758 0.318162911406714
3981 0.319738919131501
4216 0.321267656141782
4466 0.322808467446546
4731 0.324352693683245
5011 0.325836784061256
5308 0.327348127178074
5623 0.328855940997939
5956 0.330335325099683
6309 0.331796375823778
6683 0.333216220898299
7079 0.334600999675597
7498 0.335999194672936
7943 0.33739706730286
8413 0.338747494961819
8912 0.340087492618634
9440 0.341421985160907
10000 0.342747570833337
10592 0.344059415209056
11220 0.345361478785086
11885 0.346610795670685
12589 0.347848726241208
13335 0.349074554017906
14125 0.350288118411072
14962 0.35145616576954
15848 0.35260857666255
16788 0.353757062470542
17782 0.354868347608079
18836 0.355975360436081
19952 0.357057265649111
21134 0.358153067159409
22387 0.359215998182408
23713 0.360242758711348
25118 0.361261812596754
26607 0.362267173608575
28183 0.363252526991161
29853 0.364201226756519
31622 0.365176747871411
33496 0.366071859156231
35481 0.366996510271243
37583 0.367896144865804
39810 0.368761574986507
42169 0.369631713148868
44668 0.370485051849447
47315 0.371311668612965
50118 0.372128998668232
53088 0.372962131622108
56234 0.373734468990262
59566 0.374494908580143
63095 0.375230201892469
66834 0.375985276231327
70794 0.376701707141382
74989 0.377407495498283
79432 0.3781180217964
84139 0.378803230047638
89125 0.37948792287713
94406 0.380142943018677
};
\addlegendentry{Tsallis entropy with $q=0.5$ ($\legexp=0.75$)}

\addplot [thick, magenta]
table {%
1 0.000746940232633
1 0.000746940232633
1 0.000746940232633
1 0.000746940232633
1 0.000746940232633
1 0.000746940232633
1 0.000746940232633
1 0.000746940232633
1 0.000746940232633
1 0.000746940232633
1 0.000746940232633
1 0.000746940232633
1 0.000746940232633
2 0.000787779551893948
2 0.000787779551893948
2 0.000787779551893948
2 0.000787779551893948
2 0.000787779551893948
2 0.000787779551893948
2 0.000787779551893948
3 0.000843895428313574
3 0.000843895428313574
3 0.000843895428313574
3 0.000843895428313574
3 0.000843895428313574
4 0.000891753591972587
4 0.000891753591972587
4 0.000891753591972587
5 0.000879261054224874
5 0.000879261054224874
5 0.000879261054224874
5 0.000879261054224874
6 0.000929599172868062
6 0.000929599172868062
7 0.000936825069188172
7 0.000936825069188172
7 0.000936825069188172
8 0.000944371322775201
8 0.000944371322775201
9 0.000947118319222596
10 0.000898113625800452
10 0.000898113625800452
11 0.000922825756023795
11 0.000922825756023795
12 0.000922855302498585
13 0.000917422908445308
14 0.000925107717507404
14 0.000925107717507404
15 0.000927455649348076
16 0.000901092112876995
17 0.000911732162031363
18 0.000887451273757142
19 0.000899844894909834
21 0.000923072192648293
22 0.000915311460734667
23 0.000940749366062245
25 0.000958150063720607
26 0.000949546161761409
28 0.000944059251967755
29 0.000943225924851242
31 0.000935313611824026
33 0.000923738423186379
35 0.000936142974415092
37 0.000923031119377196
39 0.000918030830134668
42 0.000926334822441325
44 0.000913895376541683
47 0.000939985961555515
50 0.00093970451630517
53 0.0009253129014502
56 0.000965335486371512
59 0.000949708133898835
63 0.000994182125577282
66 0.000993383750735796
70 0.0010038730155991
74 0.000960141681839908
79 0.000955765067838406
84 0.000958967557123417
89 0.000962167590325986
94 0.000945897033536753
100 0.000950343168316678
105 0.000919605308525238
112 0.000927082540428675
118 0.000928577613219848
125 0.000919569870037284
133 0.000927357262441149
141 0.000914597970248153
149 0.000899802504172033
158 0.00091480949992295
167 0.000915479815824318
177 0.000910761460389679
188 0.000912694435821828
199 0.000920440648307046
211 0.000924496979783338
223 0.000906254688753664
237 0.000905861589759454
251 0.00088646688670687
266 0.000871621835251516
281 0.000876677078730492
298 0.00087811686647544
316 0.000886363140273703
334 0.000891848723533724
354 0.000875610552514113
375 0.000865212598114451
398 0.000852651571012174
421 0.000855942504177749
446 0.000861257354720406
473 0.000866077483537257
501 0.000862261076566893
530 0.000833032450070951
562 0.000845798291999229
595 0.00084339253855624
630 0.000843902648759729
668 0.000840330202485173
707 0.000841166012272847
749 0.000796223350150648
794 0.000782324405514736
841 0.000762375685037843
891 0.000768833384644748
944 0.000754119497703585
1000 0.000747171433333954
1059 0.000752134948318411
1122 0.000747849186403034
1188 0.000766675079474181
1258 0.0007478560143751
1333 0.000755669783683586
1412 0.000776860973277636
1496 0.000799814925270166
1584 0.000799105288328016
1678 0.000800839364001114
1778 0.000777468966373292
1883 0.000747997342128349
1995 0.000746236756311902
2113 0.000774082163012276
2238 0.000785534697959937
2371 0.000785348678436545
2511 0.000782572311434689
2660 0.000779230896188479
2818 0.000794636627312153
2985 0.000780388677286475
3162 0.000788922993919366
3349 0.000796725797278248
3548 0.000810883837763543
3758 0.000786425830073959
3981 0.000796249851513242
4216 0.000781950611393268
4466 0.000790835522494476
4731 0.000784040107831836
5011 0.000812645121324825
5308 0.000810993253360628
5623 0.000821264984868443
5956 0.000807446173050374
6309 0.000801286997196446
6683 0.000822339531110391
7079 0.000818630736127647
7498 0.000807859718739832
7943 0.000795012170710088
8413 0.00080021239446312
8912 0.000799254186529508
9440 0.000794698379984023
10000 0.000812256366084502
10592 0.00081963322656664
11220 0.000839153501835105
11885 0.000840367038577303
12589 0.000840950750443761
13335 0.000848737616604556
14125 0.000835964873649379
14962 0.000848514180379945
15848 0.000837456149420256
16788 0.000816755960391248
17782 0.000835689705511815
18836 0.000858969180681503
19952 0.000837716033698781
21134 0.000861103718113892
22387 0.000858418038596109
23713 0.000848825738997035
25118 0.00086421384573368
26607 0.00085890989149737
28183 0.000872177552976128
29853 0.000875362539762935
31622 0.000869560239915598
33496 0.000870660796731294
35481 0.000871037970014191
37583 0.000867029785616571
39810 0.000870680553277581
42169 0.000868997540537753
44668 0.000858686036752858
47315 0.000845325119062683
50118 0.000826780173058265
53088 0.000807561942870507
56234 0.00081850159680131
59566 0.000814811174277715
63095 0.000834827344922018
66834 0.000796122736399757
70794 0.00080162312984846
74989 0.000812060834432681
79432 0.000819563324836724
84139 0.000821882968274612
89125 0.000833472861530456
94406 0.000850953000124537
};
\addlegendentry{Tsallis entropy with $q=1.5$ ($\legexp = 0.25$)}
\end{axis}

\end{tikzpicture}

%% file: Conclusion.tex
%----------------------------------------------------------------------
%%% CONCLUSION
%----------------------------------------------------------------------
% !TEX root = ./Main.tex

In this work, we investigated the rate of convergence of \acl{OMD} for solving variational inequalities with stochastic oracle feedback.
Our results highlight the relationship between the rate of convergence of the last iterate and the local geometry of the \acl{DGF} at the solution. 
To capture this local geometry, we introduced
%a characterization of
the regularity exponent of the divergence relative to the norm that we dubbed the {Legendre} exponent.
This quantity plays a central role in our results:
the less regular the Bregman divergence around the solution, the slower the convergence of the last iterate.
Furthermore, we show that that the step-size policy that guarantees the best rates depends on the 
\acl{DGF} through the associated Legendre exponent.

This work opens the door to various refinements and extensions diving deeper into the geometry of the method's \ac{DGF}.
A key remark is that the method's Legendre exponent seems to depend crucially on which constraints of the problem are active at a given solution.
Deriving a precise characterization of the method's convergence rate in these cases is a fruitful direction for future research.

%% file: AppLemmas.tex
We gather in this section the basic results used in our proofs.

\subsection{Bregman divergences}

We recall here classical lemmas on Bregman divergences; see \citet{JNT11,MLZF+19}. We complement them with two elementary results
\cref{lemma: strg mon extra} on a specific bounding and 
\cref{lem:trivial-beta} which shows that the Legendre exponent is typically trivial on $\proxdom$.

\begin{lemma}%[{\citet{MLZF+19}}]
\label{lemma: strg cvx h}
    For $\point \in \points$, $\base \in \dom \partial \hreg$,
    $$
    \breg({\point,\base}) \geq \half\norm{\base-\point}^2\,.
    $$
\end{lemma}

\begin{lemma}%[{\citet{MLZF+19}}]
\label{lemma: lipschitz prox}
    For $\point \in \dom \partial \hreg$, $\dpoint, \dpointalt  \in \mathcal{Y}$,
    \begin{equation*}
        \| \proxof{\point}{\dpoint} - \proxof{\point}{\dpointalt} \| \leq \| \dpoint - \dpointalt \|_*\,. 
    \end{equation*}
\end{lemma}
\begin{lemma}%[{\citet{MLZF+19}}]
\label{lemma: bregman one step}
    For $\base \in \points$, $\point \in \dom \partial \hreg$, $\dpoint \in \mathcal{Y}$, $\dpointalt \in N_\points(\base)$ and $\point^+ = \proxof{\point}{\dpoint}$,
    \begin{align*}
        \breg(\base, \point^+) &\leq \breg(\base, \point) + \inner{\dpoint - \dpointalt, \point^+ - \base} - \breg(\point^+, \point)\\
                  &\leq \breg(\base, \point) + \inner{\dpoint - \dpointalt, \point - \base} + \half \|\dpoint - \dpointalt\|_*^2\,.
    \end{align*}
\end{lemma}

\begin{lemma}%[{\citet{MLZF+19}}]
\label{lemma: bregman two steps}
    For $\base \in \points$, $\point \in \dom \partial \hreg$, $\dpoint, \dpointalt \in \mathcal{Y}$ and $\point^+_1 =\proxof{\point}{\dpoint}$,$\point^+_2 = \proxof{\point}{\dpointalt}$,
    $$
    \breg(\base, \point^+_2) \leq \breg(\base, \point) + \inner{\dpointalt, \point^+_1 - \base} + \half \|\dpointalt - \dpoint\|_*^2 - \half \|\point_1^+ - \point\|^2\,.
    $$
\end{lemma}

% The next lemma was initially proved in \citet{JNT11} but we adopt the presentation of \citet{MLZF+19}.

\begin{lemma}\label{lemma: characterization prox}
    For $\point \in \proxdom$,
    \begin{equation*}
        \subd \hreg(\point) = \nabla \hreg(\point) + \ncone_\points(\point)\,.
    \end{equation*}
    As a consequence, $\point \in \proxdom$, $\point^+ \in \points$, $\dpoint \in \mathcal{Y}$,
    \begin{align*}
        \point^+ = \proxof{\point}{\dpoint} &\iff \nabla \hreg (\point) + \dpoint \in \partial \hreg(\point^+)\\
                     &\iff  \nabla \hreg (\point) + \dpoint - \nabla h(\point^+) \in \ncone_\points(\point^+)\,.
    \end{align*}
    Moreover, $\point^+ =\proxof{\point}{\dpoint}$ implies that $\point^+ \in \proxdom \hreg$.

\end{lemma}

\begin{lemma}\label{lemma: strg mon extra}
    If \cref{asm:strong} holds, $\new \in \basin$, and $\point \in \points$, then,

    \begin{equation*}
        \inner{\vecfield(\new), \new - \sol} \geq \frac{\strong}{2}\|\point - \sol\|^2 - \strong \|\new - \point\|^2\,.
    \end{equation*}
\end{lemma}

\begin{proof}
If $\new \in \basin$, then,
%\begin{equation*}
% \inner{\vecfield(\new) - \vecfield(\sol), \new - \sol} \geq \strong \|\new - \sol\|^2\,. \end{equation*}
%By assumption, we have that $\inner{-\vecfield(\sol), \new - \sol} \leq 0$, which implies that,
\begin{equation}\label{eq: proof lemma strg mon extra}
    \inner{\vecfield(\new), \new - \sol} \geq \mu \|\new - \sol\|^2\,.
\end{equation}
However, we are interested in the distance between $\point$ and $\sol$, not in the distance between $\new$ and $\sol$. To remedy this, we use Young's inequality,
\begin{equation*}
\|\point - \sol\|^2 \leq 2\|\point - \new\|^2 + 2\|\new - \sol\|^2\,,
\end{equation*}
which leads to,
\begin{equation*}
\|\new - \sol\|^2 \geq \half\|\point - \sol\|^2 - \|\new - \point\|^2\,.
\end{equation*}
Combined with \cref{eq: proof lemma strg mon extra}, we get,
\begin{equation*}
    \inner{\vecfield(\new), \new - \sol} \geq \frac{\mu}{2} \|\point - \sol\|^2 - \mu \|\new - \point\|^2\,.
\end{equation*}
\end{proof}

%\subsection{Legendre exponent}

% We present a simple result, which shows that the Legendre exponent is typically trivial on $\proxdom$.

\begin{lemma}[Trivial Legendre exponent on the interior]\label{lem:trivial-beta}
    Assume that $\nabla \hreg$ %, the selection of subgradients $\hreg$,
    is locally Lipschitz continuous on\;$\proxdom$. Then, for any $\base \in \proxdom$, the Legendre exponent of $\hreg$ at $\base$ is $\legexp = 0$.
\end{lemma}
\begin{proof}
    Take $\base \in \proxdom$. As $\nabla \hreg$ is assumed to be locally Lipschitz, there exists $\legnhd$ a neighborhood of $\base$ and $\legconst > 0$ such that,
    for any $\point \in \legnhd \cap \proxdom$,
    \begin{equation*}
        \dnorm{\nabla \hreg(\base) - \nabla \hreg(\point)} \leq \frac{\legconst}{2}\norm{\base - \point}\,.
    \end{equation*}
    
    Now, as $\nabla \hreg(\base) \in \subd \hreg(\base)$, for any $\point \in \legnhd \cap \proxdom$,
    \begin{align*}
        \breg(\base, \point) &= \hreg(\base) - \hreg(\point) - \inner{\nabla \hreg(\point), \base - \point}\\
        &\leq \inner{\nabla \hreg(\base) - \nabla \hreg(\point), \base - \point}\\
        &\leq \dnorm{\nabla \hreg(\base) - \nabla \hreg(\point)} \norm{\base - \point}\,.
    \end{align*}
    Using the local Lipschitz continuity of $\nabla \hreg$, we get, for any $\point \in \legnhd \cap \proxdom$,
    \begin{equation*}
        \breg(\base, \point)  \leq \frac{\legconst}{2}\norm{\base - \point}^2\,,
    \end{equation*}
    which reads $\beta=0$.
\end{proof}

\subsection{Sequences}

We recall results on sequences to transform descent-like inequalities into rates; see the classical textbook\;\cite{Pol87}. We also need variants of existing results (\Cref{lemma: stoch sequence alpha} and \Cref{lemma: stoch sequence alpha bis}) that we state here.

\begin{lemma}[{\citet[Lem.~1]{Chu54}}]\label{lemma: chung 1}
    Let $(\curr[\seq])_{\run \geq \start}$ be a sequence of non-negative scalars, $\cst > 1$, $\cstalt > 0$, $t_0 \geq 0$.
    If, for any $t \geq 1$,
    \begin{equation*}
        a_{t+1} \leq \left(1 - \frac{q}{t+t_0}\right)a_t + \frac{q'}{(t+t_0)^2}\,,
    \end{equation*}
    then, for any $T \geq 1$,
\begin{equation*}
    a_{T} \leq \frac{q'}{q-1}\frac{1}{T+t_0} + o\left(\frac{1}{T}\right)\,.
\end{equation*}
\end{lemma}

\begin{lemma}[{\citet[Lem.~4]{Chu54}}]\label{lemma: chung 4}
    Let $(a_t)_{t \geq 1}$ be a sequence of non-negative scalars, $q, q' > 0$, $t_0 \geq 0$, $1 > \expstep > 0$.
    If, for any $t \geq 1$,
    \begin{equation*}
        a_{t+1} \leq \left(1 - \frac{q}{(t+t_0)^\expstep}\right)a_t + \frac{q'}{(t+t_0)^{2\expstep}}\,,
    \end{equation*}
    then, for any $T \geq 1$,
\begin{equation*}
    a_{T} \leq \bigoh*{\frac{1}{T^\expstep}}\,.
\end{equation*}
\end{lemma}

\begin{lemma}[{\citet[\S 2.2, Lem.~6]{Pol87}}]\label{lemma: sequence}
    Let $(a_t)_{1 \leq t \leq T}$ be a sequence of non-negative scalars, $(\gamma_t)_{1 \leq t \leq T}$  be sequence of positive scalars and $\alpha > 0$.
    If, for any $t = 1,\dots,T$,
    \begin{equation*}
        a_{t+1} \leq a_t - \gamma_t a_t^{1+\alpha}\,,
    \end{equation*}
    then,
\begin{equation*}
    a_{T} \leq \frac{a_1}{\left(1 + \alpha a_1^{\alpha} \sum_{t=1}^{T-1} \gamma_t\right)^{1/\alpha}} \,.
\end{equation*}
\end{lemma}

%\begin{remark}
    Though the next two lemmas allow for step-sizes of the form $\gamma_t = \frac{\gamma}{(t+t_0)^\expstep}$, they do not exploit the term $t_0$.
%\end{remark}

\begin{lemma}\label{lemma: stoch sequence alpha}
    Let $(a_t)_{t \geq 1}$ be a sequence of non-negative scalars, $(\gamma_t)_{t \geq 1}$  be sequence of positive scalars of the form $\gamma_t = \frac{q}{(t+t_0)^{\expstep}}$ with $q, \expstep > 0$, $t_0 \geq 0$ and $\alpha > 0$, $q' > 0$ such that,
    \begin{equation*}
        a_{t+1} \leq a_t - \gamma_t a_t^{1+\alpha} + \frac{q'}{(t+t_0)^{2\expstep}}\,.
    \end{equation*}
    If,
    \begin{align*}
        1 \geq &\expstep \geq \frac{1+\alpha}{1+2\alpha}>\half\,, & q'q^{1/\alpha} &\leq c(\expstep, \alpha) \coloneqq \alpha \left(\frac{1-2^{-(1+\alpha)}}{1+\alpha}\right)^{1+\frac{1}{\alpha}}\,,
    \end{align*}
    then, for any $T \geq 1$,
\begin{equation*}
    a_{T} \leq \frac{a_1+b}{\left(1 + \alpha (a_1 + b)^{\alpha}2^{-\alpha} \sum_{t=1}^{T-1} \gamma_t\right)^{1/\alpha}} \,,
\end{equation*}
where $b = \left(\frac{1 - 2^{1-2\expstep}}{(1+\alpha)q}\right)^{\frac{1}{\alpha}}$.
\end{lemma}
\begin{proof}
    First, define $p=1+\alpha > 1$ and note that $\frac{1+\alpha}{1+2\alpha} = \frac{p}{2p-1} = \frac{1}{2 - p^{-1}} \in (\half, 1)$ so, in particular, the condition on $\expstep$ is not absurd.
    Moreover, this means that $\beta \coloneqq 2\expstep - 1$ belongs to $(0, 1]$. Then, we have $b = \left(\frac{1 - 2^{-\beta}}{pq}\right)^{\frac{1}{\alpha}} > 0$ and, define, for $t \geq 1$, $b_t \coloneqq \frac{b}{(t+t_0)^\beta}$.

    The first part of the proof consists in showing, that, for $t \geq 1$,
    \begin{equation*}
        \frac{q'}{(t+t_0)^{2\expstep}} \leq b_t - b_{t+1} - \gamma_t b_t^p\,.
    \end{equation*}
    For this, we will need the following remark. As $\beta \leq 1$, $x \mapsto (1+x)^\beta$ is concave. Hence, it is above its chords, and in particular above its chord going from $0$ to $1$. Thus, for $0 \leq x \leq 1$, $(1+x)^\beta \geq 1 + x(2^\beta - 1)$.

    We use this remark to lower bound $b_t - b_{t+1}$ for $t \geq 1$. Indeed,
    \begin{align*}
        b_t - b_{t+1} &= \frac{b}{(t+t_0)^\beta} - \frac{b}{(t+1+t_0)^\beta}\\
                      &= \frac{b}{(t+1+t_0)^\beta}\left(\left(1+\frac{1}{t+t_0}\right)^{\beta}-1\right) \\  
                      &\geq \frac{b}{(t+1+t_0)^\beta}\frac{2^\beta -1}{t+t_0}\\ 
                      &\geq \frac{b}{(t+t_0)^{\beta+1}}\frac{2^\beta -1}{2^\beta}\,. 
    \end{align*}
    Therefore,
    \begin{equation*}
        b_t - b_{t+1} - \gamma_t b_t^p \geq \frac{b}{(t+t_0)^{\beta+1}}(1-2^{-\beta}) - \frac{qb^p}{(t+t_0)^{\expstep+p\beta}}\,.
    \end{equation*}
    Now, by the definition of $\beta$, $\beta + 1 = 2\expstep$ and $\expstep + p\beta = (2p+1)\expstep -p = 2\expstep + (2p-1)\expstep - p \geq 2\expstep$ by the assumption that $\expstep \geq \frac{1+\alpha}{1+2\alpha} = \frac{p}{2p-1}$. Hence,

    \begin{equation*}
        b_t - b_{t+1} - \gamma_t b_t^p \geq \frac{1}{(t+t_0)^{2\expstep}}(b(1-2^{-\beta}) - qb^p)\,,
    \end{equation*}
    so that we only need to show that $q' \leq b(1-2^{-\beta})-qb^p$.

    %(Actually, our choice of $b$ maximizes the left-hand side, which is concave in $b$.)

    Rearranging and replacing $b$ by its expression gives,
    \begin{align*}
        b(1-2^{-\beta})-qb^p &= b((1-2^{-\beta}) - qb^{\alpha})\\
                             &= b\left((1-2^{-\beta}) - \frac{1-2^{-\beta}}{p}\right)\\
                             &= b(1-2^{-\beta})\frac{p-1}{p}\\
                             &= (1-2^{-\beta})^{1+\frac{1}{\alpha}}\frac{p-1}{p^{1+\frac{1}{\alpha}}q^\frac{1}{\alpha}}\,.
    \end{align*}
    Therefore, with $c(\expstep, \alpha) =\alpha \left(\frac{1-2^{-(1+\alpha)}}{1+\alpha}\right)^{1+\frac{1}{\alpha}} = (1-2^{-\beta})^{1+\frac{1}{\alpha}}\frac{p-1}{p^{1+\frac{1}{\alpha}}} > 0$ and $q^\frac{1}{\alpha}q' \leq c(\expstep, \alpha)$, we finally get that $q' \leq b(1-2^{-\beta})-qb^p$ and, for $t\geq 1$,
    \begin{equation*}
        \frac{q'}{(t+t_0)^{2\expstep}} \leq b_t - b_{t+1} - \gamma_t b_t^p\,.
    \end{equation*}
    Recall that $(a_t)_{t \geq 1}$ satisfies, for $t \geq 1$,
    \begin{equation*}
        a_{t+1} \leq a_t - \gamma_t a_t^{p} + \frac{q'}{(t+t_0)^{2\expstep}}\,.
    \end{equation*}
    Therefore, putting these two inequalities together gives,
    \begin{equation*}
        a_{t+1} +b_{t+1} \leq a_t+b_t - \gamma_t (a_t^p + b_t^p)\,.
    \end{equation*}
    Finally, by convexity of $x \mapsto x^p$,
    \begin{equation*}
        a_{t+1} +b_{t+1} \leq a_t+b_t - \gamma_t 2^{1-p}(a_t + b_t)^p\,.
    \end{equation*}
    Now, we can apply \cref{lemma: sequence} with $a_t \gets a_t+b_t$ and $\gamma_t \gets \gamma_t2^{1-p}$. For any $T\geq 1$,
    \begin{equation*}
        a_T \leq a_T+b_T \leq \frac{1}{\left((a_1+b_1)^{-\alpha} + \alpha 2^{-\alpha} \sum_{t=1}^{T-1} \gamma_t\right)^{1/\alpha}} \,.
    \end{equation*}
    As the right-hand side is non-decreasing in $b_1$ and $b_1 \leq b$, we get the result of the statement.
\end{proof}

\begin{lemma}\label{lemma: stoch sequence alpha bis}
    Let $(a_t)_{t \geq 1}$ be a sequence of non-negative scalars, $(\gamma_t)_{t \geq 1}$  be sequence of positive scalars of the form $\gamma_t = \frac{q}{(t+t_0)^{\expstep}}$ with $q, \expstep > 0$, $b \geq 1$ and $\alpha > 0$, $q' > 0$ such that,
    \begin{equation*}
        a_{t+1} \leq a_t - \gamma_t a_t^{1+\alpha} + \frac{q'}{(t+t_0)^{2\expstep}}\,.
    \end{equation*}
    If,
    \begin{align*}
        1 > \frac{1+\alpha}{1+2\alpha} &\geq \expstep >\half\,, & q'q^{1/\alpha} &\leq c(\expstep, \alpha)\coloneqq\frac{1-\expstep}{(1+\alpha)^{\frac{1}{\alpha}+1}}\,,
    \end{align*}
    then, for any $T \geq 1$,
\begin{equation*}
    a_{T} \leq \frac{a_1}{\left(1 + \alpha a_1^{\alpha} \sum_{t=1}^{T-1} \gamma_t\right)^{1/\alpha}} + \frac{1}{((1+\leg)q)^{\frac{1}{\alpha}}{(T+t_0)}^\expstep}\,.
\end{equation*}
\end{lemma}
\begin{proof}
    This proof is the ``mirror'' of the proof of the previous lemma.
    As before, define $p=1+\alpha > 1$ and note that $\frac{1+\alpha}{1+2\alpha} = \frac{p}{2p-1} = \frac{1}{2 - p^{-1}} \in (\half, 1)$ so the condition on $\expstep$ is not absurd.
    Moreover, this means that $\beta \coloneqq \frac{\expstep}{p}$ belongs to $(0, 1)$. Then, define $b  \coloneqq \left(\frac{1}{pq}\right)^{\frac{1}{\alpha}} > 0$ and, define, for $t \geq 1$, $b_t \coloneqq \frac{b}{(t+t_0)^\beta}$.

    Opposite to the proof of the previous lemma, the first part of the proof consists in showing, that, for $t \geq 1$,
    \begin{equation*}
        \frac{q'}{(t+t_0)^{2\expstep}} \leq b_{t+1} + \gamma_t b_t^p - b_t\,.
    \end{equation*}
    For this, we use the concavity of $x \mapsto (1+x)^\beta$, as $\beta \leq 1$, so that, for $x \geq 0$, $(1+x)^\beta \leq 1 + \beta x$.

    This remark enables us to upper bound $b_t - b_{t+1}$ for $t \geq 1$. Indeed,
    \begin{align*}
        b_t - b_{t+1} &= \frac{b}{(t+t_0)^\beta} - \frac{b}{(t+1+t_0)^\beta}\\
                      &= \frac{b}{(t+1+t_0)^\beta}\left(\left(1+\frac{1}{t+t_0}\right)^{\beta}-1\right) \\  
                      &\leq \frac{b}{(t+1+t_0)^\beta}\frac{\beta}{t+t_0}\\ 
                      &\leq \frac{\beta b}{(t+t_0)^{\beta+1}}\,. 
    \end{align*}
    Therefore,
    \begin{equation*}
        b_t - b_{t+1} - \gamma_t b_t^p \leq \frac{\beta b}{(t+t_0)^{\beta+1}}  - \frac{qb^p}{(t+t_0)^{\expstep+p\beta}}\,.
    \end{equation*}
    Now, by the definition of $\beta = \frac{\expstep}{p}$, $\expstep +p\beta = 2\expstep$ and $\beta + 1 = 2\expstep + \frac{\expstep}{p} + 1 -2\expstep = 2\expstep + \frac{p- (2p-1)\expstep}{p} \geq 2\expstep$ by the assumption that $\expstep \leq \frac{1+\alpha}{1+2\alpha} = \frac{p}{2p-1}$. Hence,

    \begin{equation*}
        b_t - b_{t+1} - \gamma_t b_t^p \leq \frac{1}{(t+t_0)^{2\expstep}}(\beta b - qb^p)\,,
    \end{equation*}
    so that,
    \begin{equation*}
       b_{t+1} + \gamma_t b_t^p - b_t \geq \frac{1}{(t+t_0)^{2\expstep}}(qb^p -\beta b)\,.
    \end{equation*}
    Again, we only need to show that $q' \leq qb^p - \beta b$.

    Rearranging and replacing $b$ by its expression gives,
    \begin{align*}
        qb^p - b\beta &= b(qb^{\alpha} - \beta)\\
                      &= b\left(\frac{1}{p} - \beta\right)\\
                      &= \frac{1}{q^\frac{1}{\alpha}p^\frac{1}{\alpha}}\left(\frac{1}{p} - \beta\right)\,.
    \end{align*}
    Therefore, with $c(\expstep, \alpha) =\frac{1}{p^\frac{1}{\alpha}}\left(\frac{1}{p} - \beta\right) > 0$ and $q^\frac{1}{\alpha}q' \leq c(\expstep, \alpha)$, we finally get that $q' \leq qb^p - \beta b$ and, for $t\geq 1$,
    \begin{equation*}
        \frac{q'}{(t+t_0)^{2\expstep}} \leq b_{t+1} + \gamma_t b_t^p - b_t\,.
    \end{equation*}
    Putting this inequality together with the one on the sequence $(a_t)_{t\geq1}$ gives,
    \begin{equation*}
        a_{t+1} - b_{t+1} \leq a_t-b_t - \gamma_t (a_t^p - b_t^p)\,.
    \end{equation*}
    Now, let us discuss separately the case when $a_t > b_t$ and when $a_t \leq b_t$.
    
    More precisely, define $T_0 = \min\{t\geq 1: a_t \leq b_t\} \in \N^* \cup \{+\infty\}$, so that, for any $1 \leq t < T_0$, $a_t > b_t > 0$.
    Note that, for $x, y > 0$, as $p > 1$, $x^p + y^p \leq (x+y)^p$. For $1 \leq t < T_0$, apply this with $x \gets a_t - b_t$, $y \gets b_t$ gives $(a_t - b_t)^p \leq a_t^p - b_t^p$ so that,
    \begin{equation*}
        a_{t+1} - b_{t+1} \leq a_t-b_t - \gamma_t (a_t - b_t)^p\,.
    \end{equation*}
    \cref{lemma: sequence} with $a_t \gets a_t - b_t$ gives, for $1 \leq T < T_0$, 
    \begin{align*}
        a_T &\leq b_T + \frac{1}{((a_1 - b_1)^{-\alpha} + \alpha\sum_{t=1}^{T-1}\gamma_t)^{\frac{1}{\alpha}}}\\
            &\leq b_T + \frac{1}{(a_1^{-\alpha} + \alpha\sum_{t=1}^{T-1}\gamma_t)^{\frac{1}{\alpha}}}\,,
    \end{align*}
    so the statement holds in this case.

    We now handle the case of $T \geq T_0$. At $t=T_0$, we have $a_{T_0} \leq b_{T_0}$. We show, by induction, that, for any $t \geq T_0$, $a_t \leq b_t$. The initialization at $t = T_0$ is trivial. So, now, assume that $a_t \leq b_t$ for some $t \geq T_0$. 
    Recall that we have,
    \begin{equation*}
        a_{t+1} - b_{t+1} \leq a_t-b_t - \gamma_t (a_t^p - b_t^p)\,,
    \end{equation*}
    so we only need to show that,
    \begin{equation*}
        a_{t} - \gamma_ta_t^p \leq b_t - \gamma_t b_t^p\,.
    \end{equation*}
    Define $\varphi_t: x \in \R_+ \mapsto x - \gamma_t x^p$. Differentiating this function shows that it is non-decreasing on $\{x \in \R_+: p\gamma_t x^{\alpha} \leq 1\} = [0, (\gamma_t p)^{-\frac{1}{\alpha}}]$.
    But $a_t \leq b_t \leq b_1 \leq b$ by definition of $(b_s)_{s \geq 1}$ and, by definition of $b$ and $(\gamma_s)_{s\geq1}$, $b = (pq)^{-\frac{1}{\alpha}} = (\gamma_1 p)^{-\frac{1}{\alpha}} \leq (\gamma_t p)^{-\frac{1}{\alpha}}$.
    %(Actually, this is the reason behind the defintion of $b$.)
    
    Hence, both $a_t$ and $b_t$ belong to the interval on which $\varphi_t$ is non-decreasing so that $\varphi_t(a_t) \leq \varphi_t(b_t)$, which implies that $a_{t+1} \leq b_{t+1}$.
    Therefore, we have shown that, for all $T \geq T_0$, $a_T \leq b_T = \frac{1}{(pq)^{\frac{1}{\alpha}}{(T+t_0)}^\expstep}$ which concludes the proof.
\end{proof}

%% file: AppOMDSto.tex
\subsection{Descent inequality}

The lemmas we present here link the divergence between two consecutive iterates and the solution. 

In particular, the first one only rely on i) the definition of one iteration of \acl{OMD}; and ii) the core properties of the Bregman regularizer \ie the $1$-strong convexity of $\hreg$ (through \cref{lemma: strg cvx h}) and the non-expansivity of the proximal mapping (\cref{lemma: lipschitz prox}). It does not require any assumption on the vector field $\vecfield$, on the noise in $\signal$, or additional properties of the Bregman divergence (such as a Legendre exponent).

\begin{lemma}\label{lemma: past mirror prox descent ineq}
If the stepsizes verify  
$
    \strong \curr[\step] + 4\lips^2\curr[\step]^2 \leq 1\,,
$
then, the iterates of \acl{OMD} initialized with  $\initlead[\state] \in \points$, $\init[\state] \in \dom \partial h$ verify for all $\run\geq 1$,
\begin{align*}
    \breg(\sol,\next)+ \phi_{\run+1}& \leq \breg(\sol,\curr) + (1 - \curr[\step] \strong)\phi_{\run} \\
    & ~~~~~ - \curr[\step] \inner{\lead[\signal], \lead - \sol} + \left(2\curr[\step]^2\lips^2 - \half\right)\|\lead - \curr\|^2 \\
    & ~~~~~ +  \curr[\step]^2\Delta_{\run+1}^2 - \curr[\step]^2 \tau^2\|\curr - \beforelead\|^2\,,
\end{align*}
where $\Delta_{\run+1}^2 = (\|\lead[\signal] - \beforelead[\signal]\|^2_* - \lips^2\|\lead - \beforelead\|^2)_+$,
$
     \phi_{\run} = \frac{\prev[\step]^2}{2}\|\beforelead[\signal] - \beforebeforelead[\signal] \|^2$ for all $\run \geq 2
$, 
and $\phi_1 = \half\|\init[\state] - \initlead[\state]\|^2$.
\end{lemma}
\begin{proof}
    First, apply \cref{lemma: bregman two steps} with $(x, y_1, y_2) \gets (\curr, -\curr[\step] \beforelead[\signal], -\curr[\step] \lead[\signal])$ and $p \gets \sol$,
    \begin{align}\label{eq: proof past descent ineq first step}
        \breg(\sol,\next)\leq \breg(\sol,\curr) - \curr[\step] \inner{\lead[\signal], \lead - \sol} + \frac{\curr[\step]^2}{2}\|\beforelead[\signal] - \lead[\signal]\| - \half\|\lead - \curr\|^2\,.
    \end{align}

    In order to upper-bound $\frac{\curr[\step]^2}{2}\|\beforelead[\signal] - \lead[\signal]\| - \half\|\lead - \curr\|^2$, we use the definition of $\phi_{\run+1} = \frac{\curr[\step]^2}{2}\|\lead[\signal] - \beforelead[\signal] \|^2$ to see that
\begin{align*}
    \frac{\curr[\step]^2}{2}\|\beforelead[\signal] - \lead[\signal]\|^2 &= {\curr[\step]^2}\|\beforelead[\signal] - \lead[\signal]\|^2 - \frac{\curr[\step]^2}{2}\|\beforelead[\signal] - \lead[\signal]\|^2\\
                                                      &= {\curr[\step]^2}\|\beforelead[\signal] - \lead[\signal]\|^2 - \phi_{\run+1}\,,
\end{align*}
    and the definition of $\Delta_{\run+1}$ to bound 
    \begin{align*}
   \curr[\step]^2\|\beforelead[\signal] -  \lead[\signal]\|^2
        &\leq \curr[\step]^2 \lips^2\|\beforelead - \lead\|^2 + \curr[\step]^2 \Delta_{\run+1}^2\\
        &\leq  2{\curr[\step]^2 \lips^2}\|\beforelead - \curr\|^2 + 2\curr[\step]^2 \lips^2\|\lead - \curr\|^2 + \curr[\step]^2 \Delta_{\run+1}^2\,,
    \end{align*}
    where we used Young's inequality to get the last inequality line. Putting the two equations together, we get that
\begin{align}
  \nonumber  &\frac{\curr[\step]^2}{2}\|\beforelead[\signal] - \lead[\signal]\|^2 - \half\|\lead - \curr\|^2\\ 
   \leq & {2\curr[\step]^2 \lips^2}\|\beforelead - \curr\|^2 + \left({2\curr[\step]^2 \lips^2} - \half\right)\|\lead - \curr\|^2  - \phi_{\run+1} + \curr[\step]^2\Delta_{\run+1}^2 . \label{eq: proof past descent ineq bound additional term}
\end{align}

For $\run \geq 2$, we use the non-expansivity of the proximal mapping (\cref{lemma: lipschitz prox}) to bound $\|\beforelead - \curr\|$ by $\prev[\step]\|\beforebeforelead[\signal] - \beforelead[\signal]\|_*$ in \eqref{eq: proof past descent ineq bound additional term} to get
\begin{align*}
    &\frac{\curr[\step]^2}{2}\|\beforelead[\signal] - \lead[\signal]\|^2 - \half\|\lead - \curr\|^2\\ 
    &\leq {(2\curr[\step]^2 \lips^2 )}\prev[\step]^2\|\beforebeforelead[\signal] - \beforelead[\signal]\|_*^2 + \left({2\curr[\step]^2 \lips^2} - \half\right)\|\lead - \curr\|^2  - \phi_{\run+1} + \curr[\step]^2\Delta_{\run+1}^2\\
    &= {(4\curr[\step]^2 \lips^2 )}\phi_{\run} + \left({2\curr[\step]^2 \lips^2} - \half\right)\|\lead - \curr\|^2  - \phi_{\run+1} + \curr[\step]^2\Delta_{\run+1}^2\\
    &\leq \left(1 - \strong \curr[\step] \right) \phi_{\run} + \left({2\curr[\step]^2 \lips^2} - \half\right)\|\lead - \curr\|^2  - \phi_{\run+1} + \curr[\step]^2\Delta_{\run+1}^2\,,
\end{align*}
where we used the definition of $\phi_{\run}$ and the assumption on the stepsizes. Combining this result with \cref{eq: proof past descent ineq first step} gives the first assertion of the lemma for $\run \geq 2$.

Now, if $\run=1$, \cref{eq: proof past descent ineq bound additional term} can be rewritten as,
\begin{align*}
    &\frac{\init[\step]^2}{2}\|\initlead[\signal] - \afterinitlead[\signal] \|^2 - \half\|\afterinitlead[\state] - \init[\state]\|^2\\ 
    \leq  & 2\init[\step]^2 \lips^2 \|\initlead[\state] - \init[\state]\|^2 + \left( 2\init[\step]^2 \lips^2 - \half\right)\|\afterinitlead[\state] - \init[\state]\|^2  - \phi_{2} + \init[\step]^2\Delta_2^2\\
    = & 4\init[\step]^2\lips^2 \phi_1 + \left(2\init[\step]^2 \lips^2  - \half\right)\|\afterinitlead[\state] - \init[\state]\|^2  - \phi_{2} + \init[\step]^2\Delta_2^2\,,
\end{align*}
which yields the assertion of the lemma for $\run=1$ as $4\init[\step]^2\lips^2 \leq 1 - \init[\step] \strong $.
\end{proof}

We now use the properties of $\signal$ and $\vecfield$ to refine the bound of \cref{lemma: past mirror prox descent ineq} into a descent inequality.

\begin{lemma}\label{lemma:descent}
    Let \cref{asm:Lipschitz,asm:strong} hold. Consider the iterates of \acl{OMD} with:
    \begin{itemize}
        \item[i)] an unbiased stochastic oracle $\curr[\signal]
	= \vecfield(\curr) + \curr[\noise]
$ (see \eqref{eq:signal});
        \item[ii)] step-sizes $0 < \curr[\step] \leq \frac{1}{4\lips}$.
    \end{itemize}
  Then, 
  \begin{align}\label{eq: proof past mirror prox stoch descent ineq}
    \breg(\sol, \next) + \phi_{\run+1}\leq &\breg(\sol, \curr) + (1 - \curr[\step] \strong)\phi_\run - \curr[\step] \inner{\lead[\signal], \lead - \sol}\\ 
 \nonumber   &+ \left(4\curr[\step]^2\lips^2 - \half\right)\|\lead - \curr\|^2 + 4\curr[\step]^2(\|\lead[\noise]\|_*^2 + \|\beforelead[\noise]\|_*^2)\,
\end{align}
where
$
     \phi_{\run} = \frac{\prev[\step]^2}{2}\|\beforelead[\signal] - \beforebeforelead[\signal] \|^2$ for all $\run \geq 2
$, 
and $\phi_1 = \half\|\init[\state] - \initlead[\state]\|^2$.
\end{lemma}

\begin{proof}
  First, let us examine the choice of $\curr[\step]$. The condition $\curr[\step] \leq \frac{1}{4\lips}$ is actually equivalent to,
    \begin{equation*}
        8\curr[\step]^2 \lips^2 +  2\curr[\step] \lips \leq 1\,.
    \end{equation*}
    As $\strong  \leq \lips$, this implies that,
    \begin{equation}\label{eq: proof past last iterate rate stochastic condition gamma}
        8\curr[\step]^2 \lips^2  + 2\curr[\step] \strong \leq 1\,.
    \end{equation}

    We first use the inequality of \ac{OMD} provided by \cref{lemma: past mirror prox descent ineq} with  $(\strong, \lips) \gets (\strong, \sqrt 2 \lips)$ (note that its assumption is satisfied than4s to \cref{eq: proof past last iterate rate stochastic condition gamma}) to get 
\begin{align}
  \label{eq:beforedescent}  \breg(\sol,\next)+ \phi_{\run+1}& \leq \breg(\sol,\curr) + (1 - \curr[\step] \strong)\phi_{\run} \\
 \nonumber   & ~~~~~ - \curr[\step] \inner{\lead[\signal], \lead - \sol} + \left(2\curr[\step]^2\lips^2 - \half\right)\|\lead - \curr\|^2 \\
 \nonumber    & ~~~~~ +  \curr[\step]^2\Delta_{\run+1}^2 - \curr[\step]^2 \tau^2\|\curr - \beforelead\|^2\,,
\end{align}
where $\Delta_{\run+1}^2 = (\|\lead[\signal] - \beforelead[\signal]\|^2_* - \lips^2\|\lead - \beforelead\|^2)_+$,
$
     \phi_{\run} = \frac{\prev[\step]^2}{2}\|\beforelead[\signal] - \beforebeforelead[\signal] \|^2$ for all $\run \geq 2
$, 
and $\phi_1 = \half\|\init[\state] - \initlead[\state]\|^2$.
    
    We now have to bound $\Delta_{t+1}^2 = (\|\lead[\signal] - \beforelead[\signal]\|^2_* - 2\lips^2\|\lead - \beforelead\|^2)_+$ for $\run\geq 1$.
    By using twice Young's inequality,
    \begin{align*}
        \|\lead[\signal] - \beforelead[\signal]\|^2_* &\leq 2\|\vecfield(\lead) - \vecfield(\beforelead)\|_*^2 + 2\|\lead[\noise] - \beforelead[\noise]\|_*^2\\
                                                &\leq 2\|\vecfield(\lead) - \vecfield(\beforelead)\|_*^2 + 4\|\lead[\noise]\|_*^2 +4\|\beforelead[\noise]\|_*^2\,.
    \end{align*}
    Using the Lipschitz continuity of $\vecfield$ (\cref{asm:Lipschitz}), we obtain that 
    \begin{align*}
        \|\lead[\signal] - \beforelead[\signal]\|^2_* &\leq 2\lips^2\|\lead - \beforelead\|_*^2 +4\|\lead[\noise]\|_*^2 + 4\|\beforelead[\noise]\|_*^2\,,
    \end{align*}    
    so that,
    \begin{equation*}
        \Delta_{t+1}^2 \leq 4\|\lead[\noise]\|_*^2 + 4\|\beforelead[\noise]\|_*^2\,.
    \end{equation*}
    Plugging this inequality in \eqref{eq:beforedescent}, we obtain the claimed result.
\end{proof}

\subsection{Stochastic local stability}

In order to use the local properties of the Bregman divergence around the solution, we need a local stability result that gives us that sufficiently close iterates will remain in a neighborhood of the solution with high probability. 
To show such a result, we will use the following lemma \citet[Lem.~F.1]{HIMM20} for bounding a recursive stochastic process.
\begin{lemma}[{\citet[Lem.~F.1]{HIMM20}}]
\label{lem:recursive-stoch}
% The definition of $A_1$ should be enlarged to deal with 1-EG (more complex initialization)
Consider a filtration $(\filter_{\run})_t$ and four $(\filter_{\run})_t$-adapted
processes $(D_{\run})_t$, $(\zeta_{\run})_t$, $(\chi_{\run})_t$, $(\xi_{\run})_t$
such that $\chi_{\run}$ is non-negative and the following recursive inequality is satisfied for all $\run\ge\start$
\[
    D_{\run+1}
    \le D_{\run} - \zeta_\run + \xi_t.
\]
For $C>0$, we define the events $A_{\run}$ by
$A_\start \defeq \{D_\start\le C/2\}$ and
$A_\run \defeq \{D_\run\le C\}\intersect\thinspace\{\chi_t\le C/4\}$
for $\run\ge\afterstart$.
We consider also the decreasing sequence of events $I_{\run}$ defined by
$I_\run\defeq \bigcap_{\start\le\runalt\le\run} A_\runalt$.
If the following three assumptions hold true
\begin{enumerate}[label=(\roman*), topsep=3pt, itemsep=1pt]
    \item $\forall\run, \zeta_\run\one_{I_\run}\ge0$,
    \item $\forall\run, \exof{\xi_t\given{\filter_\run}}\one_{I_\run}=0$,
    \item $\sum_{\run=\start}^\infty \ex[(\xi_{\run+1}^2+\chi_{\run+1})\one_{I_\run}]\le\delta\epsilon\prob(A_\start)$,
\end{enumerate}
\vspace{-4pt}
where $\epsilon=\min(C^2/16, C/4)$ and $\delta\in(0,1)$, then
$\prob\left(\bigcap_{\run\ge\start}A_\run~\vert~A_\start\right) \ge 1-\delta.$
\end{lemma}

\begin{lemma}\label{lemma: stochastic stability}
    Let \cref{asm:Lipschitz,asm:strong} hold. Consider the iterates of \acl{OMD} with:
    \begin{itemize}
        \item[i)] an unbiased stochastic oracle $\curr[\signal]
	= \vecfield(\curr) + \curr[\noise]
$ with finite variance $\sdev^2$ (see \eqref{eq:signal} and \eqref{eq:sigbounds});
        \item[ii)] step-sizes $0 < \curr[\step] \leq \frac{1}{4\lips}$ with $\sum_{\run=1}^\infty \curr[\step]^2 = \Gamma < + \infty $.
    \end{itemize}
Denote by $\legnhd$ the neighborhood of $\sol$ on which \cref{eq:Legendre} holds with $\base \gets \sol$.
  Then, for any $r>0$ such that $\pnhd \coloneqq \ball({\sol},{r}) \cap \points$  is included in both $\legnhd$ and $\basin$,
    \begin{equation*}
    \prob\left[ {\left. \forall t \geq \half,\ \curr \in \pnhd \right| \breg(\sol, \init[\state]) + \phi_1 \leq \frac{ 2 r^2}{9}} \right] \geq 1 - \frac{9(8+r^2)\sigma^2 \Gamma }{r^2\min(1, r^2/9)}\,,
    \end{equation*}
    where $\phi_1 = \half\|\init[\state] - \initlead[\state]\|^2$.
\end{lemma}

\begin{proof}
    This proof mainly consists in applying %\citet[Lem.~F.1]{hsieh2020}
    \cref{lem:recursive-stoch} indexed by $\run \in \half\N^*,\  \run \geq 1$. 
    %instead of $\runalt \in \N^*$.
    Define $C = \frac{r^2}{2+1/4}$ and, for $\run \geq 1$, the following adapted processes,
    \begin{align*}
        \zeta_{\run} &\coloneqq 0 &\zeta_{\run+1/2} &\coloneqq \curr[\step] \inner{\vecfield(\lead), \lead - \sol}\\ 
        \chi_{\run+1/2} &\coloneqq 4\curr[\step]^2\|\curr[\noise]\|_*^2 &\chi_{\run+1} &\coloneqq 4\curr[\step]^2\|\lead[\noise]\|^2_*\\ 
        \xi_{\run+1/2} &\coloneqq 0 &\xi_{\run+1} &\coloneqq -\curr[\step] \inner{\lead[\noise], \lead-\sol}\\
        D_{\run} &\coloneqq \breg(\sol, \curr) +\phi_{\run} & D_{\run+1/2} &\coloneqq D_{\run} - \zeta_{\run} + \chi_{\run+1/2} + \xi_{\run+1/2}\,.
    \end{align*}
    
    With these definitions, for any $\run \in \N^*,\ \run \geq 1$,
    \begin{align*}
        D_{\run+1/2} &\leq D_{\run} - \zeta_{\run} + \chi_{\run+1/2} + \xi_{\run+1/2}\\
        D_{\run+1}   &\leq D_{\run+1/2} - \zeta_{\run+1/2} + \chi_{\run+1} + \xi_{\run+1}\,
    \end{align*}
    where the first inequality comes directly from the definition while the second one comes from the descent inequality of \cref{lemma:descent}.
    
    Now define the events, for $\run \in \half\N^*,\ \run \geq 1$,
\begin{equation*}
    I_{\run} = \{D_1 \leq C/2\}\cap \bigcap_{\runalt \in \half\N: 3/2 \leq \runalt \leq \run} \{D_{\runalt} \leq C\}\cap \{\chi_{\runalt}^2 \leq C/4\}\,.
\end{equation*}

We first show, by induction on $\run \in \half\N^*,\ \run \geq 1$, that, $I_{\run} \subset \{X_{\runalt} \in \pnhd, \runalt=\half,1,\dots,\run-\half,\run\}$.
\begin{description}
    \item[Initialization:] For $\run=1$, the fact that $I_{1} \subset \{\initlead[\state] \in \pnhd\}$ comes from the conditioning that $\breg(\sol, \init[\state]) + \phi_1 \leq C/2$ and by definition of $\phi_1$ and the strong convexity of $h$ (\cref{lemma: strg cvx h}), 
            \begin{equation*}
        \|X_{1/2} - \sol\|^2 \leq  2(\|X_1 - \sol\|^2 + \|X_1 - X_{1/2}\|^2) \leq 4(\breg(\sol, X_1)+\phi_1) \leq 2C \leq r^2.
            \end{equation*}
        To show that $\init[\state] \in \pnhd$, use the strong convexity of $h$ (\cref{lemma: strg cvx h}) to get that, on $I_1$,
        \begin{equation*}
            \|\init[\state] - \sol\|^2 \leq 2\breg(\sol, \init[\state]) \leq 2D_1 \leq C \leq r^2\,.
        \end{equation*}

    \item[Induction step for iterates:] Assume that, for some $\run \geq 2$, $I_{\run-1/2} \subset \{\state_{\runalt} \in \pnhd, \runalt=\half,1,\dots,\run-1,\run-\half \}$.
        By definition, the sequence of events $(I_{\runalt})$ is non-increasing so that, $I_{\run} \subset \{\state_{\runalt} \in \pnhd, \runalt=\half,1,\dots,\run-1,\run-\half \}$.
        Hence, we only have to show that $I_{\run} \subset \{\curr \in \pnhd\}$. But, by definition, $I_{\run} \subset \{D_{\run} \leq C\}$. Again, using the strong convexity of $h$ (\cref{lemma: strg cvx h}), on $I_{\run}$,
        \begin{equation*}
            \|\curr - \sol\|^2 \leq 2\breg(\sol,\curr) \leq 2D_{\run} \leq 2C \leq {r^2}\,.
        \end{equation*}
        
    \item[Induction step for half-iterates:] Assume that, for some $\run \geq 1$, $I_{\run} \subset \{\state_{\runalt} \in \pnhd, \runalt=\half,1,\dots,\run-1,\run \}$.
        By definition, the sequence of events $(I_{\runalt})$ is non-increasing so that, $I_{\run+1/2} \subset \{\state_{\runalt} \in \pnhd, \runalt=\half,1,\dots,\run-1,\run \}$. So we focus on showing that $I_{\run+1/2} \subset \{\lead \in \pnhd\}$.
            For this, apply the first statement of \cref{lemma: bregman one step} with $(\point, \base, \dpoint, \dpointalt) \gets (\curr, \sol, -\curr[\step] (\vecfield(\beforelead) + \curr[\noise]), -\curr[\step] \vecfield(\sol))$, 
            \begin{equation*}
                \breg(\sol, \lead) \leq \breg(\sol,\curr) -\curr[\step]\inner{\vecfield(\beforelead) + \curr[\noise] - \vecfield(\sol), \lead - \sol}\,,
            \end{equation*}
        and apply Young's inequality twice to get
            \begin{align*}
                \breg(\sol, \lead) &\leq \breg(\sol,\curr) + \curr[\step]^2\|\vecfield(\beforelead) + \curr[\noise] - \vecfield(\sol)\|^2_* + \frac{1}{4}\|\lead - \sol\|^2\\ 
                               &\leq \breg(\sol,\curr) + 2\curr[\step]^2\|\vecfield(\beforelead)  - \vecfield(\sol)\|^2_* + 2\curr[\step]^2\|\curr[\noise]\|^2 + \frac{1}{4}\|\lead - \sol\|^2\,.
            \end{align*}
            
            By the strong convexity of $h$ (\cref{lemma: strg cvx h}) and the Lipschitz continuity of $\vecfield$ (\cref{asm:Lipschitz}),
            \begin{align*}
                \frac{1}{4}\|\lead - \sol\|^2 &\leq \breg(\sol,\curr) + 2\curr[\step]^2\lips^2 \|\beforelead - \sol\|^2 + 2\curr[\step]^2\|\curr[\noise]\|^2 \,. 
            \end{align*}
            
            But, by definition, on $I_{\run+1/2}$, $\breg(\sol,\curr) \leq C$, $\beforelead$ is in $\pnhd$ and $\chi_{\run+1/2}=4\curr[\step]^2\|\curr[\noise]\|_*^2 \leq C/4$.
            Therefore,
            \begin{align*}
                \frac{1}{4}\|\lead - \sol\|^2 &\leq \left(1+\frac{1}{8}\right)C + 2\curr[\step]^2\lips^2r^2\,. 
            \end{align*}
            Using the definition of $C$ and the bound on the step-size $\curr[\step]\leq 1/(4\lips)$ , 
            \begin{align*}
                \|\lead - \sol\|^2 &\leq \frac{r^2}{2} + \frac{r^2}{2} = r^2\,,
            \end{align*}
            which concludes the induction step.
    \end{description}
 We now verify the assumptions needed to apply \cref{lem:recursive-stoch}:
 \begin{enumerate}[(i)]
     \item For $\run \in \half\N^*,\ \run \geq 1,\ \zeta_{\run} \mathds{1}_{I_{\run}} \geq 0$. If $\run \in \N^*$, this is trivial as $\zeta_{\run} = 0$. Now, fix $\run \in \N^*$, $\zeta_{\run+1/2} = \curr[\step] \inner{\vecfield(\lead), \lead - \sol}$. But, on $I_{\run+1/2}$, $\lead \in \pnhd$ and so, by monotonicity of $\vecfield$ (\cref{asm:strong}), 
         %\begin{equation*}
         %    \inner{\vecfield(\lead) - \vecfield(\sol), \lead - \sol} \geq 0\,.
         %\end{equation*}
         %Using the definition of $\sol$ gives that
         $\inner{\vecfield(\lead), \lead - \sol} \geq 0$ on $I_{\run+1/2}$.
     \item For $\run \in \half\N^*,\ \run \geq 1,\ \Condexp{\xi_{\run+1/2}}\mathds{1}_{I_{\run}}=0$. There is only something to prove when $\run$ is of the form $\run+\half$ with $\run \in \N^*$. In this case, as $\lead$ is $\filter_{\run+1/2}$ measurable,
         \begin{align*}
             \Condexp[\run+1/2]{\xi_{\run+1}} &= -\curr[\step] \inner{\Condexp[\run+1/2]{\lead[\noise]}, \lead - \sol} = 0\,.
         \end{align*}
     \item $\sum_{t \in \half\N,\ t \geq 1} \ex\left(\left(\xi_{\run+1/2}^2+\chi_{\run+1/2}\right)\mathds{1}_{I_{\run}}\right) \leq \delta \epsilon \prob(I_1)$ with $\epsilon = \min(C^2/16, C/4)=\min(C/4, 1)C/4$ and $\delta = \frac{\sigma^2 (8+r^2) \Gamma}{\epsilon}$.
        For this let us first bound each of the terms involved individually. For $\run \in \N^*$, by assumption on the noise and as $I_{\run}$ is $\filter_{\run}$ measurable, and so $\filter_{\run+1/2}$ measurable,
        \begin{align*}
            \ex(\chi_{\run+1/2}\mathds{1}_{I_{\run}}) &= 4\curr[\step]^2\ex\left(\|\lead[\noise]\|_*^2 \mathds{1}_{I_{\run}}\right)\\
                                               &= 4\curr[\step]^2\ex\left(\Condexp[\run+1/2]{\|\lead[\noise]\|_*^2}\mathds{1}_{I_{\run}}\right)\\
                                               &\leq 4\curr[\step]^2\sigma^2\prob(I_{\run})\\
                                               &\leq 4\curr[\step]^2\sigma^2\prob(I_1)\,,
        \end{align*}
        where we used that the sequence of events $(I_{\runalt})_{\runalt}$ is non-increasing.
        Likewise, $\ex(\chi_{\run+1}\mathds{1}_{I_{\run+1/2}}) \leq 4\curr[\step]^2\sigma^2\prob(I_1)$.
        Now, by definition of the dual norm,
        \begin{align*}
            \ex(\xi_{\run+1}^2 \mathds{1}_{I_{\run+1/2}}) &= \curr[\step]^2\ex\left(\inner{\lead[\noise], \lead - \sol}^2\mathds{1}_{I_{\run+1/2}}\right)\\ 
                                                &\leq \curr[\step]^2\ex\left(\|\lead[\noise]\|_*^2 \|\lead - \sol\|^2\mathds{1}_{I_{\run+1/2}}\right)\\ 
                                                &\leq \curr[\step]^2 r^2\ex\left(\|\lead[\noise]\|_*^2\mathds{1}_{I_{\run+1/2}}\right)\,,
        \end{align*}
        where we used that $I_{\run+1/2} \subset \{\lead \in \pnhd\}$.
        Next, by the law of total expectation and since $I_{\run+1/2}$ is $\filter_{\run+1/2}$ measurable,
        \begin{align*}
            \ex(\xi_{\run+1}^2 \mathds{1}_{I_{\run+1/2}}) &\leq \curr[\step]^2 r^2\ex\left(\Condexp[\run+1/2]{\|\lead[\noise]\|_*^2}\mathds{1}_{I_{\run+1/2}}\right)\\
                                                &\leq \curr[\step]^2 r^2 \sigma^2\prob(I_{\run+1/2})\\
                                                &\leq \curr[\step]^2 r^2 \sigma^2\prob(I_1)\,.
        \end{align*}
        Combining these bounds, we get,
        \begin{align*}
            &\sum_{t \in \frac{1}{2}\N,\ t \geq 1} \ex\left(\left(\xi_{\run+1/2}^2+\chi_{\run+1/2}\right)\mathds{1}_{I_{\run}}\right) \\
            &=
            \sum_{t \in \N,\ t \geq 1} \ex\left(\xi_{\run+1}^2\mathds{1}_{I_{\run+1/2}}\right)
            + \sum_{t \in \N,\ t \geq 1} \ex\left(\chi_{\run+1/2}^2\mathds{1}_{I_{\run}}\right)
            + \sum_{t \in \N,\ t \geq 1} \ex\left(\chi_{\run+1}^2\mathds{1}_{I_{\run+1/2}}\right)
        \\
            &\leq \sigma^2(8+r^2)\prob(I_1)\sum_{t=1}^{+\infty} \curr[\step]^2\\
                                                                                                                   &\leq \underbrace{\frac{\sigma^2(8+r^2)\Gamma}{\epsilon}}_{=\delta}\epsilon \prob(I_1)\,,
        \end{align*}
 which corresponds to the statement.
 \end{enumerate}

 Hence, \citet[Lem.~F.1]{HIMM20} gives that, \begin{equation*}
     \prob \left[ \left.\bigcap_{\run \in \half\N,\ \run \geq 1}I_{\run}\right|I_1 \right] \geq 1 - \delta\,,
 \end{equation*}
 which implies our statement since $\bigcap_{\run \in \half\N,\ \run \geq 1}I_{\run} \subset \{\forall \run \geq \half,\ \curr \in \pnhd\}$.
\end{proof}

%% file: AppOMDCv.tex
\begin{proof}[Proof of \cref{prop: last iterate past mirror prox stochastic}] \leavevmode

\begin{enumerate}

\item    Start from \cref{lemma:descent} which gives us the descent inequality 
  \begin{align}
  \label{eq: proof past mirror prox stoch descent ineq2}
    \breg(\sol, \next) + \phi_{\run+1}\leq &\breg(\sol, \curr) + (1 - \curr[\step] \strong)\phi_\run - \curr[\step] \inner{\lead[\signal], \lead - \sol}\\ 
 \nonumber   &+ \left(4\curr[\step]^2\lips^2 - \half\right)\|\lead - \curr\|^2 + 4\curr[\step]^2(\|\lead[\noise]\|_*^2 + \|\beforelead[\noise]\|_*^2)\,
\end{align}
where
$
     \phi_{\run} = \frac{\prev[\step]^2}{2}\|\beforelead[\signal] - \beforebeforelead[\signal] \|^2$ for all $\run \geq 2
$, 
and $\phi_1 = \half\|\init[\state] - \initlead[\state]\|^2 = 0$ by assumption.
\item 
 The first part of the result comes the stochastic stability lemma \cref{lemma: stochastic stability} (which relies on \cite{HIMM20}), where we use that the stepsize choice is square-summable with
    \begin{equation}\label{eq: proof stoch pmp bound sum step-size}
        \sum_{t=1}^{+\infty} \curr[\step]^2 = \step  \leq\frac{\step^2}{(2\expstep-1)t_0^{2\expstep - 1}} .
    \end{equation}

\item 
We now focus on proving the rates of convergence.
Take $r> 0$ small enough so that $\pnhd \coloneqq \ball({\sol},{r}) \cap \points$  is included in both $\legnhd$ and $\basin$.

    Define the event $\curr[\event] = \{\forall 1 \leq \runalt \leq \run,\ \state_{\runalt+1/2} \in \pnhd\}$ %\cap \left\{\breg(\sol, \init[\state])+\phi_1 \leq \frac{r^2}{12}\right\}$
    for $\run \geq 0$. Note that, except for $\run=0$, $\curr[\event]$ is $\filter_{\run+1/2}$ is measurable. 
    On this event, we can apply \cref{lemma: strg mon extra} with $\point \gets \curr$ to get,
    \begin{equation*}
        \inner{\vecfield(\lead), \lead - \sol} \geq \frac{\strong}{2}\|\curr - \sol\|^2 - \strong \|\lead - \curr\|^2\,.
    \end{equation*}
  and using the Legendre exponent of $\hreg$, we get
    \begin{equation*}
        \inner{\vecfield(\lead), \lead - \sol} \geq \frac{\strong}{\legconst}\breg(\sol, \curr)^{1+\leg} - \strong \|\lead - \curr\|^2\,.
    \end{equation*}
\item    Combining this with the descent inequality above \cref{eq: proof past mirror prox stoch descent ineq2}, and the fact that $\legconst \geq 1$,
\begin{align*}
   & (\breg(\sol, \next) + \phi_{\run+1})\mathds{1}_{\curr[\event]} \\
   & \leq   (\breg(\sol, \curr) -\frac{\strong\curr[\step]}{\legconst}\breg(\sol, \curr)^{1+\leg} + (1 - \frac{\curr[\step] \strong}{\legconst})\phi_{\run})\mathds{1}_{\curr[\event]} - \curr[\step] \inner{\lead[\noise], \lead - \sol}\mathds{1}_{\curr[\event]} \\
& ~~~~~    +  \left(4\curr[\step]^2\lips^2 + \strong \curr[\step]- \half\right)\|\lead - \curr\|^2\mathds{1}_{\curr[\event]} + 4\curr[\step]^2(\|\lead[\noise]\|_*^2 + \|\beforelead[\noise]\|_*^2)\mathds{1}_{\curr[\event]}\,.
\end{align*}

Now, the sequence of events $(\event_\runalt)_{\runalt \geq 0}$ is non-increasing, so $\mathds{1}_{\curr[\event]} \leq \mathds{1}_{\prev[\event]} \leq 1$. As a consequence,
\begin{align*}
     & \left(\breg(\sol, \curr) -\frac{\strong\curr[\step]}{\legconst}\breg(\sol, \curr)^{1+\leg} + (1 - \frac{\curr[\step] \strong}{\legconst})\phi_{\run}\right)\mathds{1}_{\curr[\event]} \\
    & \leq \left(\breg(\sol, \curr) -\frac{\strong\curr[\step]}{\legconst}\breg(\sol, \curr)^{1+\leg} + (1 - \curr[\step] \strong)\phi_{\run}\right)\mathds{1}_{\prev[\event]}\,.
\end{align*}
Note that the term between parenthesis is always non-negative even when $\leg > 0$. %(\cref{rmk: alpha implies bound u}).

Moreover, by the choice of $\step$, $4\curr[\step]^2\lips^2 + \strong \curr[\step]- \half \leq 0$.
Therefore, the descent inequality can be simplified to give,
\begin{align*}
    & \left(\breg(\sol, \next) + \phi_{\run+1}\right)\mathds{1}_{\curr[\event]} \\
    & \leq \left(\breg(\sol, \curr)-\frac{\strong\curr[\step]}{\legconst}\breg(\sol, \curr)^{1+\leg} + (1 - \frac{\curr[\step] \strong}{\legconst})\phi_{\run}\right)\mathds{1}_{\prev[\event]} - \curr[\step] \inner{\lead[\noise], \lead - \sol}\mathds{1}_{\curr[\event]}\\ 
    & ~~~~ + 4\curr[\step]^2(\|\lead[\noise]\|_*^2 + \|\beforelead[\noise]\|_*^2)\,.
\end{align*}

We now take the expectation of this inequality. For this, note that, as $\curr[\event]$ is $\filter_{\run+1/2}$ measurable,
\begin{align*}
    \ex\left[\inner{\lead[\noise], \lead - \sol}\mathds{1}_{\curr[\event]}\right] &= \ex\left[\Condexp[\run+1/2]{\inner{\lead[\noise], \lead - \sol}\mathds{1}_{\curr[\event]}}\right]\\
    &= \ex\left[\inner{\Condexp[\run+1/2]{\lead[\noise]}, \lead - \sol}\mathds{1}_{\curr[\event]}\right]\\
    &= 0\,.
\end{align*}
As a consequence, and using the assumption on the noise,
\begin{equation}\label{eq: proof past mirror prox stoch descent ineq3}
\begin{aligned}
&\ex\left[(\breg(\sol, \next) + \phi_{\run+1})\mathds{1}_{\curr[\event]}\right] \\
 &\leq \ex\left[\left(\breg(\sol, \curr)-\frac{\strong\curr[\step]}{\legconst}\breg(\sol, \curr)^{1+\leg} + \left(1 - \frac{\curr[\step] \strong}{\legconst}\right)\phi_\run\right)\mathds{1}_{\prev[\event]}\right]  + 8\curr[\step]^2\sdev^2\,.
\end{aligned}
\end{equation}

\item The final step will be to study the sequence $(a_\run)$ defined by $a_{\run} =  \ex\left[(\breg(\sol,\curr) + \phi_{t})\mathds{1}_{\prev[\event]}\right]$ for all $\run\geq 1$.
Indeed, a bound on $a_\run =  \ex\left((\breg(\sol,\curr) + \phi_{\run})\mathds{1}_{\prev[\event]}\right)$ implies a bound on the quantity of interest, $$\ex\left(\breg(\sol,\curr) + \phi_{\run}\mid \event_\nhd\right)\,.$$
%For convenience, denote the event $\event_\nhd \cap \left\{\breg(\sol, \init[\state]) \leq \frac{r^2}{12}\right\}$ by $\event$.
As $\event_\pnhd \subset \curr[\event]$,
\begin{align*}
    \ex\left[\breg(\sol,\curr) + \phi_{\run}\mid \event_\pnhd\right] &= \frac{\ex\left[(\breg(\sol,\curr) + \phi_{\run})\mathds{1}_\event\right]}{\prob[\event_\pnhd]}\\
     &\leq \frac{\ex\left[(\breg(\sol,\curr) + \phi_{t})\mathds{1}_{\curr[\event]}\right]}{\prob[\event_\pnhd]}\,
\end{align*}

Now, the inequality \cref{eq: proof past mirror prox stoch descent ineq3} above can be rewritten as follows,
\begin{align}
\label{eq:ratebase}
a_{\run+1} \leq a_{\run} -  \frac{\strong\curr[\step]}{\legconst  } \ex\left[ \left( \breg(\sol, \curr)^{1+\leg} +  \phi_\run \right) \mathds{1}_{\prev[\event]} \right]  + 8\curr[\step]^2\sdev^2\,.
\end{align}

\item Now, the behavior of this sequence depends heavily on the Legendre exponent $\leg$:
\begin{itemize}
    \item If $\leg=0$, \eqref{eq:ratebase} becomes 
    \begin{align*}
a_{\run+1} \leq \left(1 - \frac{\strong\step}{\legconst (\run+t_0)^\expstep } \right) a_{\run}  + \frac{8\step^2\sdev^2}{(\run+t_0)^{2\expstep}}\, .
\end{align*}

Then, for $\expstep = 1$, \cref{lemma: chung 1} can be applied with the additional assumption that $\step > \frac{\legconst}{\strong}$ to get that for any $\nRuns\geq1$, $a_{\nRuns} = \bigoh{1/\nRuns}$. 

For  $\half < \expstep < 1$, \cref{lemma: chung 4} can be directly apply to obtain $a_{\nRuns} = \bigoh{1/\nRuns^\expstep}$. 
 \item If $\leg>0$, a little more work has to be done before on \eqref{eq:ratebase} before concluding. First,  \eqref{eq:ratebase} implies that for any $\run\geq 1$, $a_\run \leq a_0 + 8 \sdev^2\sum_{t=1}^\infty \curr[\step]^2  $ and since 
 we assume $\breg(\sol, \init[\state])+\phi_1 = \breg(\sol, \init) \leq \frac{r^2}{12}$
 %$\beforeinit[\event] = \left\{\breg(\sol, \init[\state])+\phi_1 \leq \frac{r^2}{12}\right\}$
 %and $\breg(\sol, \curr) \geq 0$ by convexity of $h$,
 we have that 
     \begin{align*}
        \ex\left[\phi_{t}\mathds{1}_{\prev[\event]}\right] &\leq   \ex\left[(\breg(\sol,\curr) + \phi_{t})\mathds{1}_{\prev[\event]}\right] = a_{\run} \\
       & \leq \ex\left[(\breg(\sol, X_{1}) + \phi_{1})\mathds{1}_{\beforeinit[\event]}\right] + 8\sdev^2 \sum_{t=1}^\infty \curr[\step]^2 \\
       &\leq \underbrace{ \frac{r^2}{12} + \frac{8\sdev^2\step^2}{(2\expstep-1)t_0^{2\expstep-1}}  }_{ := \Phi} .
    \end{align*}
    
     As a consequence, $ - \ex\left[\phi_{t}\mathds{1}_{\prev[\event]}\right] \leq - \ex\left[\phi_{t}\mathds{1}_{\prev[\event]}\right]^{1+\leg}/\Phi^\leg$, and thus  \eqref{eq:ratebase} gives us 
     \begin{align*}
a_{\run+1} &\leq a_{\run} -  \frac{\strong\curr[\step]}{\legconst } \ex\left[ \left( \breg(\sol, \curr)^{1+\leg} +  \frac{1}{\Phi^\leg}\phi_\run^{1+\leg} \right) \mathds{1}_{\prev[\event]} \right]  + 8\curr[\step]^2\sdev^2 \\
& \leq a_{\run} -  \frac{\strong\curr[\step]}{ 2^\leg \max(1,\Phi^\leg) \legconst } \ex\left[ \left( \breg(\sol, \curr) +  \phi_\run \right)^{1+\leg} \mathds{1}_{\prev[\event]} \right]  + 8\curr[\step]^2\sdev^2 \\
& \leq a_{\run} -  \frac{\strong\curr[\step]}{ 2^\leg \max(1,\Phi^\leg) \legconst } a_{\run}^{1+\leg} + 8\curr[\step]^2\sdev^2
\end{align*}
where the second inequality comes from the fact that $(x+y)^{1+\leg}/2^\leg \leq x^{1+\leg} +  y^{1+\leg}$ for positive $x,y$ and the last one from Jensen inequality applied to the convex function $\point \mapsto x^{1+\leg}$.

We will now apply one of \cref{lemma: stoch sequence alpha,lemma: stoch sequence alpha bis} to the sequence $(a_\run)$. To make this step clear, let us introduce the same notations as these lemmas. Define
\begin{align*}
    q  \coloneqq \frac{\strong \step}{2^\leg \max(1, \Phi^\leg)\legconst} \quad \text{and} \quad
    q' \coloneqq 8\sdev^2\step^2\,.
\end{align*}

With these notations, the descent inequality can be rewritten as,
\begin{equation*}
    a_{\run+1} \leq a_{\run} - \frac{q}{(\run+t_0)^{\expstep}} a_{\run}^{1+\leg} + \frac{q'}{(\run+t_0)^{2\expstep}}\,.
\end{equation*}
Both \cref{lemma: stoch sequence alpha,lemma: stoch sequence alpha bis} require that,
\begin{equation*}
    q'q^{1/\leg} \leq c(\expstep, \leg) \iff 
    \step^{2 + \frac{1}{\leg}} \leq \frac{c(1, \leg)}{4\sdev^{2}(\strong/\legconst)^{{1/\leg}}}\max(1, \Phi)\,.
\end{equation*}
Finally, we distinguish two cases.
\begin{itemize}
    \item If $\expstep \geq \frac{1+\leg}{1+2\leg}$, we apply \cref{lemma: stoch sequence alpha} to get that,
    for any $\nRuns \geq \start$,
\begin{equation*}
    a_{\nRuns} \leq \frac{a_\start+b}{\left(1 + \leg (a_\start + b)^{\leg}2^{-\leg} \sum_{\run=\start}^{\nRuns-1} \frac{q}{(\run+t_0)^\expstep}\right)^{1/\leg}} \,,
\end{equation*}
where $b = \left(\frac{1 - 2^{1-2\expstep}}{(1+\leg)q}\right)^{\frac{1}{\leg}}$.
Using asymptotic notations, this simply means that,
\begin{equation*}
    a_{\nRuns} = \bigoh*{\left( \sum_{\run=\start}^{\nRuns-1} \frac{1}{(\run+t_0)^\expstep}\right)^{-1/\leg}} \,.
\end{equation*}
The final bound on $a_\nRuns$ now comes from the fact that,
\begin{equation*}
    \sum_{\run=\start}^{\nRuns-1} \frac{1}{(\run+t_0)^\expstep} =
    \begin{cases}
        \Theta\left(T^{1-\expstep}\right) &\text{ if } \expstep < 1\\
        \Theta\left(\log T\right) &\text{ if } \expstep = 1\,.
    \end{cases}
\end{equation*}
    \item If $\expstep \leq \frac{1+\leg}{1+2\leg}$, we apply \cref{lemma: stoch sequence alpha bis} to get that,
    for any $\nRuns \geq \start$,
\begin{equation*}
    a_{\nRuns} \leq \frac{a_\start}{\left(1 + \leg a_\start^{\leg} \sum_{\run=\start}^{\nRuns-1} \frac{q}{(\run+t_0)^\expstep}\right)^{1/\leg}} + \frac{1}{((1+\leg)q)^{\frac{1}{\leg}}{(\run+t_0)}^\expstep}\,.
\end{equation*}
In other words, the sequence $(a_\run)$ satisfies, as $\expstep < 1$,
\begin{equation*}
    a_{\nRuns} =\bigoh*{\inv{T^{(1-\expstep)/\leg}}}+\bigoh*{\inv{T^\expstep}}\,.
\end{equation*}
\end{itemize}

\end{itemize}

\end{enumerate}
\end{proof}

%% file: AppNum.tex
We describe here the setup used in the illustration of \cref{thm:rates} and provide additional plots.

We consider a simple $1$-dimensional example with $\points=[0,+\infty)$ and  $\vecfield(\point) = \point$; the solution of the associated variational inequality is thus $\sol=0$. It is direct to see that $\vecfield$ is $\lips=1$ Lipschitz continuous and $\strong=1$ strongly monotone. We consider the three Bregman regularizers of our running examples %, this time 
on $\points=[0,+\infty)$:
\begin{itemize} 
\item \emph{Euclidean projection}
(\cref{ex:Eucl}):  $\hreg(\point) = \point^{2}/2$ for which $\breg(\sol,\point) = \point^{2}/2$,
and $\legexp=0$;
\item \emph{Negative entropy}
(\cref{ex:entropy}):  $\hreg(\point) = \point \log\point$ for which $\breg(\sol,\point) = \point$ and $\legexp=1/2$ at $\sol$;
\item \emph{Tsallis entropy}
(\cref{ex:Tsallis}):  $\hreg(\point) = -\frac{1}{q(1-q)} \point^q $ for which $\breg(\sol,\point) = \point^q/q$ and $\legexp=\max(0,1-q/2)$ at $\sol$. To show different behaviors we consider $q=0.5$ ($\legexp=0.75$) and  $q=1.5$ ($\legexp=0.25$).
\end{itemize}

For each of these regularizers, we run \ac{OMD} initialized with $\init=\initlead=0.1$, $(\curr[\noise])$ an i.i.d. Gaussian noise process with mean $0$ and variance $\sdev^2=10^{-4}$. We choose our stepsize sequence as $ 
\curr[\step] = \frac{1}{\run^\eta} $ with $\eta$ as prescribed by \cref{thm:rates}. We display the results averaged over $100$ runs. 

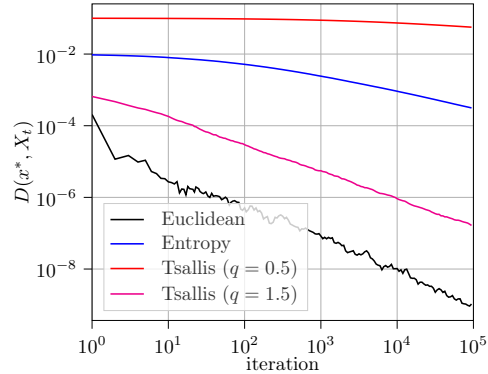
\begin{figure}[t]
         \centering
         \resizebox{0.5\textwidth}{!}{
         \input{Figs/Dx_vs_ite_opt_step}
      }
    \caption{Convergence of \ac{OMD} with different \ac{DGF}s in terms of Bregman divergence to the solution. }
    \label{fig:appcomp}
\end{figure}

We find out that the observed rates match the theory:\\[0.3cm]
\begin{center}
\begin{tabular}{c|c|c|c|c}
 Regularizer     &  $\legexp $ & $\curr[\step]$  & Theoretical rate  & Observed rate (by regression)  \\[0.3cm]
 \hline
\vphantom{\huge A} Euclidean & $0$ & $\frac{1}{\run}$ &  $\frac{1}{\run}$  &  $\frac{1}{\run^{0.99}}$ \\[0.2cm]
 \hline
\vphantom{\huge A}
  Entropy & $0.5$ & $\frac{1}{\run^{0.5+\epsilon}}$ &  $\frac{1}{\run^{0.5+\epsilon}}$  &  $\frac{1}{\run^{0.48}}$ \\[0.2cm]
   \hline
\vphantom{\huge A}
 Tsallis ($q=0.5$) & $0.75$ & $\frac{1}{\run^{0.5+\epsilon}}$ &  $\frac{1}{\run^{0.1666}}$  &  $\frac{1}{\run^{0.13}}$ \\[0.2cm]
  \hline
\vphantom{\huge A}
  Tsallis ($q=1.5$) & $0.25$ & $\frac{1}{\run^{0.75}}$ &  $\frac{1}{\run^{0.75}}$  &  $\frac{1}{\run^{0.71}}$ 
\end{tabular}
\end{center}

%% file: Figs/Dx_vs_ite_opt_step.tex
% This file was created by tikzplotlib v0.9.6.
\begin{tikzpicture}

\begin{axis}[
legend cell align={left},
legend style={fill opacity=0.8, draw opacity=1, text opacity=1, at={(0.03,0.03)}, anchor=south west, draw=white!80!black},
log basis x={10},
log basis y={10},
tick align=outside,
tick pos=left,
x grid style={white!69.0196078431373!black},
xlabel={iteration},
xmajorgrids,
xmin=1, xmax=1e5,
xmode=log,
xtick style={color=black},
xtick={0.01,0.1,1,10,100,1000,10000,100000,1000000,10000000},
xticklabels={\(\displaystyle {10^{-2}}\),\(\displaystyle {10^{-1}}\),\(\displaystyle {10^{0}}\),\(\displaystyle {10^{1}}\),\(\displaystyle {10^{2}}\),\(\displaystyle {10^{3}}\),\(\displaystyle {10^{4}}\),\(\displaystyle {10^{5}}\),\(\displaystyle {10^{6}}\),\(\displaystyle {10^{7}}\)},
y grid style={white!69.0196078431373!black},
ylabel={$\breg(\sol,\curr)$},
ymajorgrids,
ymin=3.67400209573879e-10, ymax=0.251522929747067,
ymode=log,
ytick style={color=black},
ytick={1e-12,1e-10,1e-08,1e-06,0.0001,0.01,1,100},
yticklabels={\(\displaystyle {10^{-12}}\),\(\displaystyle {10^{-10}}\),\(\displaystyle {10^{-8}}\),\(\displaystyle {10^{-6}}\),\(\displaystyle {10^{-4}}\),\(\displaystyle {10^{-2}}\),\(\displaystyle {10^{0}}\),\(\displaystyle {10^{2}}\)}
]
\addplot [thick, black]
table {%
1 0.00020517782540397
1 0.00020517782540397
1 0.00020517782540397
1 0.00020517782540397
1 0.00020517782540397
1 0.00020517782540397
1 0.00020517782540397
1 0.00020517782540397
1 0.00020517782540397
1 0.00020517782540397
1 0.00020517782540397
1 0.00020517782540397
1 0.00020517782540397
2 1.16476911948046e-05
2 1.16476911948046e-05
2 1.16476911948046e-05
2 1.16476911948046e-05
2 1.16476911948046e-05
2 1.16476911948046e-05
2 1.16476911948046e-05
3 1.47420305469619e-05
3 1.47420305469619e-05
3 1.47420305469619e-05
3 1.47420305469619e-05
3 1.47420305469619e-05
4 9.68470872985382e-06
4 9.68470872985382e-06
4 9.68470872985382e-06
5 1.08689434744888e-05
5 1.08689434744888e-05
5 1.08689434744888e-05
5 1.08689434744888e-05
6 5.26702126072282e-06
6 5.26702126072282e-06
7 4.60692398914696e-06
7 4.60692398914696e-06
7 4.60692398914696e-06
8 3.39601918157409e-06
8 3.39601918157409e-06
9 3.32193437275379e-06
10 2.715483271411e-06
10 2.715483271411e-06
11 2.59034597584769e-06
11 2.59034597584769e-06
12 2.36988139156789e-06
13 2.74471528558225e-06
14 1.33584941910073e-06
14 1.33584941910073e-06
15 1.71326396147524e-06
16 1.67574145141546e-06
17 1.17721634310674e-06
18 2.01021239277127e-06
19 1.75206617139647e-06
21 1.67497996681898e-06
22 1.41443466864428e-06
23 1.95156098470017e-06
25 1.74482843261281e-06
26 1.45887571866591e-06
28 1.30080753372566e-06
29 1.27612781733171e-06
31 1.41683940507351e-06
33 1.0893450656623e-06
35 9.96152961275269e-07
37 1.29568589980367e-06
39 1.07817919651104e-06
42 1.11548495750594e-06
44 1.09589827118034e-06
47 1.04836792076472e-06
50 8.60793939917158e-07
53 9.25172058110516e-07
56 1.01120025680302e-06
59 9.15216653712082e-07
63 6.12112467334883e-07
66 7.22203101720392e-07
70 7.85464225433607e-07
74 8.09711290453556e-07
79 7.36645127474783e-07
84 7.24243311423193e-07
89 4.69938944422142e-07
94 4.34989625237954e-07
100 5.10798684686563e-07
105 3.9521137217996e-07
112 4.57397695580861e-07
118 5.7511310106017e-07
125 5.69408178484312e-07
133 4.44084960737135e-07
141 5.07192402458865e-07
149 3.97918214687048e-07
158 4.27822331523984e-07
167 3.65388019339784e-07
177 2.49385510775003e-07
188 2.74837450976875e-07
199 2.91284546425204e-07
211 1.90891524079284e-07
223 1.85087871838619e-07
237 2.02323270646074e-07
251 2.77839383705997e-07
266 2.6605282956907e-07
281 3.05986689512239e-07
298 3.24608692105354e-07
316 2.97529818158352e-07
334 2.53495495916853e-07
354 2.77626493815108e-07
375 2.31460560951089e-07
398 1.79600613337957e-07
421 1.69474797878671e-07
446 1.57841391899063e-07
473 1.49492977812156e-07
501 1.45896189144406e-07
530 1.16408371189282e-07
562 1.32734930845024e-07
595 1.30017008178994e-07
630 1.38937053724656e-07
668 1.26414746537199e-07
707 1.17418222715669e-07
749 1.11382064581406e-07
794 1.05333906381174e-07
841 9.58861105983351e-08
891 9.42669341010453e-08
944 9.72651913685489e-08
1000 8.32328426990043e-08
1059 7.02369219313414e-08
1122 7.67568510066209e-08
1188 6.97821938654819e-08
1258 7.63939455562591e-08
1333 7.1437173946711e-08
1412 7.45494894601596e-08
1496 5.45893371909208e-08
1584 5.22734138804383e-08
1678 5.54118831944743e-08
1778 5.13659472567546e-08
1883 4.49534332920016e-08
1995 4.59439795011221e-08
2113 3.21598187944524e-08
2238 2.86633062947302e-08
2371 2.99544416523844e-08
2511 3.10823539706376e-08
2660 3.24066741758687e-08
2818 2.5755101822389e-08
2985 2.44093463586633e-08
3162 2.65802202179436e-08
3349 2.49299100075026e-08
3548 2.65906409713674e-08
3758 3.02263514414579e-08
3981 3.21605208515098e-08
4216 2.79221081984587e-08
4466 2.059781771309e-08
4731 1.9177688858822e-08
5011 1.67465015068668e-08
5308 1.37473795722157e-08
5623 1.28221582197006e-08
5956 1.24003498322232e-08
6309 9.5247651276412e-09
6683 1.16438434164429e-08
7079 1.25112237464666e-08
7498 1.12785824067433e-08
7943 1.18552806522748e-08
8413 1.34776993952985e-08
8912 1.01820477765184e-08
9440 1.02302945995788e-08
10000 1.02886783843233e-08
10592 8.65719552210825e-09
11220 9.49618640765671e-09
11885 7.57206604755001e-09
12589 9.56132273088184e-09
13335 9.9071958319296e-09
14125 7.51479378559823e-09
14962 5.6838344427028e-09
15848 5.50631467187606e-09
16788 5.5287992385672e-09
17782 5.04872305392199e-09
18836 3.79771871658773e-09
19952 4.29059169215509e-09
21134 4.29388451520113e-09
22387 3.84975176074803e-09
23713 3.53241949361351e-09
25118 3.60698758220669e-09
26607 3.29391004297421e-09
28183 3.30395789063233e-09
29853 2.92159821954312e-09
31622 2.24823720513598e-09
33496 2.61810080547098e-09
35481 2.83597623634055e-09
37583 2.62491480843612e-09
39810 2.22276733329399e-09
42169 2.39827709275659e-09
44668 1.8162654996916e-09
47315 1.63739081238641e-09
50118 1.67877731797766e-09
53088 1.81818211531335e-09
56234 1.47220632359366e-09
59566 1.3607425468959e-09
63095 1.46882825307839e-09
66834 1.47427387238846e-09
70794 1.10839969799597e-09
74989 1.07120016472919e-09
79432 1.0517058148916e-09
84139 9.70537949098964e-10
89125 9.26296933273147e-10
94406 1.05540638551593e-09
};
\addlegendentry{Euclidean }
\addplot [thick, blue]
table {%
1 0.00950570894921595
1 0.00950570894921595
1 0.00950570894921595
1 0.00950570894921595
1 0.00950570894921595
1 0.00950570894921595
1 0.00950570894921595
1 0.00950570894921595
1 0.00950570894921595
1 0.00950570894921595
1 0.00950570894921595
1 0.00950570894921595
1 0.00950570894921595
2 0.00920794385891968
2 0.00920794385891968
2 0.00920794385891968
2 0.00920794385891968
2 0.00920794385891968
2 0.00920794385891968
2 0.00920794385891968
3 0.00898068643831706
3 0.00898068643831706
3 0.00898068643831706
3 0.00898068643831706
3 0.00898068643831706
4 0.00879402191865593
4 0.00879402191865593
4 0.00879402191865593
5 0.00863598217111509
5 0.00863598217111509
5 0.00863598217111509
5 0.00863598217111509
6 0.00847947441513299
6 0.00847947441513299
7 0.00835228306496161
7 0.00835228306496161
7 0.00835228306496161
8 0.00822198218931796
8 0.00822198218931796
9 0.00811342512253442
10 0.00798055792459361
10 0.00798055792459361
11 0.00787157351500516
11 0.00787157351500516
12 0.00777199290607343
13 0.00770354366322529
14 0.0076098723669531
14 0.0076098723669531
15 0.0075365931334464
16 0.00745923957520433
17 0.00739638117003896
18 0.00733336825513124
19 0.00727212308080514
21 0.00715441492715578
22 0.00710277794316408
23 0.00705564781905762
25 0.00696919753062361
26 0.00692548022960032
28 0.00683207139151771
29 0.00678415596858594
31 0.00669283646804176
33 0.00662129992376465
35 0.00654817952513549
37 0.00648099386584472
39 0.00642338866487931
42 0.00631878593616149
44 0.00626692160752202
47 0.00618313618738067
50 0.00608714326476411
53 0.00600264937643152
56 0.00593828730106314
59 0.00585866875329629
63 0.00576795443467917
66 0.00570195110390849
70 0.00562812969501874
74 0.00555644356595864
79 0.0054703147714117
84 0.00538812570319888
89 0.00530943349618016
94 0.00524147382353558
100 0.0051608364245946
105 0.00509360143820081
112 0.00501371987405638
118 0.00494555146767077
125 0.00486372726986961
133 0.00478188172372618
141 0.00469779184382641
149 0.00463078605094853
158 0.00455923817002199
167 0.0044810760265506
177 0.00440944120968347
188 0.00433461709846897
199 0.00426328955891225
211 0.00417847481501907
223 0.00411173737293103
237 0.00403616927236293
251 0.00396350173332937
266 0.00388826002415872
281 0.00381418438114127
298 0.00374683139718224
316 0.00367762212850285
334 0.00360606866053194
354 0.00353533197142086
375 0.00346965606107602
398 0.00340064739677558
421 0.00333602921092391
446 0.00326450883864529
473 0.00320150408447356
501 0.00314025317009466
530 0.00307251471613459
562 0.00300701701120185
595 0.00294546293473045
630 0.00288638026565127
668 0.00282158092751688
707 0.00276970255973021
749 0.00270260226701883
794 0.00263745563713425
841 0.00258434494765639
891 0.00252126901240408
944 0.00246802149659357
1000 0.002418008898724
1059 0.00236206189886141
1122 0.00230732130440792
1188 0.00226179612086922
1258 0.0022121432817688
1333 0.00216250886520397
1412 0.00211781828022365
1496 0.00206572060802213
1584 0.00201509404152436
1678 0.00196830531843983
1778 0.00192259891380075
1883 0.001877291192391
1995 0.00183156470273501
2113 0.00178996616545154
2238 0.00174429692605426
2371 0.00169918417123473
2511 0.00166093830899751
2660 0.0016216490761926
2818 0.00158390757233489
2985 0.00154898111508583
3162 0.00150951440484547
3349 0.0014738394282036
3548 0.00143921951265133
3758 0.00140418908972636
3981 0.00137038842751184
4216 0.00133645237756394
4466 0.00130192017586269
4731 0.00127277208237741
5011 0.0012392601592433
5308 0.0012114636047098
5623 0.00118292639793391
5956 0.00115335478512788
6309 0.00112639369235945
6683 0.00109796994802181
7079 0.00106989579646297
7498 0.00104315965887963
7943 0.00101713392533806
8413 0.000991546312509799
8912 0.000965506531900067
9440 0.00094117285432355
10000 0.000916968624239179
10592 0.000890589034923821
11220 0.000867270729284572
11885 0.000846529544570925
12589 0.000823471534523461
13335 0.000802778660546649
14125 0.000782537707919958
14962 0.000761743146750043
15848 0.00074217712630994
16788 0.00072265965787972
17782 0.000702985570379271
18836 0.000684631226328638
19952 0.000666066633262449
21134 0.000647933986011173
22387 0.000629716344106715
23713 0.000612913973040192
25118 0.000596274356211054
26607 0.000580600741802163
28183 0.000563894920807797
29853 0.000549649917116974
31622 0.000534376451148526
33496 0.000519352468113811
35481 0.00050453900544231
37583 0.000490822034205171
39810 0.000476848807156725
42169 0.00046463655878422
44668 0.000451657201715738
47315 0.000438890397716367
50118 0.000426350782635442
53088 0.000414449773284191
56234 0.000402472416252525
59566 0.000391848065788363
63095 0.000381332849698329
66834 0.000370732679630559
70794 0.000360262421686597
74989 0.000350699713699843
79432 0.000341416743424487
84139 0.000332350664556358
89125 0.000323136435451763
94406 0.000313708520910912
};
\addlegendentry{Entropy }
\addplot [thick, red]
table {%
1 0.0997623696919428
1 0.0997623696919428
1 0.0997623696919428
1 0.0997623696919428
1 0.0997623696919428
1 0.0997623696919428
1 0.0997623696919428
1 0.0997623696919428
1 0.0997623696919428
1 0.0997623696919428
1 0.0997623696919428
1 0.0997623696919428
1 0.0997623696919428
2 0.099567387746589
2 0.099567387746589
2 0.099567387746589
2 0.099567387746589
2 0.099567387746589
2 0.099567387746589
2 0.099567387746589
3 0.0994202107780379
3 0.0994202107780379
3 0.0994202107780379
3 0.0994202107780379
3 0.0994202107780379
4 0.0992560213697497
4 0.0992560213697497
4 0.0992560213697497
5 0.0991440748759881
5 0.0991440748759881
5 0.0991440748759881
5 0.0991440748759881
6 0.0990385924230813
6 0.0990385924230813
7 0.0989415609036582
7 0.0989415609036582
7 0.0989415609036582
8 0.0988574039616227
8 0.0988574039616227
9 0.0987812891326196
10 0.0987080769509466
10 0.0987080769509466
11 0.0986295367011024
11 0.0986295367011024
12 0.0985434103534103
13 0.0984656416774883
14 0.0984101874251401
14 0.0984101874251401
15 0.0983469831331145
16 0.0982834677752319
17 0.0982271769137199
18 0.0981711664542045
19 0.0981262102550882
21 0.0980264576397399
22 0.0979727222027952
23 0.097926501916991
25 0.0978182494205423
26 0.0977702431686474
28 0.097685516304133
29 0.097642313128602
31 0.0975581684941274
33 0.0974804192773674
35 0.0974026440643526
37 0.097316163854939
39 0.0972488175050165
42 0.0971533512718558
44 0.0970813885620298
47 0.0969927134327526
50 0.0969020117100659
53 0.096814206092756
56 0.0967156259585836
59 0.0966237600155063
63 0.0965032940168868
66 0.0964229824454673
70 0.0963340777393666
74 0.096244409035155
79 0.0961209852232307
84 0.0960078412618354
89 0.0959014203999053
94 0.0957924561791899
100 0.0956658495454588
105 0.0955614056405148
112 0.0954174245365357
118 0.095306898301763
125 0.0951728501970348
133 0.095036748356418
141 0.0948990619801128
149 0.0947652569566421
158 0.0946170277004517
167 0.0944665016052921
177 0.0943090059733287
188 0.0941602088823233
199 0.0939974840370314
211 0.0938419783122564
223 0.0936847547724426
237 0.0935035708534761
251 0.0933366430964846
266 0.0931612272662569
281 0.0929901036576257
298 0.092799050786668
316 0.0926234052765797
334 0.0924375343795098
354 0.0922397607250892
375 0.0920424611399145
398 0.091838776387689
421 0.0916361046000115
446 0.0914222812837561
473 0.091201253359048
501 0.0909805800394894
530 0.0907656379051769
562 0.0905299928982694
595 0.0902939434491521
630 0.0900591172346599
668 0.0898100750955317
707 0.0895735310912894
749 0.0893222673291826
794 0.0890683628861004
841 0.0888068159126439
891 0.0885363954514568
944 0.0882660029626012
1000 0.0879912279086714
1059 0.0877165959586314
1122 0.0874403286510798
1188 0.087155092939363
1258 0.0868691924966544
1333 0.086578030545626
1412 0.0862794929135735
1496 0.0859713998839367
1584 0.0856651114569986
1678 0.0853485187678424
1778 0.0850269284379781
1883 0.0847086080401566
1995 0.0843780651424166
2113 0.0840434500072934
2238 0.0837048075726436
2371 0.0833680304735417
2511 0.0830238023503429
2660 0.082677163651472
2818 0.0823264720226796
2985 0.0819730487905005
3162 0.0816167893181702
3349 0.0812557912473036
3548 0.0808864718921835
3758 0.0805129108145613
3981 0.0801413695479245
4216 0.0797724513321122
4466 0.0793938165251415
4731 0.0790086814574646
5011 0.0786214659533614
5308 0.0782294688092893
5623 0.077833870793775
5956 0.0774332771365097
6309 0.0770298279227184
6683 0.0766251075312499
7079 0.0762209835160633
7498 0.0758106539446505
7943 0.0753983907105073
8413 0.0749774872922199
8912 0.0745624240259871
9440 0.0741401049869041
10000 0.0737148899349569
10592 0.0732872332254724
11220 0.0728577512015165
11885 0.0724306653295045
12589 0.0719979953600393
13335 0.0715602294005618
14125 0.0711275587823692
14962 0.0706853519695122
15848 0.0702362603536599
16788 0.0697927601399209
17782 0.0693477302138888
18836 0.0689021544490452
19952 0.0684486500392253
21134 0.0680004825414061
22387 0.0675482940474891
23713 0.0671004862929598
25118 0.0666526585114734
26607 0.0662017469307573
28183 0.0657491526540051
29853 0.0652971341184052
31622 0.0648477127571993
33496 0.0643903738888444
35481 0.0639355343136585
37583 0.06347908758118
39810 0.0630252985371366
42169 0.0625723223385585
44668 0.0621189027883416
47315 0.0616666770446992
50118 0.0612145944514395
53088 0.0607625079945551
56234 0.0603137883422503
59566 0.0598645558382948
63095 0.0594122803329647
66834 0.0589663383058356
70794 0.0585151914471898
74989 0.0580656211028266
79432 0.0576158422163151
84139 0.0571690077427592
89125 0.0567209579835191
94406 0.0562754595579049
};
\addlegendentry{Tsallis ($q=0.5$)}
\addplot [thick, magenta]
table {%
1 0.000654221720795338
1 0.000654221720795338
1 0.000654221720795338
1 0.000654221720795338
1 0.000654221720795338
1 0.000654221720795338
1 0.000654221720795338
1 0.000654221720795338
1 0.000654221720795338
1 0.000654221720795338
1 0.000654221720795338
1 0.000654221720795338
1 0.000654221720795338
2 0.000469118618868055
2 0.000469118618868055
2 0.000469118618868055
2 0.000469118618868055
2 0.000469118618868055
2 0.000469118618868055
2 0.000469118618868055
3 0.000368605336277478
3 0.000368605336277478
3 0.000368605336277478
3 0.000368605336277478
3 0.000368605336277478
4 0.00031313056736489
4 0.00031313056736489
4 0.00031313056736489
5 0.000285316260257571
5 0.000285316260257571
5 0.000285316260257571
5 0.000285316260257571
6 0.000253508575707366
6 0.000253508575707366
7 0.000231012309793029
7 0.000231012309793029
7 0.000231012309793029
8 0.00021541996649165
8 0.00021541996649165
9 0.000196149460824557
10 0.00018237245206412
10 0.00018237245206412
11 0.000165053994436652
11 0.000165053994436652
12 0.000151431060335649
13 0.000141212740071928
14 0.00013490997112422
14 0.00013490997112422
15 0.000130811734161197
16 0.000123592548191412
17 0.000116850176339993
18 0.000113026035778269
19 0.000109565604895658
21 0.000103160952736827
22 9.52277281169847e-05
23 9.34734155964767e-05
25 8.39299450464584e-05
26 8.00199524894418e-05
28 7.44936225985597e-05
29 7.15582323056949e-05
31 6.74699987750085e-05
33 6.58025167576332e-05
35 6.25775755411242e-05
37 5.89561400377134e-05
39 5.60344281128897e-05
42 5.55301391620136e-05
44 5.23562210954876e-05
47 4.84280871723032e-05
50 4.69267827518652e-05
53 4.4980120033621e-05
56 4.43968407569347e-05
59 4.20593142799418e-05
63 4.01535258009046e-05
66 3.84348371230954e-05
70 3.79423301094928e-05
74 3.59155844271498e-05
79 3.48113042415902e-05
84 3.3637651227479e-05
89 3.20248964105223e-05
94 3.1631504401113e-05
100 2.95346165679874e-05
105 2.81656755457485e-05
112 2.69391809692011e-05
118 2.54761435450486e-05
125 2.4245003121967e-05
133 2.26228299471149e-05
141 2.19424355359802e-05
149 2.08070064554323e-05
158 1.97469841468869e-05
167 1.90517452765367e-05
177 1.78528751120109e-05
188 1.75146154313367e-05
199 1.65485562792444e-05
211 1.60445887125727e-05
223 1.53675963921276e-05
237 1.46022071030793e-05
251 1.40090556538966e-05
266 1.35477429921757e-05
281 1.28123694376134e-05
298 1.23763108023982e-05
316 1.21558787055744e-05
334 1.19859995325455e-05
354 1.13272008797025e-05
375 1.09771088786706e-05
398 1.06316282433142e-05
421 1.0035722980008e-05
446 9.8136920886511e-06
473 9.5115965545022e-06
501 8.98852953844828e-06
530 8.33068694858997e-06
562 7.97837530134955e-06
595 7.81332456006021e-06
630 7.54366425878008e-06
668 7.14547502898175e-06
707 6.85039989499564e-06
749 6.66211582455982e-06
794 6.36168028954734e-06
841 5.84946659976572e-06
891 5.65524291074757e-06
944 5.5333801813959e-06
1000 5.55076025130079e-06
1059 5.25796303077539e-06
1122 5.1533447706034e-06
1188 4.98855768290756e-06
1258 4.61559174393114e-06
1333 4.4284288104069e-06
1412 4.27925815465141e-06
1496 3.98708869493698e-06
1584 3.91373808260785e-06
1678 3.84015246520011e-06
1778 3.56284764228363e-06
1883 3.30807090536615e-06
1995 3.27014265316031e-06
2113 3.14126519117018e-06
2238 2.95522045824608e-06
2371 2.83287112323569e-06
2511 2.69816651613158e-06
2660 2.5354069176113e-06
2818 2.38220886376589e-06
2985 2.2041422462517e-06
3162 2.1442980431409e-06
3349 2.01499577070821e-06
3548 1.96886021800961e-06
3758 1.86458185826512e-06
3981 1.85471982932336e-06
4216 1.78886177241916e-06
4466 1.71351618440542e-06
4731 1.64915906937575e-06
5011 1.54583509974695e-06
5308 1.43496484727972e-06
5623 1.35793185875393e-06
5956 1.29289440059732e-06
6309 1.27782504860387e-06
6683 1.2483693979839e-06
7079 1.20912305087518e-06
7498 1.16254695884378e-06
7943 1.13505092321172e-06
8413 1.09679103016565e-06
8912 1.03837578650423e-06
9440 9.92926385674497e-07
10000 9.21231946462841e-07
10592 8.6834468166412e-07
11220 8.35219970237753e-07
11885 8.04422171871154e-07
12589 7.84589472201372e-07
13335 7.33915411797889e-07
14125 6.79299719365186e-07
14962 6.38673868317817e-07
15848 6.07033944533347e-07
16788 5.78635614564904e-07
17782 5.62034491870855e-07
18836 5.43227987605601e-07
19952 5.21560729911489e-07
21134 5.03281625938121e-07
22387 4.80068862420733e-07
23713 4.55791393656875e-07
25118 4.40142514169826e-07
26607 4.23537135879081e-07
28183 3.92351914611439e-07
29853 3.72035098488637e-07
31622 3.50442766457012e-07
33496 3.36744027380697e-07
35481 3.36435480300133e-07
37583 3.05918229429322e-07
39810 3.04587927066833e-07
42169 2.8586179952119e-07
44668 2.86080061719596e-07
47315 2.72522988410945e-07
50118 2.61520385387871e-07
53088 2.52114037192707e-07
56234 2.37448093129081e-07
59566 2.30963467440399e-07
63095 2.30852606261878e-07
66834 2.23755974341265e-07
70794 2.12900828106922e-07
74989 2.04892397408295e-07
79432 1.98914370400934e-07
84139 1.9195940889853e-07
89125 1.78737717846624e-07
94406 1.64701354723176e-07
};
\addlegendentry{Tsallis ($q=1.5$) }
\end{axis}

\end{tikzpicture}

%% file: Thanks.tex
%----------------------------------------------------------------------
%%% THANKS
%----------------------------------------------------------------------
% !TEX root = ./Main.tex
This research was partially supported by
the French National Research Agency (ANR) in the framework of
the ``Investissements d'avenir'' program (ANR-15-IDEX-02),
the LabEx PERSYVAL (ANR-11-LABX-0025-01),
MIAI@Grenoble Alpes (ANR-19-P3IA-0003),
and the grants ORACLESS (ANR-16-CE33-0004) and ALIAS (ANR-19-CE48-0018-01).